\documentclass[10pt]{article}
\usepackage[top=1in, bottom=1in, left=1in, right=1in]{geometry}

\usepackage{amsmath}
\usepackage{amssymb,amsfonts,amsmath,amsthm,amscd,dsfont,mathrsfs,mathtools,microtype,nicefrac,pifont,soul}
\usepackage[linktocpage,colorlinks,linkcolor=blue,anchorcolor=blue,citecolor=blue,urlcolor=blue,pagebackref]{hyperref}

\renewcommand*{\backref}[1]{\ifx#1\relax \else Page #1 \fi}
\renewcommand*{\backrefalt}[4]{%
  \ifcase #1 \footnotesize{(Not cited.)}%
  \or        \footnotesize{(Cited on page~#2.)}%
  \else      \footnotesize{(Cited on pages~#2.)}%
  \fi
}
\usepackage[utf8]{inputenc} 
\usepackage[T1]{fontenc}    
\usepackage{url}            
\usepackage{booktabs}       
\usepackage{amsfonts}       
\usepackage{nicefrac}       
\usepackage{microtype}      
\usepackage{xcolor}         

\usepackage[ruled,vlined,linesnumbered]{algorithm2e}
\usepackage{amsmath, amsthm, amssymb, amsfonts, mathrsfs}
\usepackage{bbm}
\usepackage{bm}
\usepackage{color, colortbl}
\definecolor{LightCyan}{rgb}{0.88,1,1}
\usepackage{dirtytalk}
\usepackage{dsfont}
\usepackage{enumerate}
\usepackage{graphicx}
\usepackage{listings}
\usepackage{mathtools}
\usepackage{subfigure}
\usepackage{enumitem}
\usepackage{url}
\usepackage{xspace}
\usepackage{booktabs}
\usepackage{pifont}
\usepackage{makecell}
\usepackage{multirow, array}
\usepackage{verbatim}

\usepackage[capitalize,noabbrev]{cleveref}
\usepackage{tikz}
\usepackage{tcolorbox}
\newtcolorbox{thmbox}{colback=cyan!5,colframe=white}
\newtcolorbox{questionbox}{colback=red!5!white,colframe=white}
\newtcolorbox{updatebox}{colback=white,colframe=black}
\usepackage{tablefootnote}

\usepackage[symbol]{footmisc}

\newcommand{\sref}[1]{\textsuperscript{\ref{#1}}}

\usepackage{caption}
\usepackage{wrapfig}

\newtheorem{theorem}{Theorem}
\newtheorem{lemma}{Lemma}[section]

\newtheorem{assumption}{Assumption}

\theoremstyle{remark}
\newtheorem*{remark}{Remark}

\newcommand{\cC}{\mathcal{C}}
\newcommand{\cD}{\mathcal{D}}

\newcommand{\cG}{\mathcal{G}}

\newcommand{\cL}{\mathcal{L}}

\newcommand{\cO}{\mathcal{O}}

\newcommand{\cX}{\mathcal{X}}

\def\setF{\mathscr{F}} 

\newcommand{\realset}{\mathbb{R}}

\newcommand{\diag}[1]{\mathrm{diag}\left(#1\right)}

\newcommand{\E}{\mathbb{E}}

\newcommand{\norm}[1]{\left\|#1\right\|}

\newcommand{\T}{^\top}

\newcommand{\bfone}{\mathbf{1}}

\newcommand{\red}[1]{\textcolor{red}{#1}}
\newcommand{\blue}[1]{\textcolor{blue}{#1}}

\DeclareMathOperator*{\argmax}{arg\,max\,}
\DeclareMathOperator*{\argmin}{arg\,min\,}

\DeclareMathOperator{\proj}{\Pi}

\def\<#1,#2>{\left\langle #1,#2\right\rangle}
\mathchardef\mhyphen="2D
\allowdisplaybreaks

\title{Optimal Algorithms for Stochastic Bilevel Optimization under Relaxed Smoothness Conditions\footnote{The first two authors contributed equally to this work.}}

\author{Xuxing Chen\thanks{Department of Mathematics, UC Davis. \texttt{xuxchen@ucdavis.edu}. }
\and 
Tesi Xiao\thanks{Department of Statistics, UC Davis. \texttt{texiao@ucdavis.edu}. }
\and Krishnakumar Balasubramanian\thanks{Department of Statistics, UC Davis. \texttt{kbala@ucdavis.edu}. Supported by NSF grant DMS-2053918.} 
}

\begin{document}

\maketitle
\begin{abstract}

Stochastic Bilevel optimization usually involves minimizing an upper-level (\texttt{UL}) function that is dependent on the arg-min of a strongly-convex lower-level (\texttt{LL}) function. Several algorithms utilize Neumann series to approximate certain matrix inverses involved in estimating the implicit gradient of the \texttt{UL} function (hypergradient). The state-of-the-art StOchastic Bilevel Algorithm (\texttt{SOBA}) \cite{dagreou2022framework} instead uses stochastic gradient descent steps to solve the linear system associated with the explicit matrix inversion. This modification enables \texttt{SOBA} to match the lower bound of sample complexity for the single-level counterpart in non-convex settings. Unfortunately, the current analysis of \texttt{SOBA} relies on the assumption of higher-order smoothness for the \texttt{UL} and \texttt{LL} functions to achieve optimality. In this paper, we introduce a novel fully single-loop and Hessian-inversion-free algorithmic framework for stochastic bilevel optimization and present a tighter analysis under standard smoothness assumptions (first-order Lipschitzness of the \texttt{UL} function and second-order Lipschitzness of the \texttt{LL} function). Furthermore, we show that by a slight modification of our approach, our algorithm can handle a more general multi-objective robust bilevel optimization problem. For this case, we obtain the state-of-the-art oracle complexity results demonstrating the generality of both the proposed algorithmic and analytic frameworks. Numerical experiments demonstrate the performance gain of the proposed algorithms over existing ones.

\end{abstract}

\section{Introduction}
Bilevel optimization is gaining increasing popularity within the machine learning community due to its extensive range of applications, including meta-learning~\cite{bertinetto2018meta, franceschi2018bilevel, rajeswaran2019meta, ji2020convergence}, hyperparameter optimization~\cite{bengio2000gradient,franceschi2018bilevel,bertrand2020implicit}, data augmentation~\cite{cubuk2019autoaugment, rommel2022cadda}, and neural architecture search~\cite{liudarts19, zhangadvancing22}. The objective of bilevel optimization is to minimize a function that is dependent on the solution of another optimization problem:
\begin{equation}\label{eq: bo}
    \begin{aligned}
        \min_{x\in \cX \subseteq \realset^{d_x}} \Phi(x) := f(x, y^*(x)) \qquad \text{s.t. }\ y^*(x) = \argmin_{y\in \realset^{d_y}} g(x, y)
    \end{aligned}
\end{equation}
where the upper-level (\texttt{UL}) function $f$ (a.k.a. \emph{outer function}) and the lower-level (\texttt{LL}) function $g$ (a.k.a. \emph{inner function}) are two real-valued functions defined on $\realset^{d_x} \times \realset^{d_y}$. The set $\cX$ is either $\realset^{d_x}$ or a convex compact set in $\realset^{d_x}$, and the \texttt{LL} function $g$ is strongly convex. We call $x$ the \emph{outer variable} and $y$ the \emph{inner variable}. The objective function 
 $\Phi(x)$ is called the \emph{value function}. In this paper, we consider the stochastic setting in which only the stochastic oracles of $f$ and $g$ are available:
 \begin{equation}\label{eq: stochastic setting}
    \mathrm{(\texttt{UL})}\quad f(x, y) = \E_{\xi\sim \cD_f}\left[ F(x, y; \xi) \right],\quad \mathrm{(\texttt{LL})}\quad g(x, y) = \E_{\phi\sim \cD_g}\left[ G(x, y; \phi) \right].
 \end{equation}
 Stochastic bilevel optimization can be considered as an extension of bilevel empirical risk minimization~\cite{dagreou2023lower}, allowing for the effective handling of online and streaming data $(\xi, \phi)$. 
 
 In many instances, the analytical expression of $y^*(x)$ is unknown and can only be approximated using an optimization algorithm. This adds to the complexity of problem \eqref{eq: bo} compared to its single-level counterpart. Under regular conditions such that $\Phi$ is differentiable, the \emph{hypergradient} $\nabla \Phi(x)$ derived by the chain rule and the implicit function theorem is given by
 \begin{equation}\label{eq: hypergradient nabla F}
     \nabla \Phi(x) = \nabla_1 f(x, y^*(x)) - \nabla_{12}^2 g(x, y^*(x)) z^*(x),
 \end{equation}
 where $z^*(x)\in\realset^{d_y}$ is the solution of a linear system:
 \begin{equation}\label{eq: linear system z}
     z^*(x) = \left[\nabla_{22}^{2}g(x,y^*(x))\right]^{-1}\nabla_2 f(x, y^*(x)).
 \end{equation}
Solving \eqref{eq: bo} using only stochastic oracles poses significant challenges since there is no direct unbiased estimator available for $\left[\nabla_{22}^{2}g(x,y^*(x))\right]^{-1}$ and also for $\nabla \Phi(x)$ as a consequence.

To mitigate the estimation bias, many existing methods~\cite{ghadimi2018approximation, ji2021bilevel, yang2021provably, hong2023two, guo2021novel, khanduri2021near, chen2021closing, akhtar2022projection} employ a Hessian Inverse Approximation (\texttt{HIA}) subroutine, which involves drawing a mini-batch of stochastic Hessian matrices and computing a truncated Neumann series~\cite{stewart1998matrix}. However, this subroutine comes with an increased computational burden and introduces an additional factor of $\log(\epsilon^{-1})$ in the sample complexity. Some alternative methods~\cite{chen2022single, guo2021randomized, li2022fully} calculate the explicit inverse of the stochastic Hessian matrix with momentum updates. To circumvent the need for explicit Hessian inversion and the \texttt{HIA} subroutine, recent works~\cite{arbel2022amortized, dagreou2022framework} propose running Stochastic Gradient Descent (\texttt{SGD}) steps to approximate the solution $z^*(x)$ of the linear system \eqref{eq: linear system z}. In particular, the state-of-the-art Stochastic Bilevel Algorithm (\texttt{SOBA}) \emph{only} utilizes SGD steps to simultaneously update three variables: the inner variable $y$, the outer variable $x$, and the auxiliary variable $z$. Remarkably, \texttt{SOBA} achieves the same complexity lower bound of its single-level counterpart ($\Phi \in \cC_{L}^{1,1}$ \footnote[9]{\label{footnote: Smoothness} $\cC_{L}^{p,p}$ denotes $p$-times differentiability with Lipschitz $k$-th order derivatives for $0<k\leq p$.}) in the non-convex setting~\cite{arjevani2023lower}. 

Despite the superior computational and sample efficiency of \texttt{SOBA}, its current theoretical framework assumes high-order smoothness for the \texttt{UL} function $f$ and the \texttt{LL} function $g$ such that $z^*(x)$ has Lipschitz gradient. Specifically, unlike the typical assumptions in stochastic bilevel optimization that state $f\in \cC_{L}^{1,1}$ and $g\in \cC_{L}^{2,2}$ (\texttt{A1}), the current theory of \texttt{SOBA} requires $f\in \cC_{L}^{2,2}$ and $g\in \cC_{L}^{3,3}$ (\texttt{A2}). The necessity of (\texttt{A2}) is counter-intuitive as the partial gradients of $x, y, z$ utilized in constructing SGD steps are already Lipschitz continuous under (\texttt{A1}). Furthermore, assuming $g$ is strongly convex and the partial gradient of the \texttt{UL} function with respect to the inner variable $y$ is bounded for all pairs of $(x, y^*(x))$, (i.e., $\norm{\nabla_2 f(x, y^*(x))}\leq L_f$ for all $x\in\cX$), there exists a subset relation among three function classes as follows (Lemma 2.2 in \cite{ghadimi2018approximation}):
\begin{equation*}
   \mathrm{(\texttt{A2})}\,\,\, \{f\in \cC_{L}^{2,2}, g\in \cC_{L}^{3,3}\} \subset \ \mathrm{(\texttt{A1})}\ \{f\in \cC_{L}^{1,1}, g\in \cC_{L}^{2,2}\} \subset \{\Phi\in \cC_{L}^{1,1}\}.
\end{equation*}
In light of this, it can be concluded that (\texttt{A1}) is sufficient to ensure the first-order Lipschitzness of $\Phi$, which is the standard assumption in the single-level setting. Therefore, a natural question follows:
\begin{questionbox}
    \centering
    Is it possible to develop a \emph{fully single-loop} and \emph{Hessian-inversion-free} algorithm for solving stochastic bilevel optimization problems that achieves an \emph{optimal} sample complexity of $\cO(\epsilon^{-2})$ under \emph{standard smoothness assumptions} $\{f\in \cC_{L}^{1,1}, g\in \cC_{L}^{2,2}\}$\sref{footnote: Smoothness}?
\end{questionbox}
In this paper, we provide an affirmative answer to the aforementioned question. Our \textbf{contributions} can be summarized as follows:
\begin{itemize}[leftmargin=2em]
    \item We propose a class of fully single-loop and Hession-inversion-free algorithm, named Moving-Average \texttt{SOBA} (\texttt{MA-SOBA}), which builds upon the \texttt{SOBA} algorithm by incorporating an additional sequence of average hypergradients. Unlike \texttt{SOBA}, \texttt{MA-SOBA} achieves an optimal sample complexity of $\mathcal{O}(\epsilon^{-2})$ under standard smoothness assumptions, without relying on high-order smoothness. Moreover, the introduced sequence of average hypergradients converges to $\nabla \Phi(x)$, thus offering a reliable termination criterion in the stochastic setting.
    \item We expand the scope of \texttt{MA-SOBA} to tackle a broader class of problems, specifically the min-max multi-objective bilevel optimization problem with significant applications in robust machine learning. We introduce \texttt{MORMA-SOBA}, an algorithm that can find an $\epsilon$-first-order stationary point of the $\mu_\lambda$-strongly-concave regularized formulation while achieving a sample complexity of $\mathcal{O}(n^5\mu_{\lambda}^{-4}\epsilon^{-2})$, which fills a gap (in terms of the order of $\epsilon$-dependency) in the existing literature.
    \item We conduct experiments on several machine learning problems. Our numerical results show the efficiency and superiority of our algorithms.
\end{itemize}

\begin{table}[t]
\centering
\renewcommand{\arraystretch}{1.4}
\caption{Comparison of the stochastic bilevel optimization solvers in the nonconvex-strongly-convex setting under smoothness assumptions \sref{footnote: Smoothness} on $f$ and $g$. We omit the comparison with variance reduction-based methods (\texttt{VRBO}, \texttt{MRBO}~\cite{yang2021provably}; \texttt{SUSTAIN}~\cite{khanduri2021near}; \texttt{SABA}~\cite{yang2021provably}; \texttt{SRBA}~\cite{dagreou2023lower}; \texttt{SVRB}~\cite{guo2021randomized}; \texttt{FLSA}~\cite{li2022fully}; \texttt{SBFW}~\cite{akhtar2022projection}) that may achieve $\cO(\epsilon^{-1.5})$ sample complexity and  under mean-squared smoothness assumptions on stochastic functions $F_\xi$ and $G_\phi$, and \texttt{SBMA}~\cite{guo2021novel} that achieves $\cO(\epsilon^{-4})$ sample complexity.} 
\label{tab:summary}
\resizebox{\textwidth}{!}{%
\begin{tabular}{ c  c c c c c c}
\hline
\rowcolor{gray!20}\Gape[0pt][2pt]{\makecell{Method \\ \blue{(double-loop)}}} & \makecell{Sample\\ Complexity} & (\texttt{UL}) $f$ \tablefootnote[3]{\label{footnote: partial f bound} All methods also assume $\norm{\nabla_2 f(x, y^*(x))}\leq L_f <\infty$ for all $x\in\cX$.} & (\texttt{LL}) $g$ & \Gape[0pt][2pt]{\makecell{Hession\\ Inversion}} & Inner Loop & Batch Size\\
\hline
\texttt{BSA} \cite{ghadimi2018approximation} & $\tilde{\cO}(\epsilon^{-3})$ & $\cC_{L}^{1,1}$  & SC and $\cC_{L}^{2,2}$  & \blue{Neumann approx.} & SGD on inner & $\tilde{\cO}(1)$ \\
\texttt{stocBiO} \cite{ji2021bilevel} & $\tilde{\cO}(\epsilon^{-2})$ & $\cC_{L}^{1,1}$  & SC and $\cC_{L}^{2,2}$  & \blue{Neumann approx.} & SGD on inner & $\tilde{\cO}(\epsilon^{-1})$\\
\tablefootnote[5]{\label{footnote: ALSET vs TTSA} \texttt{ALSET} can achieve convergence without the need for double loops, but it comes at the cost of a worse dependence on $\kappa$ in sample complexity. The mechanisms of single-loop \texttt{ALSET} and \texttt{TTSA} are essentially the same, except that \texttt{ALSET} employs single time-scale stepsizes while \texttt{TTSA} employs two time-scales.}\texttt{ALSET} \cite{chen2021closing} & $\tilde{\cO}(\epsilon^{-2})$ & $\cC_{L}^{1,1}$  & SC and $\cC_{L}^{2,2}$  & \blue{Neumann approx.} & SGD on inner & $\tilde{\cO}(1)$\\
\texttt{AmIGO} \cite{arbel2022amortized} & $\cO(\epsilon^{-2})$ & $\cC_{L}^{1,1}$ & SC and $\cC_{L}^{2,2}$ & SGD & SGD on inner & $\cO(\epsilon^{-1})$\\
\hline
\rowcolor{gray!20}\Gape[0pt][2pt]{\makecell{Method \\ \blue{(single-loop)}}} & \makecell{Sample\\ Complexity} & (\texttt{UL}) $f$\sref{footnote: partial f bound} & (\texttt{LL}) $g$ & \Gape[0pt][2pt]{\makecell{Hession\\ Inversion}} & Inner Step & Batch Size\\
\hline
\sref{footnote: ALSET vs TTSA}\texttt{TTSA} \cite{hong2023two} &  $\tilde{\cO}(\epsilon^{-{2.5}})$ & $\cC_{L}^{1,1}$  & SC and $\cC_{L}^{2,2}$ & \blue{Neumann approx.} & SGD & $\tilde{\cO}(1)$\\
\texttt{STABLE} \cite{chen2022single} &  $\cO(\epsilon^{-2})$ & $\cC_{L}^{1,1}$  & SC and $\cC_{L}^{2,2}$ & \blue{Direct} & SGD & $\cO(1)$\\
\texttt{SOBA} \cite{dagreou2022framework} & $\cO(\epsilon^{-2})$ & $\cC_{L}^{\red{2,2}}$ & SC and $\cC_{L}^{\red{3,3}}$ & SGD & SGD &  $\cO(1)$\\
\rowcolor{LightCyan} \texttt{MA-SOBA} (Alg. \ref{alg:ma-soba}) & $\cO(\epsilon^{-2})$ & $\cC_{L}^{1,1}$ & SC and $\cC_{L}^{2,2}$ & SGD & SGD &  $\cO(1)$\\
\hline
\end{tabular}
}
\footnotesize{~\\The sample complexity corresponds to the number of calls to stochastic gradients and Hessian(Jocobian)-vector products to get an $\epsilon$-stationary point. The $\tilde{\cO}$ notation hides a factor of $\log(\epsilon^{-1})$. ``SC'' means ``strongly-convex''.}
\end{table}

\textbf{Related Work.} The concept of bilevel optimization was initially introduced in the work of \cite{bracken1973mathematical}. Since then, numerous gradient-based bilevel optimization algorithms have been proposed, broadly categorized into two groups: ITerative Differentiation (ITD) based methods~\cite{domke2012generic, maclaurin2015gradient, franceschi2018bilevel, grazzi2020iteration, ji2021bilevel} and Approximate Implicit Differentiation (AID) based methods~\cite{domke2012generic, pedregosa2016hyperparameter, gould2016differentiating, ghadimi2018approximation, grazzi2020iteration, ji2021bilevel, arbel2022amortized}. The ITD-based algorithms typically involve approximating the solution of the inner problem using an iterative algorithm and then computing an approximate hypergradient through automatic differentiation. However, a major drawback of this approach is the necessity of storing each iterate of the inner optimization algorithm in memory. The AID-based algorithms leverage the implicit gradient given by \eqref{eq: hypergradient nabla F}, which requires the solution of a linear system characterized by \eqref{eq: linear system z}. Extensive research has been conducted on designing and analyzing deterministic bilevel optimization algorithms with strongly-convex \texttt{LL} functions; see \cite{ji2021bilevel} and the references cited therein. 

In recent years, there has been a growing interest in stochastic bilevel optimization, especially in the setting of a non-convex \texttt{UL} function and a strongly-convex \texttt{LL} function. To address estimation bias, one set of methods uses SGD iterations for the inner problem and employs truncated stochastic Neumann series to approximate the inverse of the Hessian matrix in $z^*(x)$~\cite{ghadimi2018approximation, ji2021bilevel, yang2021provably, hong2023two, guo2021novel, khanduri2021near, chen2021closing, akhtar2022projection}. The analysis of such methods was refined by \cite{chen2021closing} to achieve convergence rates similar to those of SGD. However, the Neumann approximation subroutine introduces an additional factor of $\log(\epsilon^{-1})$ in the sample complexity. Some alternative approaches~\cite{arbel2022amortized, chen2022single, guo2021randomized, li2022fully} calculate the explicit inverse of the stochastic Hessian matrix with momentum updates. Nevertheless, these methods encounter challenges related to computational complexity in matrix inversion and numerical stability. 

To avoid the need for explicit Hessian inversion and the Neumann approximation, recent algorithms~\cite{arbel2022amortized, dagreou2022framework} propose running SGD steps to approximate the solution $z^*(x)$ of the linear system \eqref{eq: linear system z}. One such algorithm called \texttt{AmIGO} \cite{arbel2022amortized} employs a double-loop approach and achieves an optimal sample complexity of $\cO(\epsilon^{-2})$ under regular assumptions. However, \texttt{AmIGO} requires a growing batch size inversely proportional to $\epsilon$. On the other hand, the single-loop algorithm \texttt{SOBA} \cite{dagreou2022framework} achieves the same complexity lower bound but with constant batch size. Unfortunately, the current analysis of \texttt{SOBA} relies on the assumption of higher-order smoothness for the \texttt{UL} and \texttt{LL} functions. In this work, we introduce a novel algorithm framework that differs slightly from \texttt{SOBA} but can achieve optimal sample complexity in theory without higher-order smoothness assumptions. A summary of our results and comparison to prior work is provided in Table~\ref{tab:summary}.

In addition, there exist several variance reduction-based methods following the line of research by \cite{yang2021provably, khanduri2021near, yang2021provably, dagreou2023lower, guo2021randomized, li2022fully}. Some of these methods achieve an improved sample complexity of $\mathcal{O}(\epsilon^{-1.5})$ and match the lower bounds of their single-level counterparts when stochastic functions $F_\xi$ and $G_\phi$ satisfy mean-squared smoothness assumptions and the algorithm is allowed simultaneous queries at the same random seed \cite{arjevani2023lower}. However, since we are specifically considering smoothness assumptions on $f$ and $g$, we will not delve into the comparison with these methods.

The most recent advancements in (stochastic) bilevel optimization focus on several new ideas: (i) addressing constrained lower-level problems~\cite{shen2023penalty, xiao2023alternating, tsaknakis2022implicit, giovannelli2021inexact}, (ii) handling lower-level problems that lack strong convexity~\cite{chen2023bilevel, huang2023momentum, liu2023averaged, liu2021towards, sow2022constrained, jiang2023conditional}, (iii) developing fully first-order (Hesssian-free) algorithms~\cite{liu2022bome, kwon2023fully, sow2022convergence}, (iv) establishing convergence to the second-order stationary point~\cite{huang2022efficiently}, and (v) expanding the framework to encompass multi-objective optimization problems~\cite{giovannelli2023bilevel, gu2022min, hu2022multi}. It is promising to apply some of these advancements to our specific framework. However, in this work, we contribute to multi-objective bilevel problems with a slight modification of our approach. Other directions are left as future work.

\textbf{Notation.} We use $\norm{\cdot}$ for $\ell^2$ norm. $\bfone_n$ denotes the all-one vector in $\realset^n$. $\Delta_n = \{\lambda\mid \lambda_i\geq 0, \sum_{i=1}^{n}\lambda_i=1\}$ denotes the probability simplex. $\Pi_{\cX}(\cdot)$ denotes the orthogonal projection onto $\cX$. 

\section{Proposed Framework: the MA-SOBA Algorithm}

Similar to \cite{dagreou2022framework, arbel2022amortized}, our algorithm initiates with inexact hypergradient descent techniques and seeks to offer an alternative in the stochastic setting. To provide a clear illustration, let us initially consider the deterministic setting. The \texttt{SOBA} framework keeps track of three sequences, denoted as $\{x^k, y^k, z^k\}$, and updates them using $D_x, D_y, D_z$ as follows:
\begin{align}
   \textbf{(inner)}\quad \blue{y^{k+1}} & = \blue{y^k - \beta_k~\boxed{\nabla_2 g(x^k, y^k)} = y^k - \beta_k D_y(x^k, y^k, z^k)}\label{eq: y update mean}\\
   \textbf{(aux)}\quad  \blue{z^{k+1}} &=  z^k - \gamma_k\left\{\nabla_{22}^2g(x^k,y^*(x^k))z^k - \nabla_2 f(x^k,y^*(x^k))\right\}\notag\\
    \red{\textbf{bias} \rightarrow}&\approx \blue{z^k - \gamma_k~\boxed{\left\{\nabla_{22}^2g(x^k,y^k)z^k - \nabla_2 f(x^k,y^k)\right\}} = z^k -\gamma_k D_z(x^k, y^k, z^k)}  \label{eq: z update mean}\\
    \textbf{(outer)}\quad \blue{x^{k+1}}  &= x^{k} - \alpha_k\left\{\nabla_1 f(x^k, y^*(x^k)) - \nabla_{12}^2 g(x^k, y^*(x^k)) z^*(x^k)\right\}  = x^{k} - \alpha_k \nabla \Phi(x^k) \notag\\
    \red{\textbf{bias} \rightarrow}&\approx \blue{ x^{k} - \alpha_k~\boxed{\left\{\nabla_1 f(x^k, y^k) - \nabla_{12}^2 g(x^k, y^k) z^k\right\}} = x^k - \alpha_k D_x(x^k, y^k, z^k)} \label{eq: x update mean}
\end{align}
where \eqref{eq: y update mean} is the GD step to minimize $g(x^k, \cdot)$, \eqref{eq: x update mean} is the inexact hyper gradient descent step, and \eqref{eq: z update mean} is the GD step to minimize a quadratic minimization problem with $z^*(x^k)$ being the solution, i.e., 
\begin{equation*}
    z^*(x^k) = \underset{z}{\argmin}~\frac{1}{2}\<\nabla_{22}^2g(x^k,y^*(x^k))z, z> - \<\nabla_2 f(x^k,y^*(x^k)), z>.
\end{equation*}
Given that the above update rule, highlighted in blue, does not involve the Hessian matrix inversion, \texttt{SOBA} can directly utilize the stochastic oracles of $\nabla_1 f, \nabla_2 f, \nabla_2 g, \nabla_{22}^2 g, \nabla_{12}^2 g$ to obtain unbiased estimators of $D_x, D_y, D_z$ in Eq.\eqref{eq: y update mean}, \eqref{eq: z update mean}, \eqref{eq: x update mean}. This approach circumvents the requirement for a Neumann approximation subroutine or a direct matrix inversion. However, due to the update rule for $y$, which only utilizes one-step SGD at each iteration $k$, the value of $y^k$ does not coincide with $y^*(x^k)$. As a result, a certain bias is introduced in the partial gradient of $z$ in Eq.\eqref{eq: z update mean}. Similarly, when estimating the hypergradient $\nabla \Phi(x)$, another bias term arises in Eq.\eqref{eq: x update mean}. Although the bias decreases to zero as $y^k\rightarrow y^*(x^k)$ and $z^k\rightarrow z^*(x^k)$ under standard smoothness assumptions as indicated by Lemma 3.4 in \cite{dagreou2022framework}, the current analysis of \texttt{SOBA} requires more regularity on $f$ and $g$ to carefully handle the bias; it assume that $f$ has Lipschitz Hessian and $g$ has Lipschitz third-order derivative.

The inability to obtain an unbiased gradient estimator is a common characteristic in stochastic optimization involving nested structures; see, for example, stochastic compositional optimization \cite{wang2017stochastic, yang2019multilevel, ghadimi2020single, balasubramanian2022stochastic, chen2021solving} as a specific case of \eqref{eq: bo}. One popular approach is to introduce a sequence of dual variables that approximates the true gradient by aggregating all past biased stochastic gradients using a moving averaging technique \cite{ghadimi2020single, balasubramanian2022stochastic, xiao2022projection}. Motivated by this approach, we introduce another sequence of variables, denoted as $\{h^k\}$, and update it at $k$-th iteration given the past iterates $\setF_k$ as follows:
\begin{equation*}
    h^{k+1} = (1-\theta_k) h^k + \theta_k w^{k+1},\quad \E\left[w^{k+1}\middle|\setF_k\right] =  D_x(x^k, y^k, z^k),\quad \theta_k\in(0,1].
\end{equation*}
Following the update rule in the constrained setting ($\cX\subset\realset^{d_x}$) \cite{ghadimi2020single}, the outer variable is updated as $x^{k+1} = x^k + \alpha_k\left(\Pi_{\cX}(x^k - \tau h^k) - x^k\right)$, which is reduced to the GD step when $\cX\equiv\realset^{d_x}$. Denote the stochastic oracles of $\nabla_1 f(x^k, y^k), \nabla_2 f(x^k, y^k), \nabla_2 g(x^k, y^k), \nabla_{22}^2 g(x^k, y^k), \nabla_{12}^2 g(x^k, y^k)$ at $k$-th iteration as $u_x^{k+1}, u_y^{k+1}, v^{k+1}, H^{k+1}, J^{k+1}$ respectively. We present our method, referred to as \textbf{M}oving-\textbf{A}verage \texttt{SOBA} (\texttt{MA-SOBA}), in Algorithm \ref{alg:ma-soba}.

\begin{algorithm}[t]
    \caption{\texttt{Moving-Average SOBA}}\label{alg:ma-soba}
    \SetAlgoLined
    \KwIn{$x^0, y^0, z^0, h^0=0, \{\alpha_k\}, \{\beta_k\}, \{\gamma_k\}, \{\theta_k\}$}
    \For{$k=0, 1,\dots, K-1$}{
        $x^{k+1} = x^k + \alpha_k\left(\Pi_{\cX}(x^k - \tau h^k) - x^k\right)$ \hfill \# update $x^k$ using the average hypergradient $h^k$\\
        $y^{k+1} = y^k - \beta_k v^{k+1}$ \hfill \# update $y^k$ by one-step SGD based on \eqref{eq: y update mean}\\
        $z^{k+1} = z^k - \gamma_k(H^{k+1}z^{k} - u_{y}^{k+1})$ \hfill \# update $z^k$ by one-step SGD based on \eqref{eq: z update mean}\\
        $h^{k+1} = (1 - \theta_k)h^{k} + \theta_k \left(u_x^{k+1} - J^{k+1}z^k\right)$ \hfill \# update the average hypergradient $h^k$ 
    }
\end{algorithm}



\section{Theoretical Analysis}

In this section, we provide convergence rates of \texttt{MA-SOBA} under \emph{standard} smoothness conditions on $f, g$ and \emph{regular} assumptions on stochastic oracles. We also present a proof sketch and have detailed discussions about assumptions made in the literature. The complete proofs are deferred in Appendix.

\subsection{Preliminaries and Assumptions}

As we consider the general setting in which $\cX$ can be either $\realset^{d_x}$ or a closed compact set in $\realset^{d_x}$, we use the notion of gradient mapping to characterize the first-order stationarity, which is a classical measure widely used in the literature as a convergence criterion when solving nonconvex constrained problems \cite{nesterov2018lectures}. For $\tau>0$, we define the gradient mapping of at point $\bar x \in \cX$ as $\cG_{\cX}(\bar{x}, \nabla \Phi(\bar{x}), \tau) \coloneqq \frac{1}{\tau} (\bar{x} - \proj_\cX (\bar x - \tau \nabla \Phi(\bar{x})))$. When $\cX \equiv \realset^d$, the gradient mapping simplifies to $\nabla \Phi(\bar{x})$. Our main goal in this work is to find an $\epsilon$-stationary solution to~\eqref{eq: bo}, in the sense of  $\E[\| \mathcal{G}_{\cX}(\bar{x}, \nabla \Phi(\bar{x}), \tau)\|^2]\leq \epsilon$.

We first state some regularity assumptions on the functions $f$ and $g$. 
\begin{assumption}\label{aspt: smoothness} The functions $f$ and $g$ satisfy:
\begin{itemize}[leftmargin=2em, itemsep=2pt]
    \item[(a)] \textnormal{\bf ($f\in\cC_L^{1,1}$ and $g\in\cC_L^{2,2}$)\sref{footnote: Smoothness}} $\nabla f, \nabla g, \nabla^2 g$ are $L_{\nabla f}, L_{\nabla g}, L_{\nabla^2 g}$ Lipschitz continuous respectively.
    \item[(b)]  \textnormal{\bf (SC \texttt{LL})} $g$ is $\mu_{g}$-strongly convex.
    \item[(c)] $\norm{\nabla_2 f(x, y^*(x))} \leq L_f < \infty$ for all $x\in\cX$.
\end{itemize}
\end{assumption}
\begin{remark}
    The above assumption serves as a sufficient condition for the Lipschitz continuity of $\nabla \Phi$, $y^*(x)$, and $z^*(x)$, as well as $D_x$, $D_y$, and $D_z$ in Eq. \eqref{eq: y update mean}, \eqref{eq: z update mean}, \eqref{eq: x update mean}. The inclusion of high-order smoothness assumptions ($f\in\cC_L^{2,2}$ and $g\in\cC_L^{3,3}$) in the current analysis of \texttt{SOBA} \cite{dagreou2022framework} is primarily intended to ensure the Lipschitzness of $\nabla z^*(x)$. However, the necessity of such assumptions is subject to doubt, given that $\nabla z^*(x)$ is not involved in designing the algorithm. Furthermore, the Lipschitzness of $f$ or uniformly boundedness of $\nabla_2 f$ made in several previous works is unnecessary. Instead, the boundedness assumption on $\nabla_2 f$ is only required for all pairs of $(x,y^*(x))$ as demonstrated by (c). 
\end{remark}
Next, we discuss assumptions made on the stochastic oracles.
\begin{assumption}\label{aspt: stochastic_grad}
    For any $k\geq 0$, define $\setF_k$ denotes the sigma algebra generated by all iterates with superscripts not greater than $k$, i.e., $\setF_k = \sigma\left\{h^1, \dots, h^k, x^1,\dots, x^k, y^1, \dots, y^k, z^1,\dots,z^k\right\}$.
    The \textbf{stochastic oracles} of $\nabla_1 f(x^k, y^k), \nabla_2 f(x^k, y^k), \nabla_2 g(x^k, y^k), \nabla_{22}^2 g(x^k, y^k), \nabla_{12}^2 g(x^k, y^k)$, denoted as $u_x^{k+1}, u_y^{k+1}, v^{k+1}, H^{k+1}, J^{k+1}$ respectively, used in Algorithm \ref{algo: momasoba} at $k$-th iteration are \textbf{unbiased} with \textbf{bounded variance} given $\setF_k$. They are conditionally \textbf{independent} with respect to $\setF_k$.
\end{assumption}
\begin{remark}
    The unbiasedness and bounded variance assumptions on stochastic oracles are standard and typically satisfied in several practical stochastic optimization problems \cite{lan2020first}. It is important to highlight that we explicitly impose these assumptions on the stochastic oracles, unlike Assumption 3.6 in \cite{dagreou2022framework}, which assumes $\E[\|v^{k+1}\|^2|\setF_k] \leq B_y^2 (1+ \|D_y(x^k, y^k, z^k)\|^2)$ and $\E[\|H^{k+1}z^{k} - u_y^{k+1}\|^2|\setF_k] \leq B_z^2 (1+ \|D_z(x^k, y^k, z^k)\|^2)$. In this case, $B_y$ and $B_z$ represent constants in terms of the Lipschitz constants ($L$) and variance bounds ($\sigma^2$). Moreover, Assumption 3.7 in \cite{dagreou2022framework} assumes $\E[\|w^{k+1}\|^2|\setF_k] \leq B_x^2$ holds for a constant $B_x$, which is considerably stronger than our assumptions and may not hold for a broad class of problems.
\end{remark}

\subsection{Convergence Results}
We have the following theorem characterizing the convergence results of \texttt{MA-SOBA}.
\begin{thmbox}
\begin{theorem}\label{thm: soba_convergence}
    Define $x_+^k = \Pi_{\cX}(x^k - \tau h^k)$. Suppose Assumptions \ref{aspt: smoothness} and \ref{aspt: stochastic_grad} hold. Then there exist positive constants $c_1, c_2, c_3, \tau>0$ such that if $\alpha_k \equiv \Theta(1/\sqrt{K})$, $\beta_k = c_1\alpha_k, \gamma_k = c_2\alpha_k, \theta_k = c_3\alpha_k,$ in Algorithm \ref{alg:ma-soba}, then the iterates in Algorithm \ref{alg:ma-soba} satisfy
    \begin{equation}\label{eq: optimality_measure_thm}
        \frac{1}{K}\sum_{k=1}^{K}\frac{1}{\tau^2}\norm{x_+^k - x^k}^2 = \cO\left(\frac{1}{\sqrt{K}}\right),\qquad \frac{1}{K}\sum_{k=1}^{K}\norm{h^k - \nabla \Phi(x^k)}^2 = \cO\left(\frac{1}{\sqrt{K}}\right),
    \end{equation}
    which imply
    \[
        \frac{1}{K}\sum_{k=1}^{K}\E\left[\norm{\frac{1}{\tau} \left(x^k - \Pi_{\cX}(x^k - \tau \nabla \Phi(x^k))\right)}^2\right] = \cO\left(\frac{1}{\sqrt{K}}\right).
    \]
    That is to say, when uniformly randomly selecting a solution $x^R$ from $\{x^1, \dots, x^K\}$, the sample complexity of Algorithm \ref{alg:ma-soba} for finding an $\epsilon$-stationary point is $\mathcal{O}(\epsilon^{-2})$.
\end{theorem}
\end{thmbox}

\begin{remark}
    In contrast to most existing methods, in \texttt{MA-SOBA}, the introduced sequence of dual variables $\{h^k\}$ converges to the exact hypergradient $\nabla \Phi(x)$, even in the presence of estimation bias. This attribute provides reliable terminating criteria in practice. In addition, similar results with an extra factor of $\log(K)$ in the convergence rate can be established under decreasing $\alpha_k$ \cite{dagreou2022framework}. 
\end{remark}

\subsection{Proof Sketch of Theorem \ref{thm: soba_convergence}}\label{sec: bo_analysis} 
Define $V_k = \frac{1}{\tau^2}\|x_+^k - x^k\|^2 + \|h^k - \nabla \Phi(x^k)\|^2$. To obtain \eqref{eq: optimality_measure_thm}, we consider the merit function $W_k$: 
\begin{equation*}
    W_k = \Phi(x^k) - \eta_{\cX}(x^k, h^k, \tau) + \|y^k - y_*^k\|^2 + \|z^k - z_*^k\|^2,
\end{equation*}
where $\eta_{\cX}(x,h,\tau)= \<h, x_+ - x> + \frac{1}{2\tau}\norm{x_+-x}^2$.
By leveraging the moving average updates of $x^k$ (line 2 of Algorithm \ref{alg:ma-soba}), we can obtain
\[
    \sum_{k=0}^{K}\alpha_k\E\left[V_k\right] =\cO\left(\sum_{k=0}^{K}\left(\alpha_k\E\left[\norm{\E\left[w^{k+1}\middle|\setF_k\right] - \nabla \Phi(x^k)}^2\right] + \alpha_k^2\right)\right),
\]
which reduces the error analysis to controlling the hypergradient estimation bias, i.e., $\|\E[w^{k+1}|\setF_k] - \nabla \Phi(x^k)\|^2.$
This term, by the construction of $w^{k+1}$, satisfies
\begin{equation*}
\resizebox{\hsize}{!}{$
\sum_{k=0}^{K}\alpha_k\E\left[\norm{\E\left[w^{k+1}\middle|\setF_k\right] - \nabla \Phi(x^k)}^2\right]= \cO\left(\sum_{k=0}^{K}\alpha_k\E\left[\norm{x_+^k - x^k}^2 + \norm{y^k-y_*^k}^2 + \norm{z^k - z_*^k}^2\right]\right).$}
\end{equation*}
It is worth noting that \cite{dagreou2022framework} requires the existence and Lipschitzness of $\nabla^2 f$ and $\nabla^3g$ to ensure the Lipschitzness of $\nabla z^*(x)$ (see \eqref{eq: linear system z}) which is used in proving the sufficient decrease of $\norm{z^k - z_*^k}^2$. In contrast, based on the moving average updates of $x^k$ and $h^k$, our refined analysis does not necessitate such high-order smoothness assumptions to obtain that
\[
    \sum_{k=0}^{K}\alpha_k\E\left[\norm{y^k-y_*^k}^2 + \norm{z^k - z_*^k}^2\right] = \cO\left(\sum_{k=0}^{K}\alpha_k\E\left[\norm{x_+^k - x^k}^2\right]\right).
\]
The proof of Theorem \ref{thm: soba_convergence} can then be completed by choosing appropriate $\alpha_k, c_1, c_2, c_3, \tau > 0$.
\section{Min-Max Bilevel Optimization}

To incorporate robustness in the multi-objective setting where each objective can be expressed as a bilevel optimization problem in \eqref{eq: bo}, the following mini-max bilevel problem formulation was proposed in \cite{gu2022min}:
\begin{equation}\label{eq: minmaxbo}
        \min_{x\in \cX}\ \max_{1\leq i \leq n } \ \Phi_i(x) \coloneqq f_i(x, y_i^*(x)) \qquad \text{s.t. }\ y_i^*(x) = \argmin_{y_i\in \realset^{d_{y_i}}} g_i(x, y_i), 1\leq i\leq n.
\end{equation}
Note that~\eqref{eq: minmaxbo} can be reformulated as a general nonconvex-concave min-max optimization problem (with a bilevel substructure):
\begin{equation}\label{eq: minmaxbo_reform}
        \min_{x\in \cX}\ \max_{\lambda\in \Delta_n}\Phi(x,\lambda) \coloneqq \sum_{i=1}^{n}\lambda_i \Phi_i(x).
\end{equation}
Instead of solving \eqref{eq: minmaxbo_reform} directly, in this work, we focus on solving the following regularized version,
\begin{equation}\label{eq: minmaxbo_reform_reg}
        \min_{x\in \cX}\ \max_{\lambda\in \Delta_n}\Phi_{\mu_{\lambda}}(x,\lambda) \coloneqq \Phi(x,\lambda) - \frac{\mu_{\lambda}}{2}\norm{\lambda - \frac{\bfone_n}{n}}^2.
\end{equation}
Note that in \eqref{eq: minmaxbo_reform_reg}, we include an $\ell^2$ regularization term that penalizes the discrepancy between $\lambda$ and $\frac{\bfone_n}{n}$. When $\mu_{\lambda} = 0$, it corresponds to equation \eqref{eq: minmaxbo}, and as $\mu_{\lambda}\rightarrow +\infty$, it enforces $\lambda = \frac{\bfone_n}{n}$, leading to direct minimizing of the average loss. It is important to note that minimizing the worst-case loss (i.e., $\max_{1\leq i\leq n}f_i(x,y_i^*(x))$) does not necessarily imply the minimization of the average loss (i.e., $\frac{1}{n}\sum_{i=1}^{n}f_i(x, y_i^*(x))$). Therefore, in practice, it may be preferable to select an appropriate $\mu_{\lambda}>0$ \cite{qian2019robust, wang2021adversarial} to strike a balance between these two types of losses.
\subsection{Proposed Framework: the MORMA-SOBA Algorithm}
The proposed algorithm, which we refer as to \textbf{M}ulti-\textbf{O}bjective \textbf{R}obust \texttt{MA-SOBA} (\texttt{MORMA-SOBA}), for solving \eqref{eq: minmaxbo_reform_reg} is presented in \ref{algo: momasoba}. In addition to the basic framework of Algorithm \ref{alg:ma-soba}, we also maintain a moving average step in the updates of $\lambda^k$ for solving the max part of problem \ref{thm: momasoba_convergence}. It is worth noting that in its single-level counterpart without the inner variable $y$, the proposed \texttt{MORMA-SOBA} algorithm is fundamentally similar to the single-timescale averaged \texttt{SGDA} algorithm proposed in the general nonconvex-strongly-concave setting~\cite{qiu2020single}. Moreover, our algorithm framework can be leveraged to solve the distributionally robust compositional optimization problem discussed in \cite{gao2021convergence}.
\begin{remark}[Comparison with \texttt{MORBiT} \cite{gu2022min}]
    In contrast to our approach in \eqref{eq: minmaxbo_reform_reg}, the work of~\cite{gu2022min}, for the min-max bilevel problem, attempted to combine \texttt{TTSA}~\cite{hong2023two} and \texttt{SGDA}~\cite{lin2020gradient} to solve the nonconvex-concave problem as \eqref{eq: minmaxbo_reform}. However, \textit{we identified an issue in} \cite{gu2022min} related to the ambiguity and inconsistency in the expectation and filtration, which may not be easily resolved within their current proof framework. As a consequence, their current proof is unable to demonstrate $\E[\max_{i\in [n]}\|y_i^{k} - y_i^*(x^{(k-1)})\|^2 ]\leq \tilde{\cO}(\sqrt{n}K^{-2/5})$ as claimed in Theorem 1 (10b) of \cite{gu2022min}. Thus, the subsequent arguments made regarding the convergence analysis of $x$ and $\lambda$ are incorrect (at least in its current form); see Section \ref{sec: app_discussion} for further discussions. Moreover, the practical implementation of \texttt{MORBiT} incorporates momentum and weight decay techniques to optimize the simplex variable $\lambda$. This approach can be seen as a means of solving the regularized formulation in \eqref{eq: minmaxbo_reform_reg}.
\end{remark}

\begin{algorithm}[t]
    \caption{\texttt{Multi-Objective Robust Moving-Average SOBA}}\label{algo: momasoba}
    \SetAlgoLined
    \KwIn{$x^0,\lambda^0, \{y^0_i\}, \{z^0_i\}, h_x^0=0, h_\lambda^0=0, \{\alpha_k\}, \{\beta_k\}, \{\gamma_k\}, \{\theta_k\}$}
    \For{$k=0, 1,\dots, K-1$}{
        $x^{k+1} = x^k + \alpha_k\left(\Pi_{\cX}(x^k - \tau_x h_x^k) - x^k\right)$\hfill \# update $x^k$ using the average hypergradient $h_x^k$\\
        $\lambda^{k+1} = \lambda^k + \alpha_k \left(\Pi_{\Delta_n}(\lambda^k + \tau_{\lambda}h_{\lambda}^k) - \lambda^k\right)$ \hfill \# update $\lambda^k$ using the average gradient $h_\lambda^k$ \\
        \For{$i=1, \dots, n$ (in parallel)}{
        $y_i^{k+1} = y_i^k - \beta_k v_i^{k+1}$ \hfill \# update each $y_i^k$ by one-step SGD based on \eqref{eq: y update mean}\\
        $z_i^{k+1} = z_i^k - \gamma_k\left(H_i^{k+1}z_i^k - u_{y, i}^{k+1}\right)$ \hfill \# update each $z_i^k$ by one-step SGD based on \eqref{eq: z update mean}\\
        }
        $h_x^{k+1} = (1 - \theta_k)h_x^k + \theta_k \sum_{i=1}^{n}\lambda_i^k\left(u_{x,i}^{k+1} - J_i^{k+1}z_i^k\right)$  \hfill \# update the average hypergradient $h_x^k$\\
        $h_{\lambda}^{k+1} = (1 - \theta_k)h_{\lambda}^k + \theta_k\left(s^{k+1} - \mu_{\lambda}\left(\lambda^k - \frac{\bfone_n}{n}\right)\right)$ \hfill \# update the average gradient $h_\lambda^k$
    }
\end{algorithm}

\subsection{Convergence Results}\label{sec: convergence_momasoba}
We first present additional assumptions required in the analysis of \texttt{MORMA-SOBA}. 
\begin{assumption}\label{aspt: f_bound}
    For any $k\geq 0$, functions $\Phi(x), \nabla \Phi_i(x)$ are bounded, functions $f_i$ are $L_f$-Lipschitze continuous in the second input, and their stochastic versions are unbiased with bounded variance, i.e., there exists $L_{\Phi}, L_f, \sigma_{f,0}\geq 0$ such that 
    \begin{align*}
        &\left|\Phi_i(x)\right|\leq b_{\Phi},\ \norm{\nabla \Phi_i(x)}\leq L_{\Phi},\ \left|f_i(x, y) - f_i(x,\tilde y)\right|\leq L_f\norm{y - \tilde y},\ \text{ for all } x, y, \tilde y, 1\leq i\leq n, \\
        &s^{k+1} = \left(s_1^{k+1}, ..., s_n^{k+1}\right)\T,\ \E\left[s_i^{k+1}\middle|\setF_k\right] = f_i(x^k, y_i^k),\ \E\left[ \norm{s_i^{k+1} -f_i(x^k, y_i^k)}^2\middle|\setF_k\right]\leq \sigma_{f,0}^2.
    \end{align*}
    $\bigcup_{i=1}^{n}\left\{u_{x, i}^{k+1}, u_{y, i}^{k+1}, v_i^{k+1}, H_i^{k+1}, J_i^{k+1}\right\}\cup \left\{s^{k+1}\right\}$ are conditionally independent with respect to $\setF_k$.
\end{assumption}
We have the following convergence theorem of \texttt{MORMA-SOBA}.
\begin{thmbox}
\begin{theorem}\label{thm: momasoba_convergence}
    Suppose Assumptions \ref{aspt: smoothness}, \ref{aspt: stochastic_grad} (for all $f_i, g_i$) and Assumption \ref{aspt: f_bound} hold. Then there exist positive constants $c_1, c_2, c_3, \tau_x, \tau_{\lambda}>0$ such that if $\alpha_k \equiv \Theta(1/\sqrt{nK}), \beta_k = c_1\alpha_k, \gamma_k = c_2\alpha_k, \theta_k = c_3\alpha_k, \mu_{\lambda}<1$ in Algorithm \ref{algo: momasoba}, then the iterates in Algorithm \ref{algo: momasoba} satisfy
    \[
        \frac{1}{K}\sum_{k=0}^{K}\E\left[\norm{\frac{1}{\tau_x}\left(x^k - \Pi_{\cX}\left(x^k - \tau_x\nabla\Psi_{\mu_{\lambda}}(x^k)\right) \right)}^2\right] = \cO\left(\frac{n^2}{\mu_{\lambda}^2\sqrt{K}}\right), 
    \]
    where $\Psi_{\mu_{\lambda}}(x):= \max_{\lambda\in\Delta_n}\Phi_{\mu_{\lambda}}(x, \lambda)$. That is to say, when uniformly randomly selecting a solution $x^R$ from $\{x^1, \dots, x^K\}$, the sample complexity (the total number of calls to stochastic oracles) of finding an $\epsilon$-stationary point by Algorithm \ref{algo: momasoba} is $\cO(n^5\mu_{\lambda}^{-4}\epsilon^{-2})$.
\end{theorem}
\end{thmbox}

Theorem \ref{thm: momasoba_convergence} indicates that Algorithm \ref{algo: momasoba} is capable of generating an $\epsilon$-first-order stationary point of $\min_x \Psi_{\mu_{\lambda}}(x)$ with $K\gtrsim n^5\mu_{\lambda}^{-4}\epsilon^{-2}$. As $\mu_{\lambda}\rightarrow 0$, the problem \eqref{eq: minmaxbo_reform_reg} changes towards the nonconvex-concave problem \eqref{eq: minmaxbo_reform} and the sample complexity becomes worse, which to some extent implies the difficulty of directly solving \eqref{eq: minmaxbo_reform}. Define $x_+^k = \Pi_{\cX}(x^k - \tau_x h_x^k),\ \lambda_+^k = \Pi_{\Delta_n}(\lambda^k + \tau_{\lambda}h_{\lambda}^k)$. Our proof of Theorem \ref{thm: momasoba_convergence} mainly relies on analyzing $\tilde V_k$ as follows:
\[
    \tilde V_k = \frac{1}{\tau_x^2}\|x_+^k - x^k\|^2 + \|h_x^k - \nabla_1\Phi_{\mu_{\lambda}}(x^k,\lambda^k)\|^2 + \frac{1}{\tau_{\lambda}^2}\|\lambda_+^k  - \lambda^k\|^2 + \|h_{\lambda}^k - \nabla_2\Phi_{\mu_{\lambda}}(x^k, \lambda^k)\|^2,
\]
which is commonly used in min-max optimization (see, e.g., \cite{qiu2020single}). The construction of the first two terms follows the same idea in Section~\ref{sec: bo_analysis}, and for the last two terms we have
\[
    \tau_{\lambda}^2\mu_{\lambda}^2\norm{\lambda^k - \lambda_*^k}^2 = \cO\left(\norm{\lambda_+^k - \lambda^k}^2 + \tau_{\lambda}^2\norm{h_{\lambda}^k - \nabla_2\Phi_{\mu_{\lambda}}(x^k, \lambda^k)}^2\right).
\]
Therefore, $\tilde V_k = 0$ implies $x^k = \Pi_{\cX}\left(x^k - \tau_x\nabla_1\Phi_{\mu_{\lambda}}(x^k, \lambda^k)\right)$ as well as $\lambda^k = \lambda_*^k$, which further imply the gradient mapping of problem $\min_x\Psi_{\mu_{\lambda}}(x)$ at $x^k$ is $0$, and hence the validity of our optimality measure. We defer the proof details to Section \ref{sec: thm2proof} in the Appendix. 

\begin{remark}
    Note that in Theorem~\ref{thm: momasoba_convergence} we explicitly characterize the dependency on $n$ and $\mu_{\lambda}$ in the convergence rate and the sample complexity. It is worth noting that two variants of stochastic gradient descent ascent (\texttt{SGDA}) algorithms for solving the nonconvex-strongly-concave min-max optimization problems (without bilevel substructures), 
    have been studied in~\cite{lin2020gradient, qiu2020single}. While such algorithms are not immediately applicable to solve~\eqref{eq: minmaxbo_reform_reg} due to the presence of the additional bilevel substructure, it is instructive to compare to those methods assuming direct access to $y_i^*(x)$ in \eqref{eq: minmaxbo}. Specifically, we  observe that the sample complexity of \texttt{SGDA} with batch size $M=\Theta(n^{1.5}\epsilon^{-1})$ in \cite{lin2020gradient} and moving-average \texttt{SGDA} with $\cO(1)$ batch size in~\cite{qiu2020single} for solving \eqref{eq: minmaxbo_reform_reg} assuming direct access to $y_i^*(x)$ will be $\cO\left(n^4\mu_{\lambda}^{-2}\epsilon^{-2}\right)$ and $\cO\left(n^5\mu_{\lambda}^{-4}\epsilon^{-2}\right)$\footnote{Note that $\Phi_{\mu_{\lambda}}(x, \lambda)$ in \eqref{eq: minmaxbo_reform} is quadratic in $\lambda$, and these two sample complexities are obtained under this special case, i.e., $\nabla_2^2f(x,y) = -\mu \mathbf{I}$ applied to \cite{lin2020gradient, qiu2020single}.} respectively. Our results in Theorem \ref{thm: momasoba_convergence} indicate that the sample complexity of the proposed algorithm \texttt{MORMA-SOBA} for solving min-max bilevel problems has the same dependency on $n$ and $\mu_\lambda$ as the sample complexity of the moving-average \texttt{SGDA} introduced in~\cite{qiu2020single} for solving min-max single-level problems, while also computing $y_i^*(x)$ instead of assuming direct access. 
\end{remark}

\section{Experiments}

While our contributions primarily focus on theoretical aspects, we also conducted experiments to validate our results. We first compare the performance of \texttt{MA-SOBA} with other benchmark methods on two common tasks proposed in previous works~\cite{ji2021bilevel, hong2023two, dagreou2022framework}, \emph{hyperparameter optimization} for $\ell^2$ penalized logistic regression and \emph{data hyper-cleaning} on the corrupted MNIST dataset. Our experiments are performed with the aid of the recently developed package \texttt{Benchopt}~\cite{moreau2022benchopt} and the open-sourced bilevel optimization benchmark\footnote{\label{footnote: bilevel benchopt}\url{https://github.com/benchopt/benchmark_bilevel}}. For a fair comparison, we exclusively consider benchmark methods that do not utilize variance reduction techniques in Table \ref{tab:summary}: (i) \texttt{BSA}~\cite{ghadimi2018approximation}; (ii) \texttt{stocBiO}~\cite{ji2021bilevel}; (iii) \texttt{TTSA}~\cite{hong2023two}/\texttt{ALSET}~\cite{chen2021closing}; (iv) \texttt{SOBA}~\cite{dagreou2022framework}. Noting that \texttt{ALSET} only differs from \texttt{TTSA} regarding time scales, we use \texttt{TTSA} to represent this class of approach. Also, we omit the comparison with \texttt{AmIGO}~\cite{arbel2022amortized} below, given that it is essentially a double-loop \texttt{SOBA} with increasing batch sizes. Detailed setups and additional experiments on all other methods are deferred in Appendix. The tunable parameters in benchmark methods are selected in the same manner as those in \texttt{benchmark\_bilevel}\sref{footnote: bilevel benchopt}.

\begin{figure}[t]
    \centering
    \addtocounter{subfigure}{-1}
    \subfigure{\includegraphics[width=0.5\textwidth,trim={0 8.5cm 0 0},clip]{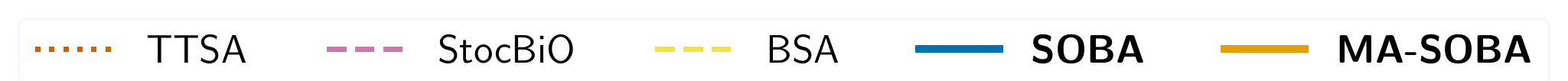}}\\
    \subfigure[$\ell^2$ Penalized Logistic Regression on IJCNN1]{\label{fig: logreg_ijcnn}\includegraphics[width=0.44\textwidth, trim={0 0 0 1.4cm},clip]{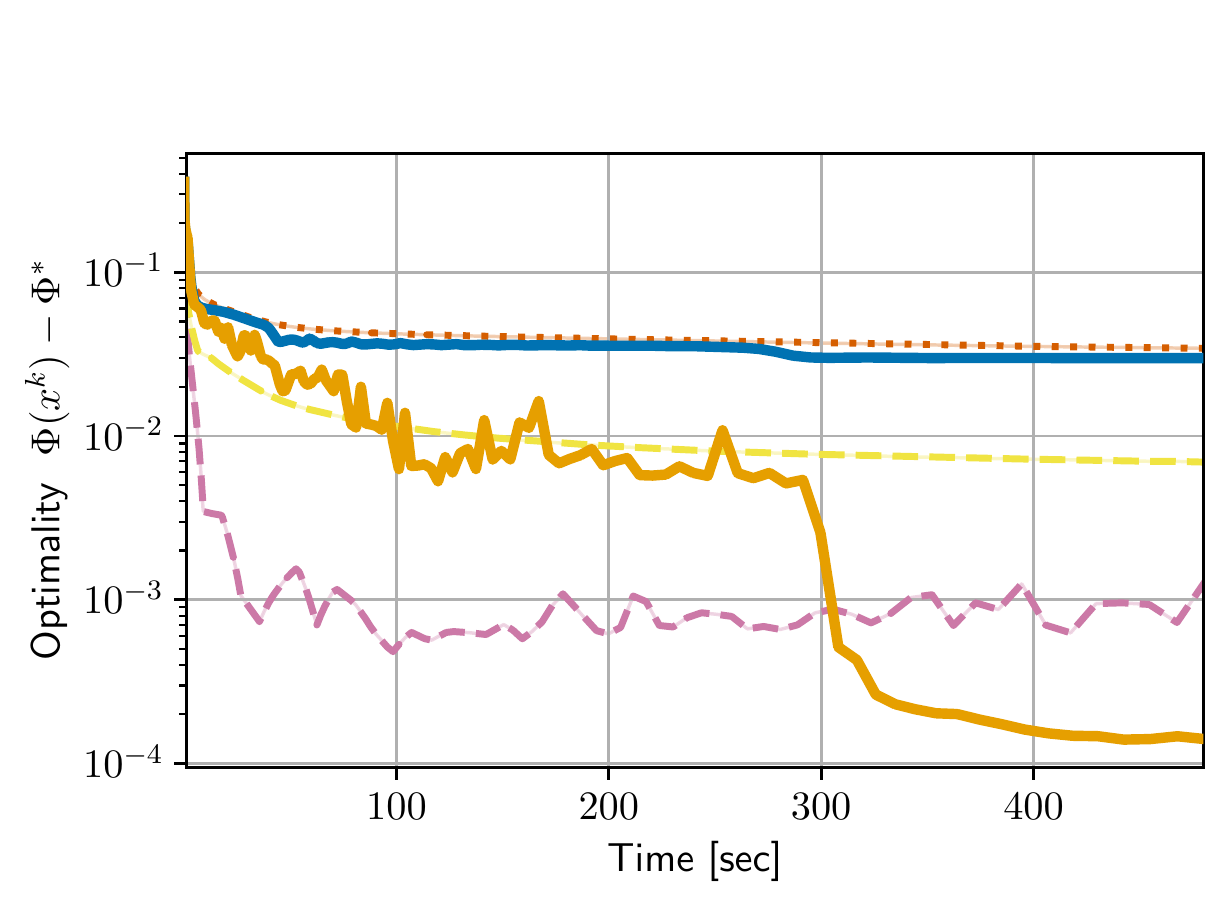}}
    \subfigure[Data Hyper-Clearning on MNIST]{\label{fig: dataclean_mnist}\includegraphics[width=0.44\textwidth, trim={0 0 0 1.4cm},clip]{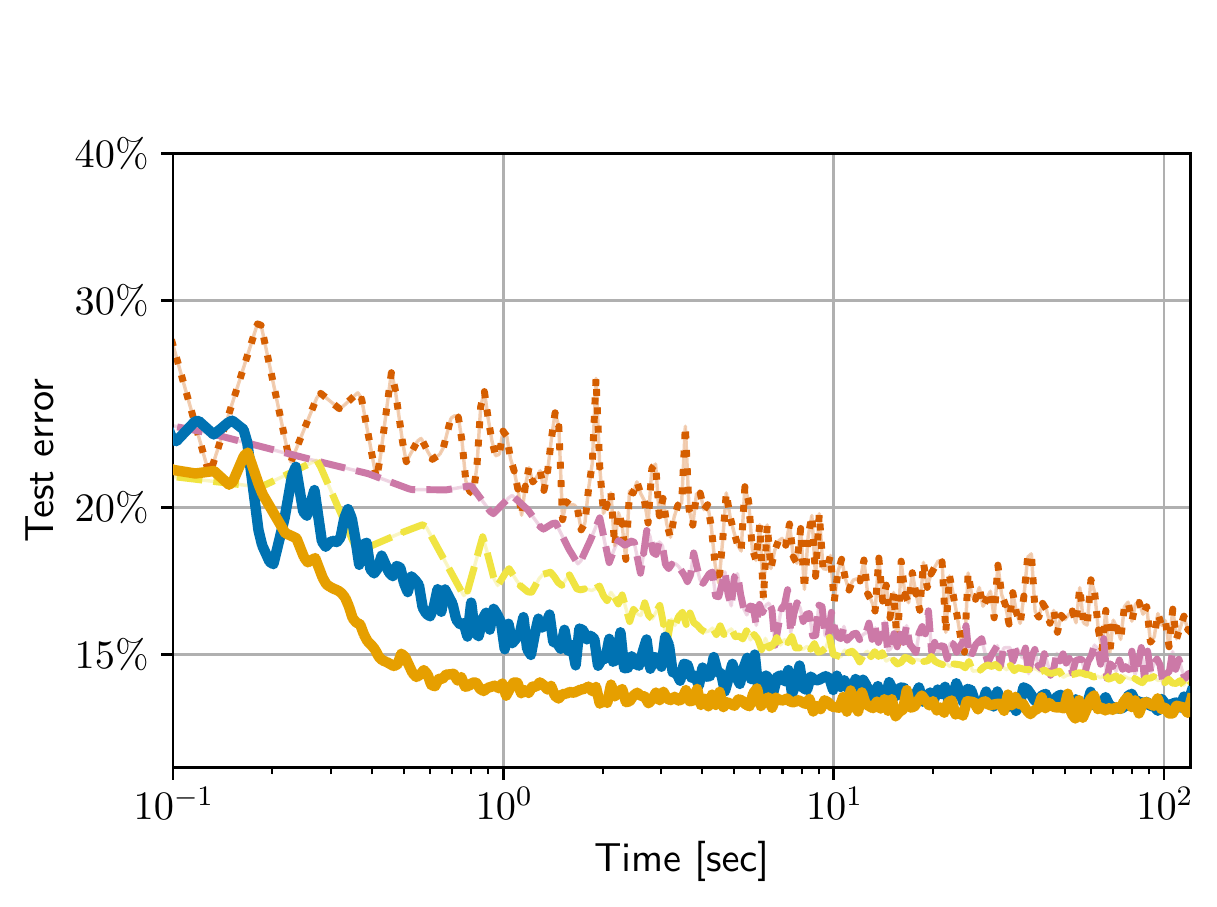}}
    \caption{Comparison of \texttt{MA-SOBA} with other stochastic bilevel optimization methods without using variance reduction techniques. For each algorithm, we plot the median performance over 10 runs. \textbf{Left:} Hyperparameter optimization for $\ell^2$ penalized logistic regression on IJCNN1 dataset. \textbf{Right.} Data hyper-cleaning on MNIST with $p=0.5$ (corruption rate).}
\end{figure}

In the first task, we fit binary classification models on the IJCNN1 dataset\footnote{\url{https://www.csie.ntu.edu.tw/~cjlin/libsvmtools/datasets/binary.html}}. The function $f$ and $g$ of the problem \eqref{eq: bo} are the average logistic loss on the validation set and training set respectively, with $\ell^2$ regularization for $g$. In Figure \ref{fig: logreg_ijcnn}, we plot the suboptimality gap against the runtime for each method. Surprisingly, we observed that \texttt{MA-SOBA} achieves lower objective values after several iterations compared to all benchmark methods. This improvement can be attributed to the convergence of average hypergradients $\{h^k\}$. These findings demonstrate the practical superiority of our algorithm framework, even with the same sample complexity results.

In the second task, we conduct data hyper-cleaning on the MNIST dataset introduced in \cite{franceschi2017forward}. Data cleaning aims to train a multinomial logistic regression model on the corrupted training set and determine a weight for each training sample. These weights should approach zero for samples with corrupted labels. We randomly replace the label by $\{0,1\dots, 9\}$ in the training set of MNIST with probability $p$. The task can be formulated into the bilevel optimization problem \eqref{eq: bo} with the inner variable $y$ being the regression coefficients and the outer variable $x$ being the sample weight. The \texttt{LL} function $g$ is the sample-weighted cross-entropy loss on the corrupted training set with $\ell^2$ regularization. The \texttt{UL} function $f$ is the cross-entropy loss on the validation set. We report the test error in Figure \ref{fig: dataclean_mnist}. We observe that \texttt{MA-SOBA} outperforms other benchmark methods by achieving lower test errors faster.

To demonstrate the practical performance of \texttt{MORMA-SOBA}, we conduct experiments in \emph{robust multi-task representation learning} introduced in \cite{gu2022min} on the FashionMNIST dataset~\cite{xiao2017online}. Each bilevel objective $\Phi_i$ in this setup represents a distinct learning ``task'' $i\in[n]$ with its own training and validation sets. The optimization variable is engaged in a shared representation network, parameterized by the outer variable $x$, along with per-task linear models parameterized by each inner variable $y_i$. The \texttt{UL} function $f_i$ is the average cross-entropy loss over the $i$-th validation set, and the \texttt{LL} function $g_i$ is the $\ell^2$ regularized cross-entropy loss over the $i$-th training set. The goal is to learn a shared representation and per-task models that generalize well on each task, which is usually solved by a single-objective problem that minimizes $(1/n)\sum_i \Phi_i$. In Figure \ref{fig: min-max}, we compare our algorithm with the existing min-max bilevel algorithm \texttt{MORBiT}~\cite{gu2022min} in terms of the average loss $((1/n)\sum_i \Phi_i)$ and maximum loss $(\max_i \Phi_i)$. The results demonstrate the superiority of \texttt{MORMA-SOBA} over \texttt{MORBiT} in terms of lowering both the max loss and average loss at a faster rate.

\begin{figure}[t]
  \begin{center}
    \includegraphics[width=0.6\textwidth]{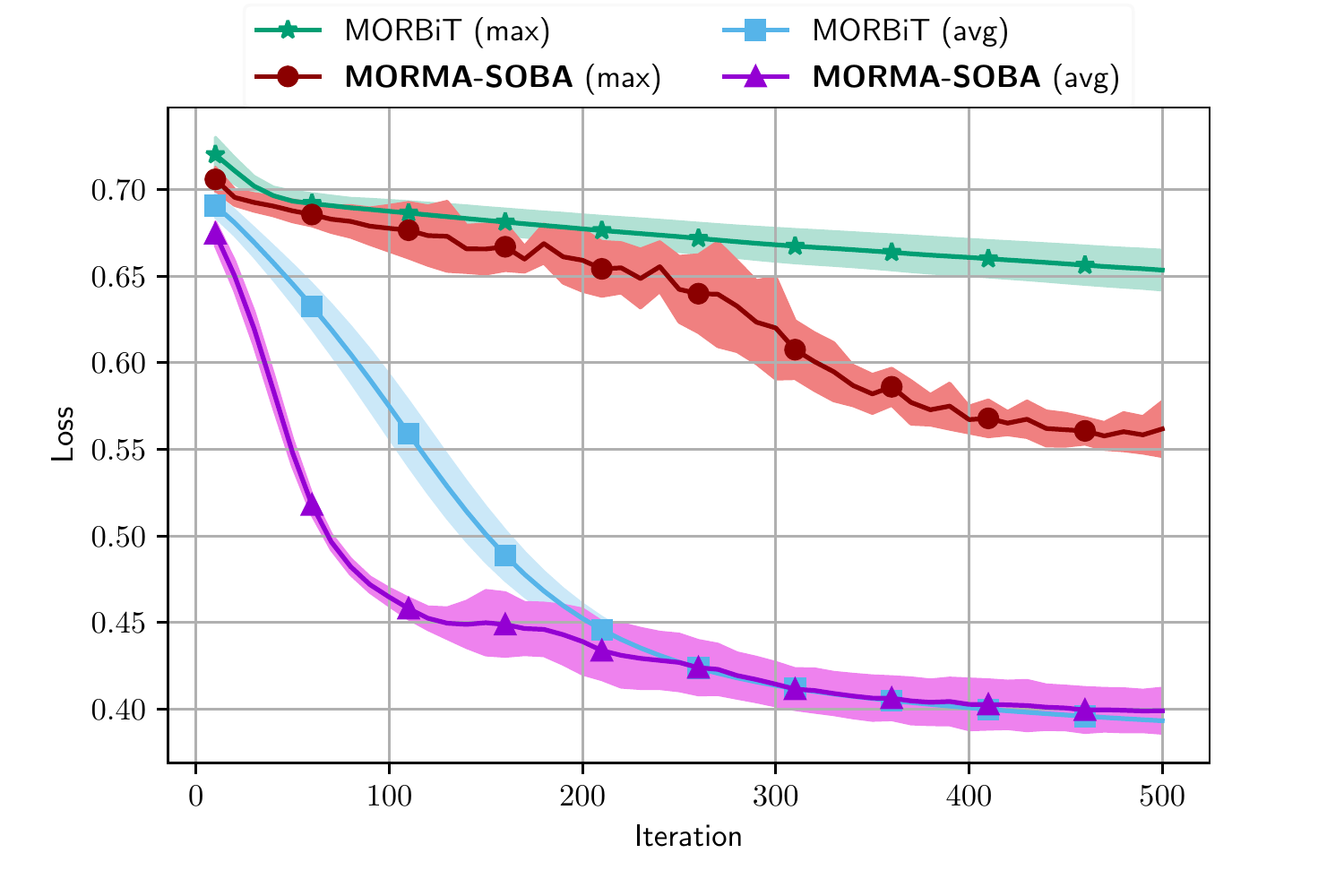}
  \end{center}
  \caption{\texttt{MORMA-SOBA} ($\mu_\lambda=0.01$) vs. \texttt{MORBiT} on robust multi-task representation learning}
  \label{fig: min-max}
\end{figure}

\section{Conclusion}

In this work, we propose a novel class of algorithms (\texttt{MA-SOBA}) for solving stochastic bilevel optimization problems in \eqref{eq: bo} by introducing the moving-average step to estimate the hypergradient. We present a refined convergence analysis of our algorithm, achieving the optimal sample complexity without relying on the high-order smoothness assumptions employed in the literature. Furthermore, we extend our algorithm framework to tackle a generic min-max bilevel optimization problem within the multi-objective setting, identifying and addressing the theoretical gap present in the literature.

\subsubsection*{Acknowledgements}
We thank the authors of~\cite{gu2022min} for clarifications regarding their paper.

\bibliographystyle{abbrv}
\bibliography{bibfile}

\appendix

\section{Experimental Details}\label{sec: exp}

All the experiments were conducted using Python. The initial two tasks, involving the comparison of \texttt{MA-SOBA} with other stochastic bilevel optimization algorithms, utilized the \texttt{Benchopt} package~\cite{moreau2022benchopt} and the open-sourced bilevel benchmark~\cite{dagreou2022framework}\footnote{\label{bench_bilevel}\url{https://github.com/benchopt/benchmark_bilevel}}. The final task, focused on robust multi-objective representation learning, was implemented in PyTorch, following the source code provided by \cite{gu2022min}\footnote{\url{https://github.com/minimario/MORBiT}}.

\subsection{Experimental Details for MA-SOBA}

\textbf{Setup.} In our experiments, we strictly adhere to the settings provided in \texttt{bench\_bilevel}\sref{bench_bilevel}, as detailed in Appendix B.1 of \cite{dagreou2022framework}. The previous results and setups of \cite{dagreou2022framework} have also been available in \url{https://benchopt.github.io/results/benchmark_bilevel.html}. For completeness, we provide a summary of the setup below.
\begin{itemize}[leftmargin=1.5em]
    \item To avoid redundant computations, we utilize oracles for the function $F_\xi, G_\phi$, which provide access to quantities such as $\nabla_1 F_\xi(x,y)$, $\nabla_2 F_\xi(x, y)$, $\nabla_2 G_\phi(x,y)$, $\nabla_{22}^{2} G_\phi(x,y)v$, and $\nabla_{12}^{2} G_\phi(x,y)v$, although this approach may violate the independence assumption in Assumption \ref{aspt: stochastic_grad}.
    \item In all our experiments, we employ a batch size of 64 for all methods, even for \texttt{BSA} and \texttt{AmIGO} that theoretically require increasing batch sizes.
    \item For methods involving an inner loop (\texttt{stocBiO}, \texttt{BSA}, \texttt{AmIGO}), we perform 10 inner steps per each outer iteration as proposed in those papers. 
    \item For methods that involve the Neumann approximation for the Hessian vector product (such as \texttt{BSA}, \texttt{TTSA}, \texttt{SUSTAIN}, and \texttt{MRBO}), we perform 10 steps of the subroutine per outer iteration. For \texttt{AmIGO}, we perform 10 steps of SGD to approximate the inversion of the linear system.
    \item The step sizes and momentum parameters used in all benchmark algorithms are directly adopted from the fine-tuned parameters provided by \cite{dagreou2022framework}. From a grid search, we select the best constant step sizes for \texttt{MO-SOBA}.
\end{itemize}

At present, we have excluded \texttt{SRBA}~\cite{dagreou2023lower} from the benchmark due to the unavailability of an open-sourced implementation and its limited reported improvement over \texttt{SABA}.

\subsubsection{Hyperparameter Optimization on IJCNN1}
In this experiment, we focus on selecting the regularization parameters for a multi-regularized logistic regression model on the IJCNN1 dataset, where we have one hyperparameter per feature. Specifically, the problem can be formulated as:

\begin{equation}
\begin{split}
    \underset{\nu\in\realset^d}{\min}&\quad  \Phi(\nu) \coloneqq  \underbrace{\underset{(X, Y) \sim \cD_{\text{val}}}{\E} \left[\ell\left(\<\omega^*(\nu), X> , Y\right)\right]}_{f(\nu, \omega^*(\nu))}\\
    \text{s.t.} &\quad \omega^*(\nu) =\underset{\omega\in\realset^d}{\argmin} \underbrace{\underset{(X, Y) \sim \cD_{\text{train}}}{\E} \left[\ell\left(\<\omega, X>, Y\right)\right] + \frac{1}{2} \omega^\top \diag{e^{\nu_1}, \dots, e^{\nu_d}} \omega}_{g(\nu, \omega)}
\end{split}    
\end{equation}

In this case, $|\cD_{\text{train}}| = 49,990$, $|\cD_{\text{val}}| = 91,701$, and $d=22$. For each sample, the covariate and label are denoted as $(X, Y)$, where $X\in\realset^{22}$ and $Y\in \{0,1\}$. The inner variable ($\omega\in\realset^{22}$) is the regression coefficient. The outer variable ($\nu\in\realset^{22}$) is a vector of regularization parameters. The loss function $\ell(y', y) = -y\log (y') - (1-y) \log(1-y')$ is the log loss.

To complement the comparison presented in the main paper, we conducted additional experiments that involved comparing all benchmark methods, including the variance reduction based method. In Figure \ref{fig: ijcnn_all_appendix}, we plot the suboptimality gap ($\Phi(x) - \Phi^*$) against runtime and the number of calls to oracles. Unfortunately, the previous results obtained for \texttt{MRBO} and \texttt{AmIGO} on the IJCNN1 dataset are not reproducible at the moment due to some conflicts in the current developer version of \texttt{Benchopt}. As reported in \cite{dagreou2022framework}, \texttt{MRBO} exhibits similar performance to \texttt{SUSTAIN}, while the curve of \texttt{AmIGO} initially follows a similar trend as \texttt{SUSTAIN} and eventually reaches a similar level as \texttt{SABA} towards the end. Following a grid search, we have selected the parameters in \texttt{MA-SOBA} as $\alpha_k\tau = 0.02$, $\beta_k=\gamma_k=0.01$, and $\theta_k = 0.1$. As shown in Figure \ref{fig: ijcnn_all_appendix}, our proposed method \texttt{MA-SOBA} outperforms \texttt{SOBA} significantly, achieving a slightly lower suboptimality gap compared to the state-of-the-art variance reduction-based method \texttt{SABA}.

\begin{figure}[h]
    \centering
    \subfigure{\includegraphics[width=0.6\textwidth,trim={0 8cm 0 0},clip]{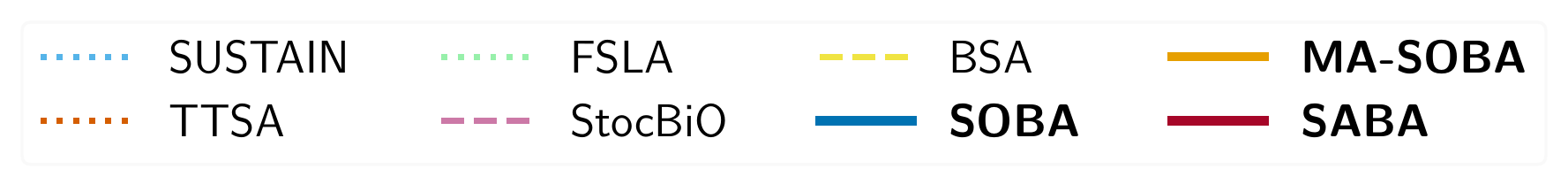}}\\
    \subfigure{\includegraphics[width=0.45\textwidth,trim={1.5cm 0 3.8cm 2cm},clip]{figures/ijcnn1_all.pdf}}
    \subfigure{\includegraphics[width=0.45\textwidth,trim={1.5cm 0 3.8cm 2cm},clip]{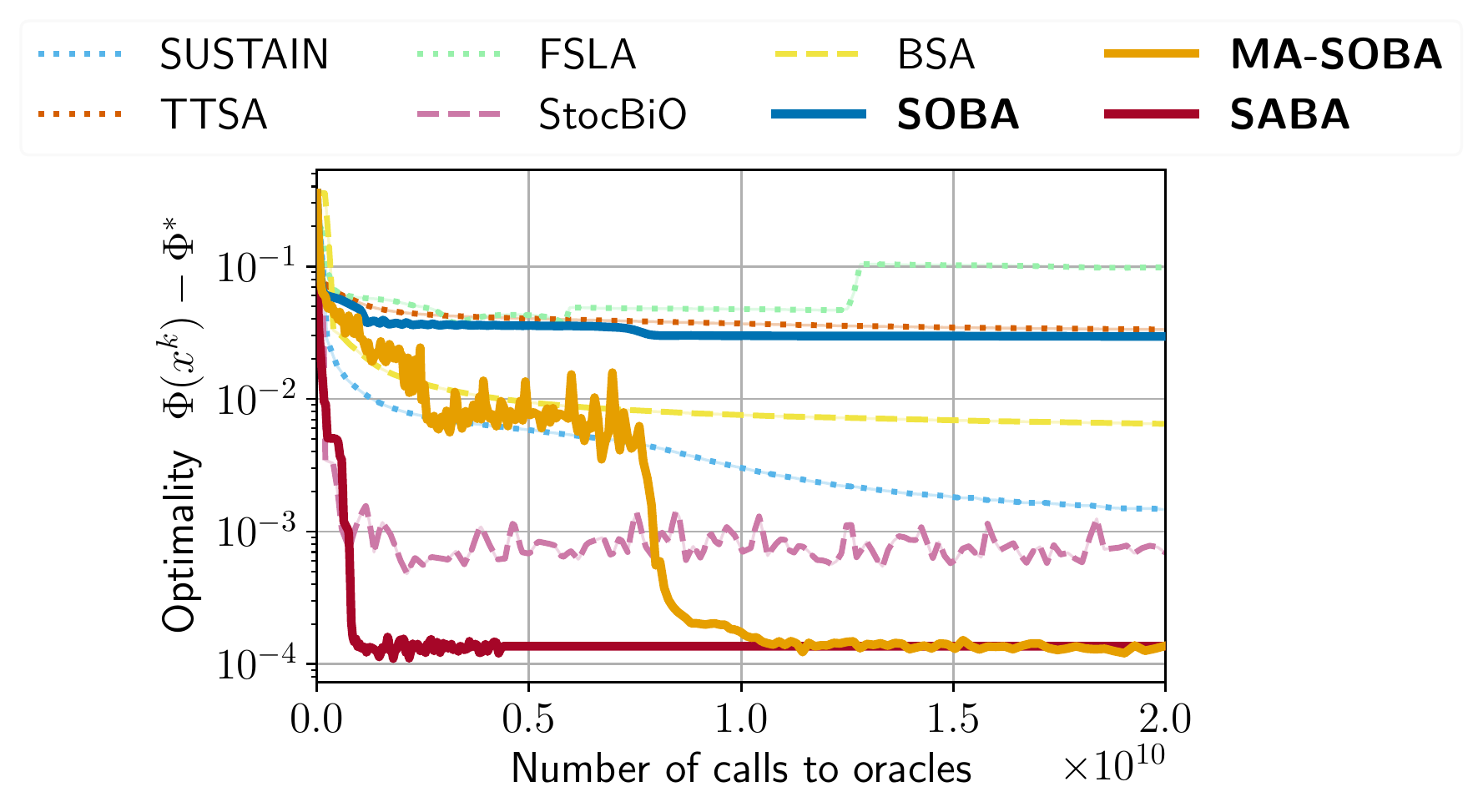}}
    \caption{Comparison of \texttt{MA-SOBA} with other stochastic bilevel optimization methods in the problem of hyperparameter optimization for $\ell^2$ regularized logistic regression on the IJCNN1 dataset. We plot the median performance over 10 runs for each method. \textbf{Left:} Performance in runtime; \textbf{Right:} Performance in the number of gradient/Hessian(Jacobian)-vector products sampled.}
    \label{fig: ijcnn_all_appendix}
\end{figure}

\subsubsection{Data Hyper-Cleaning on MNIST}

The second experiment we perform involves data hyper-cleaning on the MNIST dataset. The dataset is partitioned into a training set $\cD_{\text{train}}$, a validation set $\cD_{\text{val}}$, and a test set $\cD_{\text{test}}$, where $|\cD_{\text{train}}|=20,000$, $|\cD_{\text{val}}|=5,000$, and $|\cD_{\text{test}}|=10,000$. Each sample is represented as a vector $X$ of dimension 784, where the input image is flattened. The corresponding label takes values from the set $\{0,1,\dots,9\}$. We use $Y\in\realset^{10}$ to denote its one-hot encoding. Each sample in the training set is corrupted with probability $p$ by replacing its label with a random label $\{0,1,\dots,9\}$. The task of data hyper-cleaning can be formulated into the bilevel optimization problem as below:
\begin{equation}
\begin{split}
    \underset{\nu\in\realset^{|\cD_{\text{train}}|}}{\min}&\quad  \Phi(\nu) \coloneqq \underbrace{\underset{(X, Y) \sim \cD_{\text{val}}}{\E} \left[\ell(W^*(\nu) X, Y)\right]}_{f(\nu, W^*(\nu))}\\
    \text{s.t.} &\quad W^*(\nu) =\underset{\omega\in\realset^d}{\argmin} \underbrace{ \frac{1}{|\cD_{\text{train}}|} \sum_{(X_i, Y_i)\sim \cD_{\text{train}}} \sigma(\nu_i)\ell(WX_i, \overbrace{\widetilde{Y_i}}^{\text{corrupted}}) + C_r\norm{W}^2 }_{g(\nu, W)},
\end{split}    
\end{equation}
where the outer variable ($\nu\in\realset^{20,000}$) is a vector of sample weights for the training set, the inner variable $W\in\realset^{10\times 784}$, and $\ell$ is the cross entropy loss and $\sigma$ is the sigmoid function. The regularization parameter $C_r=0.2$ following \cite{dagreou2022framework}. The objective of data hyper-cleaning is to train a multinomial logistic regression model on the training set and determine a weight for each training sample using the validation set. The weights are designed to approach zero for corrupted samples, thereby aiding in the removal of these samples during the training process.

To supplement the comparison presented in the main paper, we conducted additional experiments that involved comparing all benchmark methods, including the variance reduction-based method. Following a grid search, we have selected the parameters in \texttt{MA-SOBA} as $\alpha_k\tau = 10^3$, $\beta_k = \gamma_k = 10^{-2}$, and $\theta_k=10^{-1}$. In Figure \ref{fig: dataclean-mnist-appendix}, we plot the test error against runtime and the number of calls to oracles with different corruption probability $p\in \{0.5, 0.7, 0.9\}$. We observe that \texttt{MA-SOBA} has comparable performance to the state-of-the-art method \texttt{SABA}. Remarkably, \texttt{MA-SOBA} is the fastest algorithm to reach the best test accuracy when $p=0.5$.

\begin{figure}[h]
    \centering
    \subfigure{\includegraphics[width=0.8\textwidth,trim={0 8.2cm 0 0},clip]{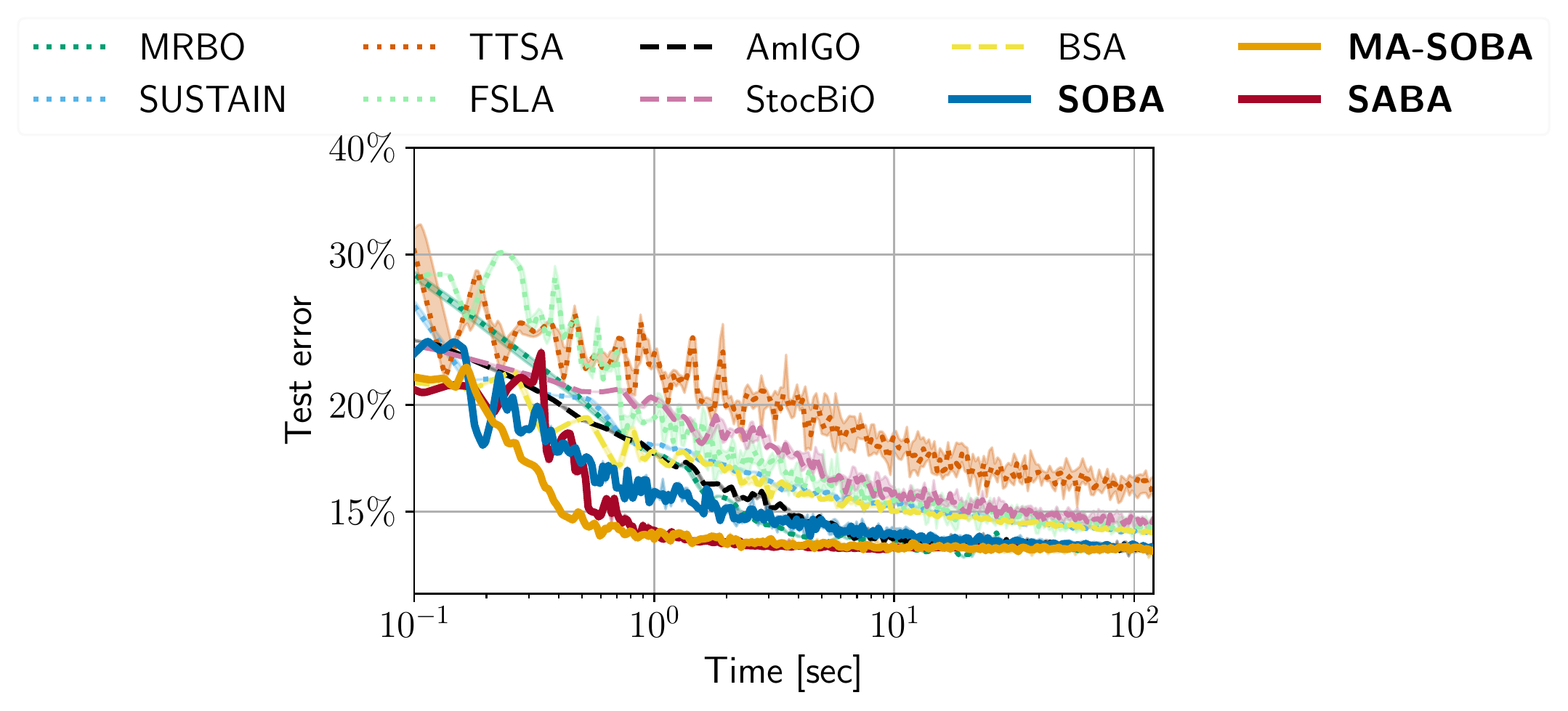}}\\
    \subfigure{\includegraphics[width=0.31\textwidth,trim={3.8cm 0 5.5cm 1.8cm},clip]{figures/datacleaning0_5_all.pdf}}
    \subfigure{\includegraphics[width=0.31\textwidth,trim={3.8cm 0 5.5cm 1.8cm},clip]{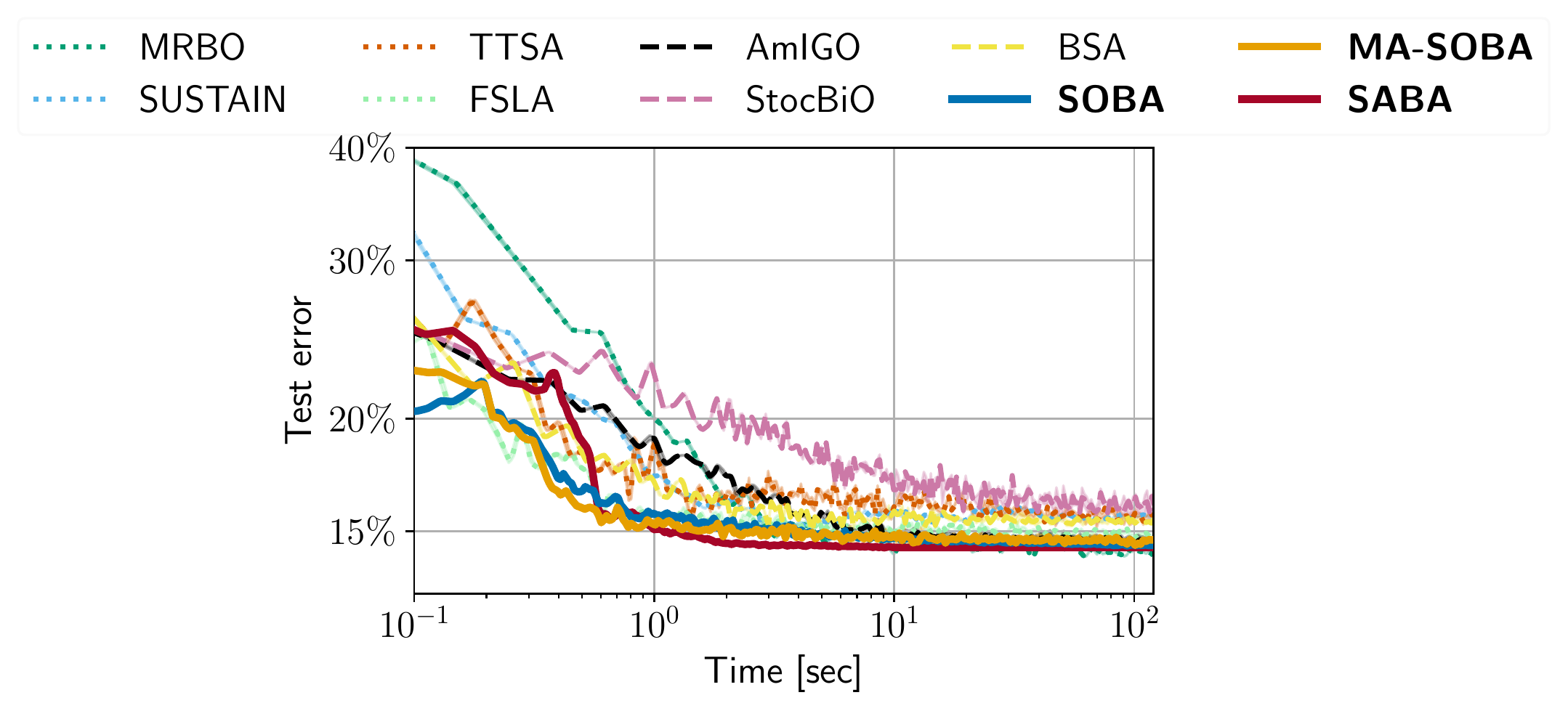}}
    \subfigure{\includegraphics[width=0.31\textwidth,trim={3.8cm 0 5.5cm 1.8cm},clip]{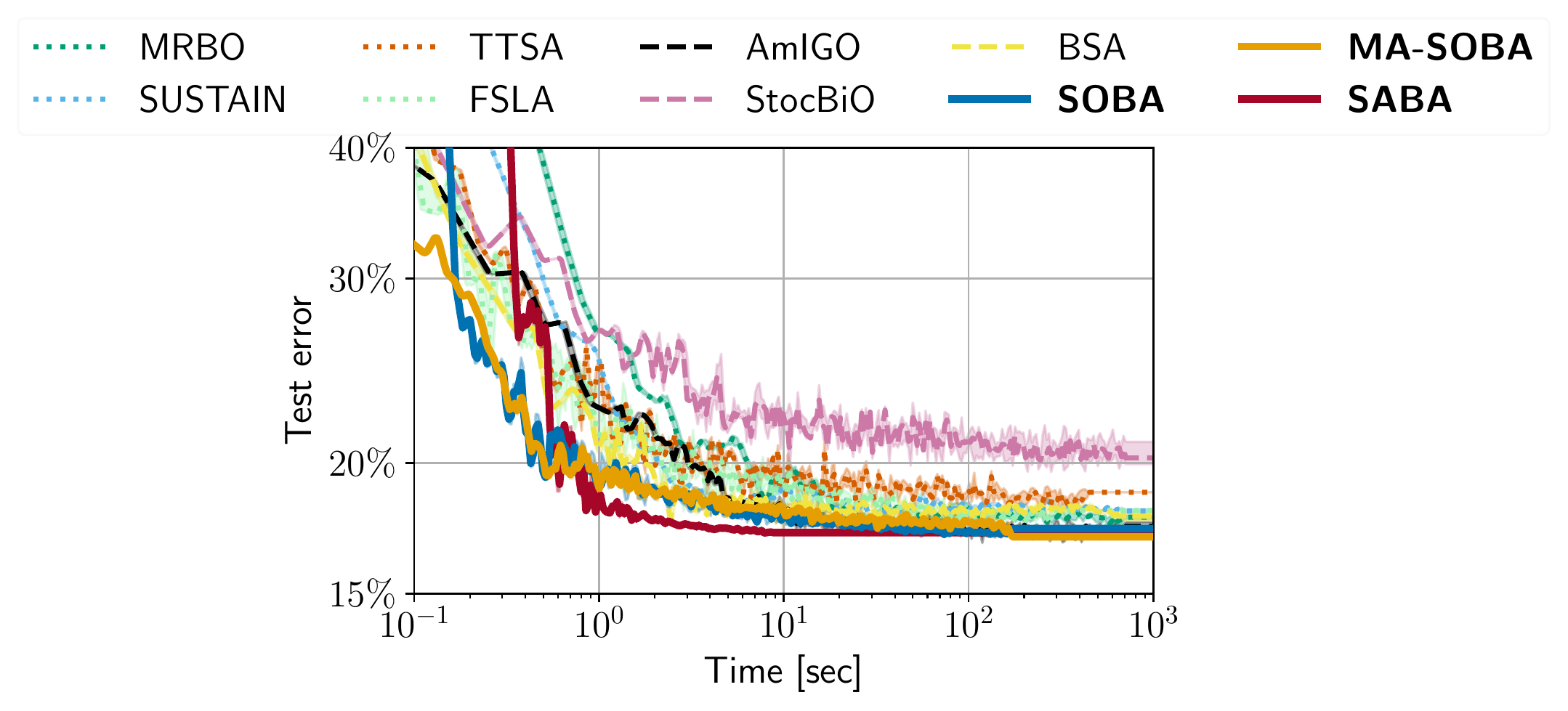}}\\
    \addtocounter{subfigure}{-4}
    \subfigure[$p=0.5$]{\includegraphics[width=0.31\textwidth,trim={3.8cm 0 5.5cm 1.8cm},clip]{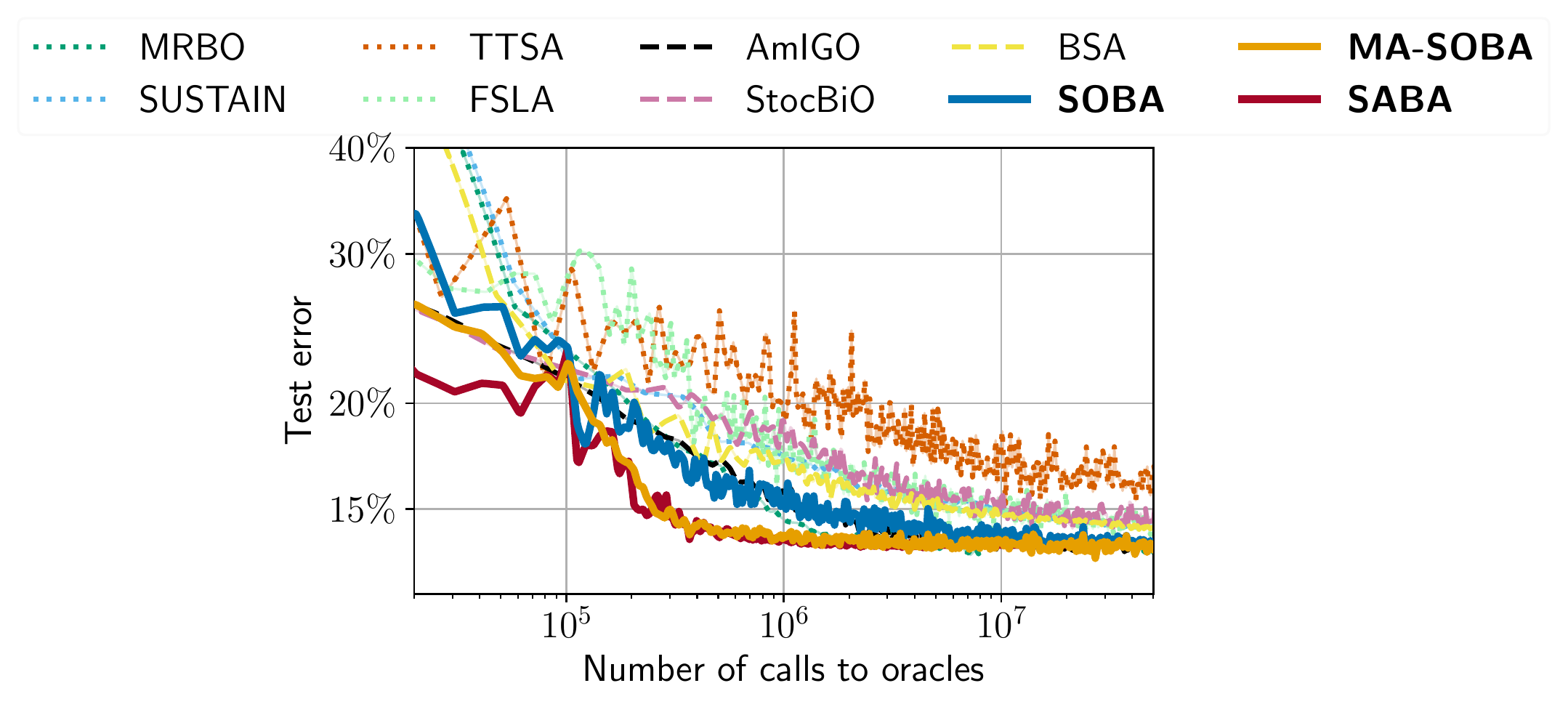}}
    \subfigure[$p=0.7$]{\includegraphics[width=0.31\textwidth,trim={3.8cm 0 5.5cm 1.8cm},clip]{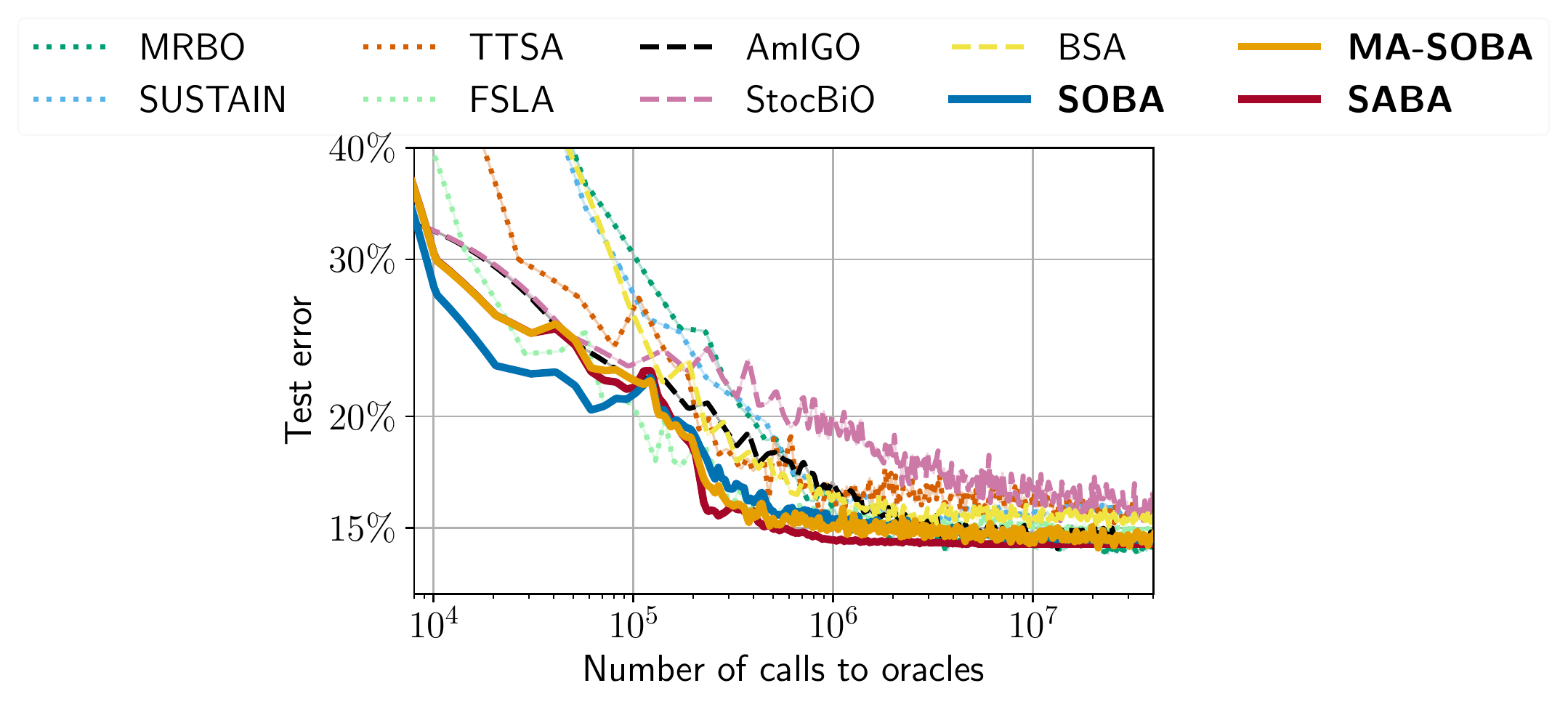}}
    \subfigure[$p=0.9$]{\includegraphics[width=0.31\textwidth,trim={3.8cm 0 5.5cm 1.8cm},clip]{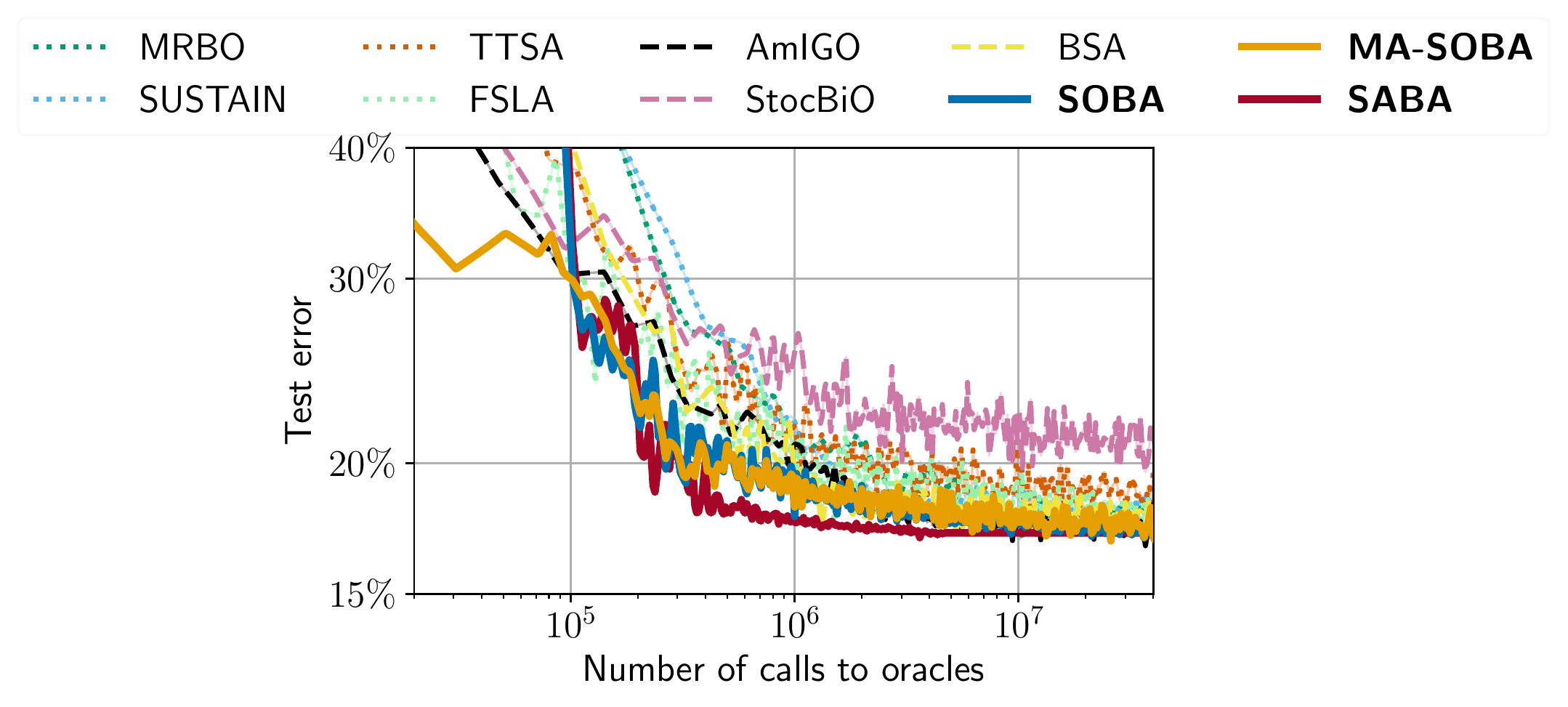}}
    \caption{Comparison of \texttt{MA-SOBA} with other stochastic bilevel optimization methods in the problem of data hyper-cleaning on the MNIST dataset when the corruption probability $p\in\{0.5, 0.7, 0.9\}$. We plot the median performance over 10 runs for each method. \textbf{Top:} Performance in runtime; \textbf{Bottom:} Performance in the number of gradient/Hessian(Jacobian)-vector products sampled.}
    \label{fig: dataclean-mnist-appendix}
\end{figure}

\subsubsection{Moving Average vs. Variance Reduction}

Through empirical studies, we have demonstrated that our proposed method, \texttt{MA-SOBA}, which utilizes a moving average (MA) technique, achieves comparable performance to the state-of-the-art variance reduction-based approach \texttt{SABA} using \texttt{SAGA} updates~\cite{defazio2014saga}. In this context, we would like to highlight the key difference and relationship between these two methods.

We start with presenting the update rules of the sequence of estimated gradients $\{g^k\}$ for the variance reduction techniques \texttt{SAGA}~\cite{defazio2014saga} and our moving average method (\texttt{MA}) for the single-level problem.
\begin{thmbox}
    \textbf{SAGA (finite-sum:) $\min~\frac{1}{n}\sum_{i=1}^{n} f_i(x)$}
    \begin{align*}
        g^k = \nabla f_{i_k}(x^k) - \nabla f_{i_k}(\bar{x}_{i_k}) + \frac{1}{n}\sum_{j=1}^{n}\nabla f_j(\bar{x}_j)
    \end{align*}
\end{thmbox}

The \texttt{SAGA} update is designed for finite-sum problems with offline batch data. At each iteration $k$, the algorithm randomly selects an index $i_k \in [n]$ and updates the gradient variable $g^k$ using a reference point $\bar{x}_{i_k}$, which corresponds to the last evaluated point for $\nabla f_{i_k}$. However, it should be noted that \texttt{SAGA} requires storing the previously evaluated gradients ${\nabla f_j(\bar{x}_j)}$ in a table, which can be memory-intensive when sample size $n$ or dimension $d$ is large. In the finite-sum setting, there exist several other variance reduction methods, such as \texttt{SARAH}~\cite{dagreou2023lower}, that can be employed to further enhance the dependence on the number of samples, $n$, for bilevel optimization problems. However, the \texttt{SARAH}-type method requires double gradient evaluations on each iteration of $x^k$ and $x^{k-1}$.


\begin{thmbox}
    \textbf{MA (expectation):} $\min~\E_{\xi}[f(x;\xi)]$
    \begin{align*}
        g^k = (1-\alpha_k) g^{k-1} + \alpha_k \nabla f(x^k; \xi^{k+1})
    \end{align*}
\end{thmbox}

Unlike variance reduction techniques, the moving average methods can solve the general expectation-form problem with online and streaming data using a simple update per iteration. In addition, the moving average techniques offer two advantages compared to variance reduction-based methods:

\textbf{Theoretical Assumption.} All variance reduction methods, including \texttt{SVRG}~\cite{reddi2016stochastic}, \texttt{SAGA}~\cite{defazio2014saga}, \texttt{SARAH}~\cite{nguyen2017sarah}, \texttt{STORM}~\cite{cutkosky2019momentum}, and others, typically rely on assuming mean-squared smoothness assumptions. In particular, for stochastic optimization problems in the form of $\min_x \{f(x) = \mathbb{E}[F(x,\xi)]\}$, the definition of mean-squared smoothness (MSS) is:
$$(\text{MSS}) \quad \mathbb{E}_{\xi}[\|\nabla F(x, \xi) - \nabla F(x', \xi)\|^2] \leq L^2 \| x-x' \|^2.$$
However, MSS is a stronger assumption than the general smoothness assumption on $f$:
$$\|\nabla f(x) - \nabla f(x')\| \leq L\| x- x' \|.$$
By Jensen’s inequality, we have that MSS is stronger than the general smoothness assumption on $f$:
$$\|\nabla f(x) - \nabla f(x')\|^2 \leq \mathbb{E}_{\xi}[\|\nabla F(x, \xi) - \nabla F(x', \xi)\|^2].$$
In this work, the theoretical results of the proposed methods are only built on the smoothness assumption on the \texttt{UL} and \texttt{LL} functions $f, g$ without further assuming MSS on $F_\xi$ and $G_\phi$. It is worth noting that a clear distinction in the lower bounds of sample complexity for solving the single-level stochastic optimization has been proven in \cite{arjevani2023lower}. Specifically, they establish a separation under the MSS assumption on $F_\xi$ and smoothness assumptions on $f$ ($\cO(\epsilon^{-1.5})$ vs. $\cO(\epsilon^{-2})$). 
Thus, it is important to emphasize that \texttt{MA-SOBA} achieves the optimal sample complexity $\cO(\epsilon^{-2})$ under our weaker assumptions.

\textbf{Practical Implementation.} Variance reduction methods often entail additional space complexity, require double-loop implementation or double oracle computations per iteration. These requirements can be unfavorable for large-scale problems with limited computing resources. For instance, in the second task, the runtime improvement achieved by using \texttt{SABA} is limited. This limitation can be attributed to the dimensionality of the variables $\nu$ (with a dimension of $20,000$) and $W$ (with a dimension of $10 \times 784$). The benefit of using variance reduction methods is expected to be less significant for more complex problems involving computationally expensive oracle evaluations.

\subsection{Experimental Details for MORMA-SOBA}

We adopt the same setup as described in \cite{gu2022min}, which can be summarized as follows.

\textbf{Setup.} We consider binary classification tasks generated from the FashionMNIST data set where we select 8 ``easy'' tasks (lowest loss $\sim 0.3$ from independent training) and 2 ``hard'' tasks (lowest loss  $\sim 0.45$ from independent training) for multi-objective robust representation learning:
\begin{itemize}
    \item ``easy'' tasks: (0, 9), (1, 7), (2, 7), (2, 9), (4, 7), (4, 9), (3, 7), (3, 9)
    \item ``hard'' tasks: (0, 6), (2, 4)
\end{itemize}
For each task $i\in[10]$ above, we partition its dataset into the training set $\cD_i^{\text{train}}$, validation set $\cD_i^{\text{val}}$, and test set $\cD_{i}^{\text{test}}$. We also generate 7 (unseen) binary classification tasks for testing:
\begin{itemize}
    \item ``easy'' tasks: (1, 9), (2, 5), (4, 5), (5, 6)
    \item ``hard'' tasks: (2, 6), (3, 6), (4, 6)
\end{itemize}
We train a shared representation network that maps the 784-dimensional (vectorized 28x28 images) input to a 100-dimensional space. Subsequently, each task learns a binary classifier based on this shared representation. To learn a shared representation and per-task models that generalize well on each task, we aim to solve the following min-max bilevel optimization problem:
\begin{equation*}
\begin{split}
        &\quad \underset{E\in \realset^{100\times 784}}{\min}\ \max_{1\leq i \leq n } \ \Phi_i(E) \coloneqq \underbrace{\underset{(X, Y)\sim \cD_i^{\text{val}}}{\E}\left[\ell\left( W_i^*(E) \circ\overbrace{\text{ReLU}(EX)}^{\text{representation}} + b_i^*(E), Y\right)\right]}_{f_i(E, (W_i^*, b_i^*))} \\
        &\text{s.t. }\ \begin{pmatrix}W_i^*(E)\\ b_i^*(E)\end{pmatrix} = \\
        &\underset{W_i\in\realset^{10\times 100},b_i\in\realset^{10}}{\argmin} \underbrace{\underset{(X, Y)\sim \cD_i^{\text{train}}}{\E}\left[\ell\left(\overbrace{W_i}^{\text{weight}} \circ \text{ReLU}(EX) + \overbrace{b_i}^{\text{bias}}, Y\right)\right]+ \rho\norm{W_i}_F^2}_{g_i(E, (W_i, b_i))} , 1\leq i\leq n.
\end{split}
\end{equation*}
Each sample is represented as a vector $X$ of dimension 784, where the input image is flattened. The corresponding label takes values from the set $\{0,1,\dots,9\}$. We use $Y\in\realset^{10}$ to denote its one-hot encoding. Each bilevel objective $\Phi_i$ above represents a distinct binary classification task $i\in[n]$. The optimization variable is engaged in a shared representation network, parameterized by the outer variable $E\in\realset^{100\times 784}$, along with per-task linear models parameterized by each inner variable $(W_i, b_i)$. The \texttt{UL} function $f_i$ is the average cross-entropy loss over the $\cD_{i}^{\text{val}}$, and the \texttt{LL} function $g_i$ is the $\ell^2$ regularized cross-entropy loss over $\cD_i^{\text{train}}$. 

In the experiment, the regularization parameter in the \texttt{LL} function $\rho=5\times 10^{-4}$. The implementation of \texttt{MORBiT} follows the same manner described in \cite{gu2022min}. Specifically, the code of \texttt{MORBiT} \cite{gu2022min} uses vanilla SGD with a learning rate scheduler and incorporates momentum and weight decay techniques to optimize each variable:
\begin{itemize}
    \item Outer variable: learning rate = $0.01$, momentum = $0.9$, weight\_decay = $10^{-4}$
    \item Inner variable: learning rate = $0.01$, momentum = $0.9$, weight\_decay = $10^{-4}$
    \item Simplex variable: learning rate = $0.3$, momentum = $0.9$, weight\_decay = $10^{-4}$
\end{itemize}
In addition, \texttt{MORBiT} adopts a straightforward iterative auto-differentiation to calculate the hypergradient without using the Neumann approximation of the Hession inversion.

For the implementation of \texttt{MORMA-SOBA}, the regularization parameter $\mu_{\lambda}$ in \ref{eq: minmaxbo_reform_reg} is set to be $0.01$. All remaining parameters are chosen as constant values, as listed below:
\begin{itemize}
    \item Outer variable: $\tau_x= 1, \alpha_k = 0.02,$
    \item Inner variable: $\beta_k = 0.02$
    \item Auxiliary variable: $\gamma_k=0.02$
    \item Simplex variable: $\tau_\lambda = 1, \alpha_k = 0.02$
    \item Average gradient: $\theta_k=0.6$
\end{itemize}

Both evaluated methods use batch sizes of 8 and 128 to compute $g_i$ for each inner step and $f_i$ for each outer iteration, respectively. In addition to Figure \ref{fig: min-max}, which showcases the performance on 10 seen tasks used for representation learning, we present Figure \ref{fig: minmax_all_appendix}. This figure displays the maximum/average loss values against the number of iterations on test sets consisting of 10 seen tasks and 7 unseen tasks. Our proposed approach, \texttt{MORMA-SOBA}, demonstrates superior performance in terms of faster reduction of both the maximum and average loss.
\begin{figure}[h]
    \centering
    \subfigure{\includegraphics[width=0.8\textwidth,trim={0 9.2cm 0 0},clip]{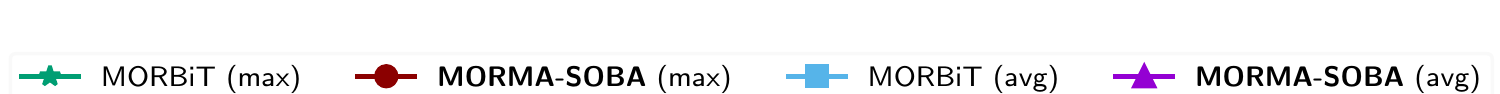}}\\
    \subfigure{\includegraphics[width=0.49\textwidth,trim={0.4cm 0.5cm 0.8cm 1.2cm},clip]{figures/minmax_seen_loss_appendix.pdf}}
    \subfigure{\includegraphics[width=0.49\textwidth,trim={0.4cm 0.5cm 0.8cm 1.2cm},clip]{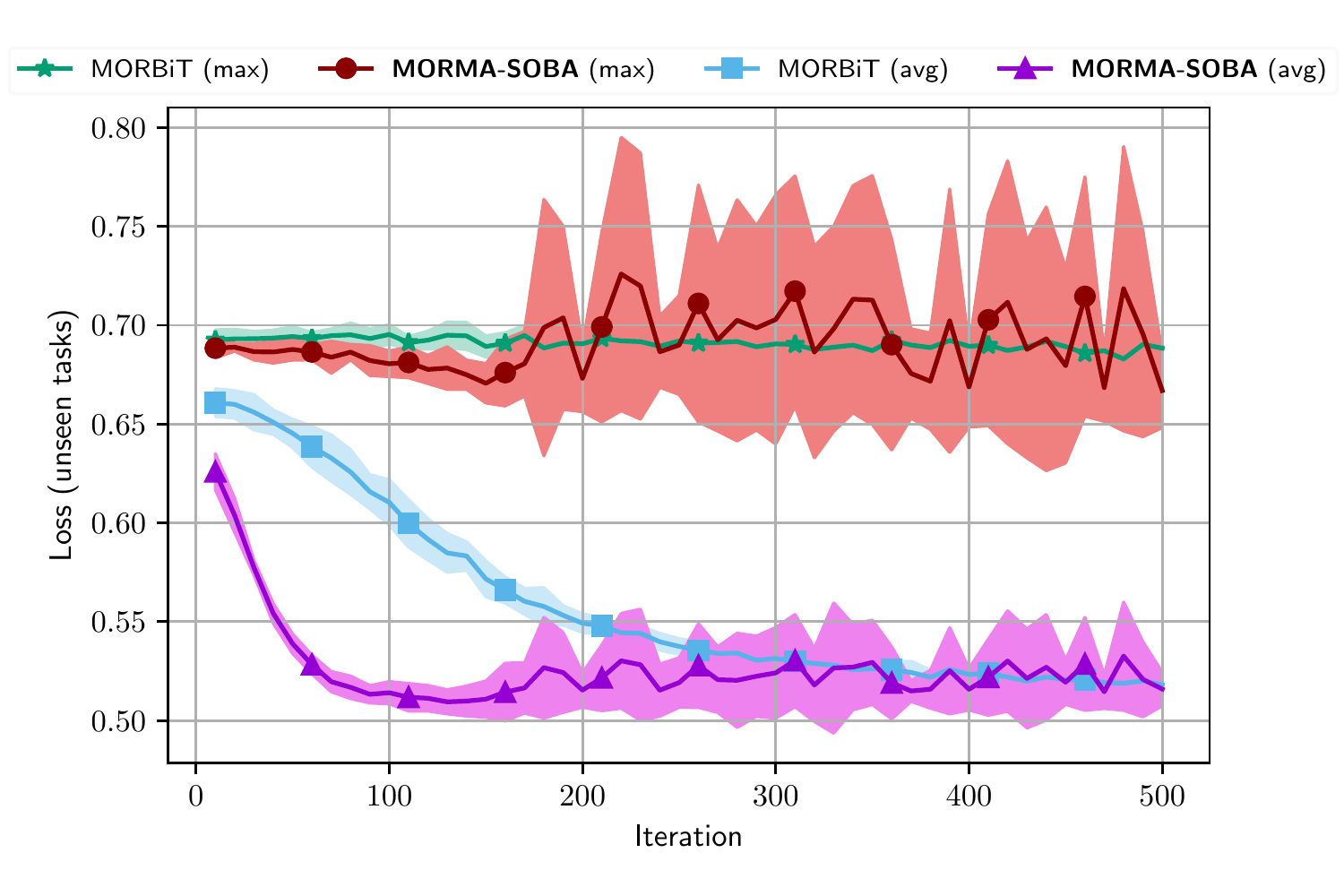}}
    \caption{Comparison of \texttt{MORMA-SOBA} with \texttt{MORBiT} in the problem of multi-objective robust representation learning for binary classification tasks on the FashionMNIST dataset. We aggregate the results over 10 runs for each method. \textbf{Left:} Performance on test sets of seen tasks; \textbf{Right:} Performance on unseen tasks.}
    \label{fig: minmax_all_appendix}
\end{figure}

\section{Proofs}
We will prove Theorems \ref{thm: soba_convergence} and \ref{thm: momasoba_convergence} in Section \ref{sec: thm1proof} and \ref{sec: thm2proof} respectively. In each section we will first establish the relations between the optimality measure (see $V_k, \tilde V_k$ in Sections \ref{sec: bo_analysis} and \ref{sec: convergence_momasoba}) and the gradient mapping, which reduce the proof of main theorems to proving the convergence of primal variables ($x^k$ in Theorem \ref{thm: soba_convergence} or $(x^k, \lambda^k)$ in Theorem \ref{thm: momasoba_convergence}) and dual variables ($h^k$ in Theorem \ref{thm: soba_convergence} or $(h_x^k, h_{\lambda}^k)$ in Theorem \ref{thm: momasoba_convergence}). Then we will prove the hypergradient estimation error, primal convergence and dual convergence separately. In our notation convention, the superscript $k$ usually denotes the iteration number and the subscript $i$ represents variables related to functions $f_i, g_i$. $L_{\#}$ with being a function $\#$ denotes its Lipschitz constant.

We first specify the constants in Assumption \ref{aspt: stochastic_grad}.\\
\textbf{Assumption \ref{aspt: stochastic_grad}.}
    For any $k\geq 0$, define $\setF_k$ denotes the sigma algebra generated by all iterates with superscripts not greater than $k$, i.e., $\setF_k = \sigma\left\{h^1, h^2, ..., h^k, x^1,..., x^k, y^1, ..., y^k, z^1,...,z^k\right\}$.
    The stochastic oracles used in Algorithm \ref{algo: momasoba} at $k$-th iteration are unbiased with bounded variance given $\setF_k$, i.e., there exist positive constants $\sigma_{f,1}, \sigma_{f,2}, \sigma_{g,1}, \sigma_{g,2}$ such that
    \begin{align*}
    &\E\left[u_x^{k+1}\middle|\setF_k\right] = \nabla_1 f(x^k, y^k),\  \E\left[\norm{u_x^{k+1} - \nabla_1 f(x^k, y^k)}^2\middle|\setF_k\right]\leq \sigma_{f,1}^2,\\
    &\E\left[u_y^{k+1} \middle|\setF_k\right] = \nabla_2 f(x^k, y^k),\  \E\left[\norm{u_y^{k+1} - \nabla_2 f(x^k, y^k)}^2 \middle|\setF_k\right]\leq \sigma_{f,2}^2,\\
    &\E\left[v^{k+1}\middle|\setF_k\right] = \nabla_2 g(x^k, y^k),\ \E\left[\norm{v^{k+1} - \nabla_2 g(x^k, y^k)}^2\middle|\setF_k \right]\leq \sigma_{g,1}^2, \\
    &\E\left[H^{k+1}\middle|\setF_k\right] = \nabla_{22}^2 g(x^k,y^k),\ \E\left[\norm{H^{k+1} - \nabla_{22}^2g(x^k,y^k)}^2\middle|\setF_k\right]\leq \sigma_{g,2}^2,\\
    &\E\left[J^{k+1}\middle|\setF_k\right] = \nabla_{12}^2g(x^k, y^k),\  \E\left[\norm{J^{k+1} - \nabla_{12}^2g(x^k, y^k)}^2\middle|\setF_k\right] \leq \sigma_{g,2}^2.
\end{align*}
In addition, they are conditionally independent conditioned on $\setF_k$.\\

Next we state some technical lemmas that will be used in both sections.
\begin{lemma}\label{lem: gd_decrease}
    Suppose $f(x)$ is $\mu$-strongly convex and $L$-smooth. For any $x$ and $\gamma<\frac{2}{\mu + L}$, define $x^+ = x - \gamma\nabla f(x),\ x^*=\argmin f(x)$. Then we have
    \[
        \|x^+ - x^*\| \leq (1-\gamma\mu)\|x-x^*\|
    \]
\end{lemma}
\begin{proof}
    See, e.g., Lemma 10 in \cite{qu2017harnessing}.
\end{proof}


\begin{lemma}\label{lem: hypergrad}
    Define 
    \[
        \kappa = \max\left(\frac{L_{\nabla f} }{\mu_g}, \frac{L_{\nabla g}}{\mu_g}\right),\ z^*(x) = \left(\nabla_{22}^2g(x, y^*(x))\right)^{-1}\nabla_2f(x, y^*(x)).
    \]
    Suppose Assumption \ref{aspt: smoothness} holds. Then $\Phi(x)$ is differentiable and $\nabla\Phi(x)$ is given by
    Then $\Phi(x), y^*(x), z^*(x)$ are differentiable and $\nabla \Phi(x), y^*(x), z^*(x)$ are $L_{\nabla\Phi}, L_{y^*}, L_{z^*}$-Lipschitz continuous respectively, with their expressions as
    \begin{align}
        &\nabla \Phi(x) = \nabla_1 f(x, y^*(x)) - \nabla_{12}^2 g(x, y^*(x))\left(\nabla_{22}^2g(x, y^*(x))\right)^{-1}\nabla_2 f(x, y^*(x)), \label{eq: hypergrad} \\
        &\nabla y^*(x) = -\nabla_{12}^2 g(x, y^*(x))\left(\nabla_{22}^2g(x, y^*(x))\right)^{-1}. \label{eq: nabla_y}
    \end{align}
    and the constants given by
    \begin{align*}
    L_{y^*} &= \frac{L_{\nabla g}}{\mu_g} = \cO\left(\kappa\right),\  L_{z^*} = \sqrt{1 + L_{y^*}^2}\left(\frac{L_{\nabla f}}{\mu_g} + \frac{L_fL_{\nabla_{22}^2g}}{\mu_g^2}\right) = \cO\left(\kappa^3\right), \\
        L_{\nabla \Phi} &= L_{\nabla f} + \frac{2L_{\nabla f}L_{\nabla g} + L_{f}^2L_{\nabla^2 g}}{\mu_g} + \frac{2L_{f}L_{\nabla g}L_{\nabla^2 g}+L_{\nabla f}L_{\nabla g}^2}{\mu_g^2} + \frac{L_{f}L_{\nabla^2 g}L_{\nabla g}^2}{\mu_g^3} = \cO\left(\kappa^3\right).
    \end{align*}
    Moreover, we have
    \begin{equation}\label{ineq: z_star_bound}
         \norm{z^*(x)}\leq \frac{L_f}{\mu_g}.
    \end{equation}
\end{lemma}
\begin{proof}
    See Lemma 2.2 in \cite{ghadimi2018approximation} for the proof of \eqref{eq: hypergrad}, Lipschitz continuity of $\nabla\Phi$ and $y^*$. For the Lipschitz continuity of $z^*$ we have for any $x, \tilde x$, we know
    \begin{align*}
        &\norm{z^*(x)- z^*(\tilde x)} \\
        = &\norm{\left(\nabla_{22}^2g(x, y^*(x))\right)^{-1}\nabla_2f(x, y^*(x)) - \left(\nabla_{22}^2g(\tilde x, y^*(\tilde x))\right)^{-1}\nabla_2f(\tilde x, y^*(\tilde x))} \\
        \leq & \norm{\left(\nabla_{22}^2g(x, y^*(x))\right)^{-1}\nabla_2f(x, y^*(x)) - \left(\nabla_{22}^2g(\tilde x, y^*(\tilde x))\right)^{-1}\nabla_2f(x, y^*(x))} \\
        & + \norm{\left(\nabla_{22}^2g(\tilde x, y^*(\tilde x))\right)^{-1}\nabla_2f(x, y^*(x)) - \left(\nabla_{22}^2g(\tilde x, y^*(\tilde x))\right)^{-1}\nabla_2f(\tilde x, y^*(\tilde x))} \\
        \leq &L_f\norm{\left(\nabla_{22}^2g(x,y^*(x))\right)^{-1}}\norm{\nabla_{22}^2g(x,y^*(x)) - \nabla_{22}^2g(\tilde x,y^*(\tilde x))}\norm{\left(\nabla_{22}^2g(x,y^*(x))\right)^{-1}} \\
        &+ \frac{1}{\mu_g}\norm{\nabla_2f(x, y^*(x)) - \nabla_2f(\tilde x, y^*(\tilde x))} \\
        \leq &\frac{L_f L_{\nabla_{22}^2g}}{\mu_g^2}\sqrt{\norm{x - \tilde x}^2 + \norm{y^*(x) - y^*(\tilde x)}^2} + \frac{L_{\nabla f}}{\mu_g}\sqrt{\norm{x - \tilde x}^2 + \norm{y^*(x) - y^*(\tilde x)}^2} \\
        \leq &L_{z^*}\norm{x - \tilde x},
    \end{align*}
    where the first inequality uses triangle inequality, the second and third inequalities use Assumption \ref{aspt: smoothness}, and the fourth inequality uses the Lipschitz continuity of $y^*(x)$. The inequality in \eqref{ineq: z_star_bound} holds since $g(x,\cdot)$ is $\mu_g$-strongly convex and $\norm{\nabla_2f(x, y^*(x))}\leq L_f$ (see Assumption \ref{aspt: smoothness}).
\end{proof}

\begin{lemma}\label{lem: eta_property}
    For any convex compact set $\cX$, function $\eta_{\cX}(x,h,\tau)$ defined in Section \ref{sec: bo_analysis} is differentiable 
    and $\nabla \eta_{\cX}$ is $L_{\nabla\eta_{\cX}}$-Lipschitz continuous, with the closed form exression and constant given by
    \begin{equation*}
        \nabla_1\eta_{\cX}(x,h,\tau) = -h + \frac{1}{\tau}(x - \bar d),\ \nabla_2\eta_{\cX}(x,h,\tau) = \bar d - x,\  L_{\nabla\eta_{\cX}} = 2 \sqrt{\left(1+\frac{1}{\tau} \right)^2 + \left(1 + \frac{\tau}{2}\right)^2}. \label{eq: nabla_eta}
    \end{equation*}
    where $\bar d$ is defined as
    \[
        \bar d = \argmin_{d\in \cX}\left\{\<h, d - x> + \frac{1}{2\tau}\norm{d-x}^2\right\} = \Pi_{\cX}(x - \tau h),
    \]
    which satisfies the optimality condition
    \begin{equation}\label{ineq: eta_opt_condition}
        \<h + \frac{1}{\tau}(\bar d - x), d - \bar d>\geq 0,\ \text{ for all } d\in \cX.
    \end{equation}
\end{lemma}
\begin{proof}
    See Lemma 3.2 in \cite{ghadimi2020single}.
\end{proof}

\subsection{Proof of Theorem \ref{thm: soba_convergence}}\label{sec: thm1proof}
For simplicity, we summarize the notations that will be used in Section \ref{sec: thm1proof} as follows. 
\begin{align}\label{append: notations_simple}
    \begin{aligned}
        &\kappa = \max\left(\frac{L_{\nabla g}}{\mu_g},\ \frac{L_{\nabla f}}{\mu_g}\right),\ w^{k+1} =u_x^{k+1} - J^{k+1}z^k, \\
        &y_{*}^k = y^*(x^k) = \argmin_{y\in \realset^{d_y}} g(x^k, y),\ z_{*}^k  =  \left(\nabla_{22}^2g(x^k, y_{*}^k)\right)^{-1}\nabla_2f(x^k, y_{*}^k), \\
        &\Phi(x) = f(x, y^*(x)),\ \eta_{\cX}(x, h, \tau) = \min_{d\in X}\left\{\<h, d - x> + \frac{1}{2\tau}\norm{d-x}^2\right\}.
    \end{aligned}
\end{align}
In this section we suppose Assumptions \ref{aspt: smoothness} and \ref{aspt: stochastic_grad} hold.
We assume stepsizes in Algorithm \ref{alg:ma-soba} satisfy
\begin{equation}\label{eq: stepsize_ratio_simple}
    \beta_k = c_1\alpha_k,\ \gamma_k = c_2\alpha_k, \theta_k = c_3\alpha_k,
\end{equation}
where $c_1, c_2, c_3 > 0$ are constants to be determined. We will utilize the following merit function in our analysis:
\begin{align}\label{eq: merit_W}
    \begin{aligned}
        &W_k = W_{k,1} + W_{k,2},\\
        & W_{k,1} = \Phi(x^k) -\inf_{x\in\cX} \Phi(x) - \frac{1}{c_3}\eta_{\cX}(x^k,h^k,\tau) \\
        &W_{k,2} = \frac{1}{c_1}\norm{y^k - y_*^k}^2 + \frac{1}{c_2}\norm{z^k - z_*^k}^2.
    \end{aligned}
\end{align}
By definition of $\eta_{\cX}$, we can verify that $W_{k,1}\geq 0$. Moreover, as discussed in Section \ref{sec: bo_analysis}, we consider the following optimality measure:
\begin{equation}\label{eq: V_k_masoba}
    V_k = \frac{1}{\tau^2}\|x_+^k - x^k\|^2 + \|h^k - \nabla \Phi(x^k)\|^2.
\end{equation}
The following Lemma characterizes the relation between $V_k$ and gradient mapping of problem \ref{eq: bo}.
\begin{lemma}\label{lem: x_convergence}
    Suppose Assumptions \ref{aspt: smoothness} and \ref{aspt: stochastic_grad} hold. In Algorithm \ref{alg:ma-soba} we have
    \begin{align*}
        \frac{1}{\tau^2}\norm{x^k - \Pi_{\cX}\left(x^k - \tau\nabla\Phi(x^k)\right) }^2\leq 2V_k.
    \end{align*}
\end{lemma}
\begin{proof}
    Note that we have
    \begin{align*}
        &\norm{x^k - \Pi_{\cX}\left(x^k - \tau\nabla\Phi(x^k)\right)}^2\\
        \leq &2\left(\norm{x_+^k - x^k}^2 + \norm{\Pi_{\cX}\left(x^k - \tau h^k\right) - \Pi_{\cX}\left(x^k - \tau\nabla\Phi(x^k)\right)}^2\right) \\
        \leq &2\left(\norm{x_+^k - x^k}^2 + \tau^2\norm{h^k - \nabla\Phi(x^k) }^2\right) = 2V_k,
    \end{align*}
    where the first inequality uses Cauchy-Schwarz inequality and the second inequality uses the non-expansiveness of projection onto a convex compact set. This completes the proof.
\end{proof}

Next we present a technical lemma about the variance of $w^{k+1}$ and the bound for $\norm{h^{k+1} - h^k}$.
\begin{lemma}\label{lem: variance_simple}
    Suppose Assumptions \ref{aspt: smoothness} and \ref{aspt: stochastic_grad} hold. In Algorithm \ref{alg:ma-soba} we have
    \begin{align}
        &\E\left[\norm{w^{k+1} - \E\left[w^{k+1}\middle|\setF_k\right]}^2\right] \leq \sigma_{w,k+1}^2 \notag\\
        &\sigma_{w,k+1}^2 := \sigma_{w}^2 + 2\sigma_{g,2}^2\E\left[\norm{z^k - z_*^k}^2\right],\ \sigma_w^2 = \sigma_{f,1}^2 + \frac{2\sigma_{g,2}^2L_f^2}{\mu_g^2}, \label{ineq: w_var_simple}\\
        &\E\left[\norm{h^{k+1} - h^k}^2\right]\leq \sigma_{h, k}^2,\notag \\
        &\sigma_{h, k}^2 := 2\theta_k^2\E\left[\norm{h^k - \nabla\Phi(x^k)}^2 + \norm{\E\left[w^{k+1}\middle|\setF_k\right] -\nabla\Phi(x^k)}^2\right]+\theta_k^2\sigma_{w,k+1}^2 \label{ineq: h_var_simple}
    \end{align}
\end{lemma}
\begin{proof}
We first consider $w^k$. Note that
\[
    w^{k+1} - \E\left[w^{k+1}\middle|\setF_k\right] = u_x^{k+1} - \E\left[u_x^{k+1}\middle|\setF_k\right] - \left(J^{k+1} - \E\left[J^{k+1}\middle|\setF_k\right]\right)z^k.
\]
Hence we know 
\begin{align*}
    &\E\left[\norm{w^{k+1} - \E\left[w^{k+1}\middle|\setF_k\right]}^2\middle|\setF_k\right] \\
        = &\E\left[\norm{u_x^{k+1} - \E\left[u_x^{k+1}\middle|\setF_k\right]}^2\middle|\setF_k\right] + 
        \E\left[\norm{J^{k+1} - \E\left[J^{k+1}\middle|\setF_k\right]}^2\middle|\setF_k\right]\norm{z^k}^2\\
        \leq &\sigma_{f,1}^2 + 2\sigma_{g,2}^2\norm{z_*^k}^2 + 2\sigma_{g,2}^2\norm{z^k - z_*^k}^2 \leq \sigma_{f,1}^2 + \frac{2\sigma_{g,2}^2L_f^2}{\mu_g^2} + 2\sigma_{g,2}^2\norm{z^k - z_*^k}^2,
\end{align*}
where the first equality uses independence, the first inequality uses Cauchy-Schwarz inequality, and the second inequality uses \eqref{ineq: z_star_bound}. This proves \eqref{ineq: w_var_simple}. Next for $\norm{h^{k+1} - h^k}$ we have
\begin{align*}
    &\E\left[\norm{h^{k+1} - h^k}^2\middle|\setF_k\right] \\
    = &\theta_k^2\E\left[\norm{h^k - \E\left[w^{k+1}\middle|\setF_k\right]}^2\middle|\setF_k\right] + \theta_k^2\E\left[\norm{w^{k+1} - \E\left[w^{k+1}\middle|\setF_k\right] }^2\middle|\setF_k\right] \\
    \leq &2\theta_k^2\E\left[\norm{h^k - \nabla\Phi(x^k)}^2\middle|\setF_k\right] + 2\theta_k^2\E\left[\norm{\E\left[w^{k+1}\middle|\setF_k\right] -\nabla\Phi(x^k) }^2 \middle|\setF_k\right] + \theta_k^2\sigma_{w,k+1}^2,
\end{align*}
which proves of \eqref{ineq: h_var_simple} by taking expectation on both sides.
\end{proof}
\begin{remark}
    We would like to highlight that in \eqref{ineq: w_var_simple}, we explicitly characterize the upper bound of the variance of $w^{k+1}$, which contains $\E\left[\norm{z^k -z_*^k}^2\right]$ and requires further analysis. In contrast, Assumption 3.7 in \cite{dagreou2022framework} directly assumes the second moment of $D_x^t$ is uniformly bounded, i.e.,
    \[
        \E\left[\norm{D_x^t}^2\right]\leq B_x^2 \text{ for some constant } B_x \geq 0,
    \]
    Note that $D_x^t$ in \cite{dagreou2022framework} is the same as our $w^{k+1}$ (see \eqref{eq: x update mean}, line 5 of Algorithm \ref{alg:ma-soba} and definition of $w^{k+1}$ in \eqref{append: notations_simple}). The second moment bound can directly imply the variance bound, i.e.,
    \[
        \E\left[\norm{D_x^t - \E\left[D_x^t\right]}^2\right]\leq \E\left[\norm{D_x^t}^2\right]\leq B_x^2.
    \]
    This implies that some stronger assumptions are needed to guarantee Assumption 3.7 in \cite{dagreou2022framework}, as also pointed out by the authors (see discussions right below it). Instead, our refined analysis does not require that.
\end{remark}

\subsubsection{Hypergradient Estimation Error}
Note that Assumptions 3.1 and 3.2 in \cite{dagreou2022framework} state that the upper-level function $f$ is twice differentiable, the lower-level function $g$ is three times differentiable and $\nabla^2f, \nabla^3g$ are Lipschitz continuous so that $z_*^k$, as a function of $x^k$ (see \eqref{append: notations_simple}), is smooth, which is a crucial condition for (63) - (67) in \cite{dagreou2022framework}, which follows the analysis in Equation (49) in \cite{chen2021closing}. In this section we show that, by incorporating the moving average technique recently introduced to decentralized bilevel optimization \cite{chen2022decentralizeddsbo}, we can remove this additional assumption. We have the following lemma characterizing the error induced by $y^k$ and $z^k$.
\begin{lemma}\label{lem: yz_error}
    Suppose Assumptions \ref{aspt: smoothness} and \ref{aspt: stochastic_grad} hold. If the stepsizes satisfy
    \begin{equation}\label{ineq: beta_gamma_condition_simple}
        \beta_k < \frac{2}{\mu_g + L_{\nabla g}},\ \gamma_k \leq \min\left(\frac{1}{4\mu_g},\ \frac{0.06\mu_g}{\sigma_{g,2}^2}\right)
    \end{equation}
    then in Algorithm \ref{alg:ma-soba} we have
    \begin{align}\label{ineq: yz_decrease_simple}
        \begin{aligned}
            &\sum_{k=0}^{K}\alpha_k\E\left[\norm{y^k - y_*^k}^2\right]\leq C_{yx}\sum_{k=0}^{K}\alpha_k \E\left[\norm{x_+^k - x^k}^2\right] + C_{y,0} + C_{y,1}\left(\sum_{k=0}^{K}\alpha_k^2\right) \\
        &\sum_{k=0}^{K}\alpha_k\E\left[\norm{z^k-z_*^k}^2\right] \leq C_{zx}\sum_{k=0}^{K}\alpha_k \E\left[\norm{x_+^k - x^k}^2\right] + C_{z,0} + C_{z,1}\left(\sum_{k=0}^{K}\alpha_k^2\right).
        \end{aligned}
    \end{align}
    where the constants are defined as
    \begin{align*}
         &C_{yx} = \frac{2L_{y^*}^2}{c_1^2\mu_g^2},\ C_{y,0} = \frac{1}{c_1\mu_g}\E\left[\norm{y^0 - y_*^0}^2\right],\ C_{y,1}=\frac{2c_1\sigma_{g,1}^2}{\mu_g}, \\
         &C_{zx} = \frac{5L_f^2}{\mu_g^2}\left(\frac{L_{\nabla_{22}^2g}^2}{\mu_g^2} + 1\right)\frac{2L_{y^*}^2}{c_1^2\mu_g^2} + \frac{4L_{z^*}^2}{c_2^2\mu_g^2}\\
         &C_{z,0} = \frac{5L_f^2}{\mu_g^2}\left(\frac{L_{\nabla_{22}^2g}^2}{\mu_g^2} + 1\right)\cdot\frac{1}{c_1\mu_g}\E\left[\norm{y^0 - y_*^0}^2\right] + \frac{1}{c_2\mu_g}\E\left[\norm{z^0 - z_*^0}^2\right]\\
         &C_{z,1} = \frac{5L_f^2}{\mu_g^2}\left(\frac{L_{\nabla_{22}^2g}^2}{\mu_g^2} + 1\right)\cdot\frac{2c_1\sigma_{g,1}^2}{\mu_g} + \frac{2c_2\sigma_w^2}{\mu_g}.
    \end{align*}
\end{lemma}

\begin{proof}
    We first consider the error induced by $y^k$. We have
\begin{equation}\label{ineq: y_decrease_1_simple}
    \begin{aligned}
        \norm{y^{k+1} - y_*^{k+1}}^2 &\leq \left(1 + \beta_k\mu_g\right)\norm{y^{k+1} - y_*^k}^2 + \left(1 + \frac{1}{\beta_k\mu_g}\right)\norm{y_*^{k+1} - y_*^k}^2 \\
        &\leq \left(1 + \beta_k\mu_g\right)\norm{y^{k+1} - y_*^k}^2 + \left(\frac{\alpha_k^2}{\beta_k\mu_g} + \alpha_k^2\right)L_{y^*}^2\norm{x_+^k - x^k}^2,
    \end{aligned}
\end{equation}
where the first inequality uses Cauchy-Schwarz inequality:
\[
    \norm{u + v}^2 \leq (1 + c)\left(\norm{u}^2 + \frac{1}{c}\norm{v}^2\right), \text{ for any vectors } u, v \text{ and constant }c>0.
\]
Thanks to the moving average step of $x^k$, our analysis of $\norm{y_*^{k+1} - y_*^k}$ is simplified comparing to that in \cite{chen2021closing}. We also have
\begin{align}\label{ineq: y_decrease_2_simple}
    \begin{aligned}
        \E\left[\norm{y^{k+1} - y_*^k}^2\middle|\setF_k\right] = &\E\left[\norm{y^k - \beta_k\nabla_2g(x^k, y^k) - y_*^k - \beta_k(v^{k+1} - \nabla_2g(x^k, y^k))}^2\middle|\setF_k\right] \\
        \leq &\norm{y^k - \beta_k\nabla_2g(x^k, y^k) - y_*^k}^2 + \beta_k^2\sigma_{g,1}^2\\
        \leq &(1-\beta_k\mu_g)^2\norm{y^k - y_*^k}^2 + \beta_k^2\sigma_{g,1}^2
    \end{aligned}
\end{align}
where the first inequality uses Assumption \eqref{aspt: stochastic_grad} and Lemma \ref{lem: gd_decrease}, and the second inequality uses Lemma \ref{lem: gd_decrease} (which requires strong convexity of $g$, Lipschtiz continuity of $\nabla_2g$, and the first inequality in \eqref{ineq: beta_gamma_condition_simple}). Combining \eqref{ineq: y_decrease_1_simple} and \eqref{ineq: y_decrease_2_simple}, we know
\begin{align}\label{ineq: y_decrease_3_simple}
    \begin{aligned}
        &\E\left[\norm{y^{k+1} - y_*^{k+1}}^2\middle|\setF_k\right]\\
        \leq &\left(1 + \beta_k\mu_g\right)(1-\beta_k\mu_g)^2\norm{y^k - y_*^k}^2 + \left(\frac{\alpha_k^2}{\beta_k\mu_g} + \alpha_k^2\right)L_{y^*}^2\norm{x_+^k - x^k}^2 + \left(1 + \beta_k\mu_g\right)\beta_k^2\sigma_{g,1}^2 \\
        \leq & (1- \beta_k\mu_g)\norm{y^k - y_*^k}^2 + \frac{2\alpha_k^2L_{y^*}^2}{\beta_k\mu_g}\norm{x_+^k - x^k}^2 + 2\beta_k^2\sigma_{g,1}^2.
    \end{aligned}
\end{align}
where the second inequality uses $\beta_k<\frac{2}{\mu_g + L_{\nabla g}}\leq \frac{1}{\mu_g}$. Taking summation ($k$ from $0$ to $K$) on both sides and taking expectation, we know
\begin{align*}
    \sum_{k=0}^{K}\beta_k\mu_g\E\left[\norm{y^k - y_*^k}^2\right]\leq\E\left[\norm{y^0 - y_*^0}^2\right] + \sum_{k=0}^{K}\frac{2\alpha_k^2L_{y^*}^2}{\beta_k\mu_g}\E\left[\norm{x_+^k - x^k}^2\right] + \sum_{k=0}^{K}2\beta_k^2\sigma_{g,1}^2,
\end{align*}
which proves the first inequality in \eqref{ineq: yz_decrease_simple} by dividing $c_1\mu_g$ on both sides. 
Next we analyze the error induced by $z^k$. Our analysis is substantially different from \cite{dagreou2022framework}. We first notice that
\begin{align}\label{ineq: z_decrease_1_simple}
    \begin{aligned}
        \norm{z^{k+1} - z_*^{k+1}}^2 &\leq \left(1 + \frac{\gamma_k\mu_g}{3}\right)\norm{z^{k+1} - z_*^k}^2 + \left(1 + \frac{3}{\gamma_k\mu_g}\right)\norm{z_*^{k+1} - z_*^k}^2 \\
    &\leq \left(1 + \frac{\gamma_k\mu_g}{3}\right)\norm{z^{k+1} - z_*^k}^2 + \left(\frac{3\alpha_k^2}{\gamma_k\mu_g} + \alpha_k^2 \right)L_{z^*}^2\norm{x_+^k - x^k}^2 
    \end{aligned}
\end{align}
where we use Cauchy-Schwarz inequality in the first and second inequality, we use the facts that $\nabla y^*$ is Lipschitz continuous. 
For $\norm{z^{k+1} - z_*^k}$, we may follow the analysis of SGD under the strongly convex setting:
\begin{align*}
    z^{k+1} - z_*^k = z^k - \gamma_k(H^kz^k - u_{y}^k) & - z_*^k= z^k - \gamma_k\nabla_{22}^2g(x^k, y^k)z^k + \gamma_k\nabla_2 f(x^k, y^k) - z_*^k\\
    &- \gamma_k(H^{k+1} - \nabla_{22}^2g(x^k, y^k))z^k + \gamma_k(u_{y}^k - \nabla_2 f(x^k, y^k))
\end{align*}
which gives
\begin{align}
    &\E\left[\norm{z^{k+1} - z_*^k}^2\middle|\setF_k\right] \notag\\
    \leq &\norm{z^k - \gamma_k\nabla_{22}^2g(x^k, y^k)z^k + \gamma_k\nabla_2 f(x^k, y^k) - z_*^k}^2 + \gamma_k^2\sigma_{g,2}^2\norm{z^k}^2 + \gamma_k^2\sigma_{f,1}^2 \notag\\
    = & \norm{(I - \gamma_k\nabla_{22}^2g(x^k, y^k))(z^k - z_*^k) - \gamma_k(\nabla_{22}^2g(x^k,y^k)z_*^k - \nabla_2 f(x^k, y^k))}^2 + \gamma_k^2\sigma_{g,2}^2\norm{z^k}^2 + \gamma_k^2\sigma_{f,1}^2 \notag\\
    \leq & \left(1 + \frac{\gamma_k\mu_g}{2}\right)\norm{(I - \gamma_k\nabla_{22}^2g(x^k, y^k))(z^k - z_*^k)}^2 \notag\\
    &+ \left(1 + \frac{2}{\gamma_k\mu_g}\right)\norm{\gamma_k\left(\nabla_{22}^2g(x^k,y^k)z_*^k - \nabla_{22}^2g(x^k, y_*^k)z_*^k + \nabla_2 f(x^k, y_*^k) - \nabla_2 f(x^k, y^k)\right)}^2 \notag\\
     & + 2\gamma_k^2\sigma_{g,2}^2\left(\norm{z^k - z_*^k}^2 + \norm{z_*^k}^2\right)+ \gamma_k^2\sigma_{f,1}^2 \notag\\
    \leq & \left(\left(1 + \frac{\gamma_k\mu_g}{2}\right)(1-\gamma_k\mu_g)^2 + 2\gamma_k^2\sigma_{g,2}^2\right)\norm{z^k-z_*^k}^2 \notag\\
     & + \left(\frac{4\gamma_k}{\mu_g} + 2\gamma_k^2\right)\left(L_{\nabla_{22}^2g}^2\norm{z_*^k}^2 + L_{\nabla_2 f}^2\right)\norm{y^k - y_*^k}^2 + 2\gamma_k^2\sigma_{g,2}^2\norm{z_*^k}^2 + \gamma_k^2\sigma_{f,1}^2. \notag\\
    \leq & \left(1 - \frac{4\gamma_k\mu_g}{3}\right)\norm{z^k-z_*^k}^2 + \left(\frac{4\gamma_k}{\mu_g} + 2\gamma_k^2\right)\left(\frac{L_{\nabla_{22}^2g}^2L_f^2}{\mu_g^2} + L_f ^2\right)\norm{y^k - y_*^k}^2 + \left(\frac{2\sigma_{g,2}^2L_f^2}{\mu_g^2} + \sigma_{f,1}^2\right)\gamma_k^2, \label{ineq: z_decrease_2_simple}
\end{align}
where the first inequality uses Assumption \ref{aspt: stochastic_grad}, the second inequality uses Cauchy-Schwarz inequality and the definition of $z_*^k$, the third inequality uses Cauchy-Schwarz inequality and the fact that $g$ is $\mu_g$-strongly convex, and the fourth inequality uses Cauchy-Schwarz inequality, \eqref{ineq: z_star_bound} and 
\[
    -\frac{\gamma_k\mu_g}{6} + 2\gamma_k^2\sigma_{g,2}^2 + \frac{\gamma_k^3\mu_g^3}{2}\leq 0,
\]
which is a direct result from the bound of $\gamma_k$ in \eqref{ineq: beta_gamma_condition_simple}. It is worth noting that our estimation can be viewed as a refined version of (72) - (75) in \cite{dagreou2022framework}
Combining \eqref{ineq: z_decrease_1_simple} and \eqref{ineq: z_decrease_2_simple} we may obtain
{\small
\begin{align*}
    &\E\left[\norm{z^{k+1} - z_*^{k+1}}^2\middle|\setF_k\right] \\
    \leq &\left(1 + \frac{\gamma_k\mu_g}{3}\right)\E\left[\norm{z^{k+1} - z_*^k}^2\middle|\setF_k\right] + \left(\frac{3\alpha_k^2}{\gamma_k\mu_g} + \alpha_k^2 \right)L_{z^*}^2\norm{x_+^k - x^k}^2\\
    \leq &\left(1 + \frac{\gamma_k\mu_g}{3}\right)\left[\left(1 - \frac{4\gamma_k\mu_g}{3}\right)\norm{z^k-z_*^k}^2 + \left(\frac{4\gamma_k}{\mu_g} + 2\gamma_k^2\right)\left(\frac{L_{\nabla_{22}^2g}^2L_f^2}{\mu_g^2} + L_f^2\right)\norm{y^k - y_*^k}^2\right] \\
     & + \left(1 + \frac{\gamma_k\mu_g}{3}\right)\left(\frac{2\sigma_{g,2}^2L_f^2}{\mu_g^2} + \sigma_{f,1}^2\right)\gamma_k^2 + \left(\frac{3\alpha_k^2}{\gamma_k\mu_g} + \alpha_k^2 \right)L_{z^*}^2\norm{x_+^k - x^k}^2 \\
    = &\left(1 - \gamma_k\mu_g\right)\norm{z^k-z_*^k}^2 + \left(\frac{4\gamma_k}{\mu_g} + \frac{10\gamma_k^2}{3} + \frac{2\gamma_k^3\mu_g}{3}\right)\left(\frac{L_{\nabla_{22}^2g}^2L_f^2}{\mu_g^2} + L_f^2\right)\norm{y^k - y_*^k}^2 \\
     & + \sigma_w^2\left(\gamma_k^2  + \frac{\gamma_k^3\mu_g}{3}\right) + \left(\frac{3\alpha_k^2}{\gamma_k\mu_g} + \alpha_k^2 \right)L_{z^*}^2\norm{x_+^k - x^k}^2 \\
    \leq &\left(1 - \gamma_k\mu_g\right)\norm{z^k-z_*^k}^2 + \frac{5\gamma_kL_f^2}{\mu_g}\left(\frac{L_{\nabla_{22}^2g}^2}{\mu_g^2} + 1\right)\norm{y^k - y_*^k}^2 + 2\sigma_w^2\gamma_k^2 + \frac{4\alpha_k^2L_{z^*}^2}{\gamma_k\mu_g}\norm{x_+^k - x^k}^2,
\end{align*}
}%
where the equality uses the definition of $\sigma_w^2$ in \eqref{ineq: w_var_simple} and the third inequality uses $\gamma_k\mu_g \leq \frac{1}{4}$. Taking summation ($k$ from $0$ to $K$) and expectation, we know 
\begin{align*}
    \sum_{k=0}^{K}\gamma_k\mu_g\E\left[\norm{z^k-z_*^k}^2\right] \leq &\E\left[\norm{z^0 - z_*^0}^2\right]+ \sum_{k=0}^{K}\frac{5\gamma_kL_f^2}{\mu_g}\left(\frac{L_{\nabla_{22}^2g}^2}{\mu_g^2} + 1\right) \E\left[\norm{y^k - y_*^k}^2\right] \\
        + &\sum_{k=0}^{K}2\sigma_w^2\gamma_k^2 + \sum_{k=0}^{K}\frac{4\alpha_k^2L_{z^*}^2}{\gamma_k\mu_g}\E\left[\norm{x_+^k - x^k}^2\right].
\end{align*}
This completes the proof of the second inequality in \eqref{ineq: yz_decrease_simple} by dividing $c_2\mu_g$ on both sides and replacing $\sum_{k=0}^{K}\alpha_k\E\left[\norm{y^k - y_*^k}^2\right]$ with its upper bound in \eqref{ineq: yz_decrease_simple}.
\end{proof}

\begin{lemma}\label{lem: hypergrad_error_simple}
    Suppose Assumptions \ref{aspt: smoothness} and \ref{aspt: stochastic_grad} hold. We have
    \begin{align*}
        \norm{\E\left[w^{k+1}\middle|\setF_k\right] - \nabla\Phi(x^k)}^2\leq &3 \left(\left(L_{\nabla f}^2 + L_{\nabla^2 g}^2\right)\norm{y^k - y_*^k}^2 + L_{\nabla g}^2\norm{z^k - z_*^k}^2\right),
    \end{align*}
\end{lemma}

\begin{proof}
Note that we have the following decomposition:
\begin{align}\label{eq: w_exp_decompose_simple}
    \begin{aligned}
        &\E\left[w^{k+1}\middle|\setF_k\right] - \nabla\Phi(x^k) \\
        = &\E\left[u_x^{k+1}\middle|\setF_k\right] - \nabla_1 f(x^k, y_*^k) -  \left(\E\left[J^{k+1}\middle|\setF_k\right]z^k - \nabla_{12}^2g(x^k,y_*^k)z_*^k\right) \\
        = &\nabla_1 f(x^k, y^k) - \nabla_1 f(x^k, y_*^k) -  \nabla_{12}^2g(x^k, y^k)\left( z^k - z_*^k\right) - \left(\nabla_{12}^2g(x^k, y^k) - \nabla_{12}^2g(x^k,y_*^k)\right)z_*^k.
    \end{aligned}
\end{align}
which, together with Cauchy-Schwarz inequality, implies
\begin{align*}
    &\norm{\E\left[w^{k+1}\middle|\setF_k\right] - \nabla\Phi(x^k)}^2 \\
    \leq &3\norm{\nabla_1 f(x^k, y^k) - \nabla_1 f(x^k, y_*^k)}^2 + 3\norm{\nabla_{12}^2g(x^k, y^k)\left( z^k - z_*^k\right)}^2 + 3\norm{\left(\nabla_{12}^2g(x^k, y^k) - \nabla_{12}^2g(x^k,y_*^k)\right)z_*^k}^2\\
    \leq &3 \left(\left(L_{\nabla f}^2 + L_{\nabla^2 g}^2\right)\norm{y^k - y_*^k}^2 + L_{\nabla g}^2\norm{z^k - z_*^k}^2\right).
\end{align*}
\end{proof}

\subsubsection{Primal Convergence}
\begin{lemma}\label{lem: primal_simple}
    Suppose Assumptions \ref{aspt: smoothness} and \ref{aspt: stochastic_grad} hold. If
    \begin{equation}\label{ineq: conditions_primal_simple}
        \alpha_k\leq \min\left(\frac{\tau^2}{20c_3},\ \frac{c_3}{2\tau\left(c_3L_{\nabla\Phi} + L_{\nabla\eta_{\cX}}\right)},\ 1\right),\ \tau < 1,\ c_3 \leq \frac{1}{10},
    \end{equation}
    then in Algorithm \ref{alg:ma-soba} we have
    \begin{align}\label{ineq: primal_decrease_simple}
        \begin{aligned}
            \sum_{k=0}^{K}\frac{\alpha_k}{\tau^2}\E\left[\norm{x_+^k-x^k}^2\right]\leq &\frac{2}{\tau}\E\left[W_{0,1}\right] + 3\sum_{k=0}^{K}\alpha_k\E\left[\norm{\nabla\Phi(x^k) - \E\left[w^{k+1}\middle|\setF_k\right]}^2\right] \\
            &+ \frac{1}{2}\sum_{k=0}^{K}\alpha_k\E\left[\norm{h^k - \nabla\Phi(x^k)}^2\right] + \sum_{k=0}^{K}\left(\alpha_k^2\sigma_{g,2}^2\E\left[\norm{z^k - z_*^k}^2\right] + \alpha_k^2\sigma_w^2\right),
        \end{aligned}
    \end{align}
\end{lemma}

\begin{proof}
    The $L_{\nabla \Phi}$-smoothness of $\Phi(x)$ and $L_{\nabla\eta_{\cX}}$-smoothness of $\eta_{\cX}$ in Lemma \ref{lem: hypergrad} and \ref{lem: eta_property} imply
\begin{align}\label{ineq: phi_decrease_simple}
    \begin{aligned}
        \Phi(x^{k+1}) - \Phi(x^k)\leq\alpha_k\<\nabla \Phi(x^k), x_+^k - x^k> + \frac{L_{\nabla \Phi}}{2}\norm{x^{k+1}-x^k}^2
    \end{aligned}
\end{align}
and
\begin{equation}\label{ineq: eta_x_decrease_simple}
    \begin{aligned}
        &\eta_{\cX}(x^k,h^k,\tau) - \eta_{\cX}(x^{k+1},h^{k+1},\tau)\\
        \leq & \<-h^k + \frac{1}{\tau}(x^k - x_+^k), x^k - x^{k+1}> + \<x_+^k - x^k, h^k - h^{k+1}> \\
        &+ \frac{L_{\nabla\eta_{\cX}}}{2}\left(\norm{x^{k+1}-x^k}^2 + \norm{h^{k+1} - h^k}^2\right)\\
        = &\alpha_k\<h^k, x_+^k - x^k> + \frac{\alpha_k}{\tau}\norm{x_+^k - x^k}^2 +\theta_k\<h^k, x_+^k - x^k> - \theta_k\<w^{k+1}, x_+^k - x^k> \\
         & + \frac{L_{\nabla\eta_{\cX}}}{2}\left(\norm{x^{k+1}-x^k}^2 + \norm{h^{k+1} - h^k}^2\right)\\
    \leq &- \frac{\theta_k}{\tau}\norm{x_+^k - x^k}^2 - \theta_k\<w^{k+1}, x_+^k-x^k> + \frac{L_{\nabla\eta_{\cX}}}{2}\left(\norm{x^{k+1}-x^k}^2 + \norm{h^{k+1} - h^k}^2\right),
    \end{aligned}
\end{equation}
where the first inequality uses $L_{\nabla\eta_{\cX}}$-smoothness of $\nabla \eta_{\cX}$, and the second inequality uses the optimality condition \eqref{ineq: eta_opt_condition} (with $d = x^k$). Hence by computing $\eqref{ineq: phi_decrease_simple} + \eqref{ineq: eta_x_decrease_simple}/ c_3$ and taking conditional expectation with respect to $\setF_k$ we know
\begin{align}
    &\frac{\alpha_k}{\tau}\norm{x_+^k-x^k}^2\notag \\
    \leq &\frac{1}{c_3}\left(\E\left[\eta_{\cX}(x^{k+1},h^{k+1},\tau)\middle|\setF_k\right] - \eta_{\cX}(x^k,h^k,\tau)\right) + \Phi(x^k) - \E\left[\Phi(x^{k+1})\middle|\setF_k\right] \notag\\
     &+ \alpha_k\<\nabla\Phi(x^k) - \E\left[w^{k+1}\middle|\setF_k\right], x_+^k - x^k> + \frac{\left(c_3L_{\nabla \Phi} + L_{\nabla\eta_{\cX}}\right)}{2c_3}\norm{x^{k+1}-x^k}^2 \notag\\
     &+ \frac{L_{\nabla\eta_{\cX}}}{2c_3}\E\left[\norm{h^{k+1} - h^k}^2\middle|\setF_k\right]. \notag\\
    = & W_{k,1} - \E\left[W_{k+1,1}\middle|\setF_k\right] + \alpha_k\<\nabla\Phi(x^k) - \E\left[w^{k+1}\middle|\setF_k\right], x_+^k - x^k> \notag\\
     & + \frac{\left(c_3L_{\nabla \Phi} + L_{\nabla\eta_{\cX}}\right)}{2c_3}\norm{x^{k+1}-x^k}^2 + \frac{L_{\nabla\eta_{\cX}}}{2c_3}\E\left[\norm{h^{k+1} - h^k}^2\middle|\setF_k\right] \notag\\
    \leq & W_{k,1} - \E\left[W_{k+1,1}\middle|\setF_k\right] + \alpha_k\left(\tau\norm{\nabla\Phi(x^k) - \E\left[w^{k+1}\middle|\setF_k\right]}^2 + \frac{1}{4\tau}\norm{ x_+^k - x^k}^2\right) \notag\\
     & + \frac{\alpha_k}{4\tau}\norm{x_+^k-x^k}^2 + \frac{5}{2c_3\tau}\E\left[\norm{h^{k+1} - h^k}^2\middle|\setF_k\right], \label{ineq: primal_x_simple}
\end{align}
where the second inequality uses Young's inequality and the following inequalities:
\[
    \frac{\alpha_k^2\left(c_3L_{\nabla \Phi} + L_{\nabla\eta_{\cX}}\right)}{2c_3}\leq \frac{\alpha_k}{4\tau},\  L_{\nabla \eta_{\cX}}<\frac{5}{\tau} \text{ when }\eqref{ineq: conditions_primal_simple}\ \text{holds}.
\]
Note that by \eqref{ineq: h_var_simple} we know 
\begin{align}\label{ineq: h_term_simple}
    \begin{aligned}
        &\frac{5}{c_3\tau^2}\E\left[\norm{h^{k+1} - h^k}^2\right] \\
        \leq &\frac{10c_3\alpha_k^2}{\tau^2}\E\left[\norm{h^k - \nabla\Phi(x^k)}^2 + \norm{\E\left[w^{k+1}\middle|\setF_k\right] -\nabla\Phi(x^k)}^2\right] \\
         & + \frac{5c_3\alpha_k^2}{\tau^2}\sigma_w^2  + \frac{10c_3\alpha_k^2\sigma_{g,2}^2}{\tau^2}\E\left[\norm{z^k - z_*^k}^2\right]. \\
        \leq &\frac{\alpha_k}{2}\E\left[\norm{h^k - \nabla\Phi(x^k)}^2\right] + \alpha_k\E\left[\norm{\E\left[w^{k+1}\middle|\setF_k\right] -\nabla\Phi(x^k)}^2\right]\\
        & + \alpha_k^2\sigma_w^2 + \alpha_k^2\sigma_{g,2}^2\E\left[\norm{z^k - z_*^k}^2\right],
    \end{aligned}
\end{align}
where the second inequality uses \eqref{ineq: conditions_primal_simple}. Taking summation and expectation on both sides of \eqref{ineq: primal_x_simple} and using \eqref{ineq: h_term_simple}, we obtain \eqref{ineq: primal_decrease_simple}

\end{proof}

\subsubsection{Dual Convergence}
\begin{lemma}\label{lem: x_dual_simple}
    Suppose Assumptions \ref{aspt: smoothness} and \ref{aspt: stochastic_grad} hold. In Algorithm \ref{alg:ma-soba} we have
    \begin{align}\label{ineq: dual_decrease_simple}
        \begin{aligned}
        \sum_{k=0}^{K}\alpha_k\E\left[\norm{h^k - \nabla\Phi(x^k)}^2\right]\leq &\frac{1}{c_3}\E\left[\norm{h^0 - \nabla\Phi(x^0)}^2\right] + 2\sum_{k=0}^{K}\alpha_k\E\left[\norm{\E\left[w^{k+1}\middle|\setF_k\right] - \nabla\Phi(x^k)}^2\right] \\
        & + \frac{2L_{\nabla \Phi}^2}{c_3^2}\sum_{k=0}^{K}\alpha_k\E\left[\norm{x_+^k - x^k}^2\right] + 2c_3\sigma_{g,2}^2\sum_{k=0}^{K}\alpha_k^2\E\left[\norm{z^k - z_*^k}^2\right] + \sum_{k=0}^{K}c_3\alpha_k^2\sigma_w^2.
        \end{aligned}
    \end{align}
\end{lemma}

\begin{proof}
    Note that by moving average update of $h^k$, we have
\begin{align}
    \begin{aligned}
        &h^{k+1} - \nabla\Phi(x^{k+1}) \\
        = &(1 - \theta_k)h^k + \theta_k(w^{k+1} - \E\left[w^{k+1}\middle|\setF_k\right]) + \theta_k\E\left[w^{k+1}\middle|\setF_k\right] - \nabla\Phi(x^{k+1})\\
        = &(1 - \theta_k)(h^k - \nabla\Phi(x^k)) + \theta_k(\E\left[w^{k+1}\middle|\setF_k\right] - \nabla\Phi(x^k)) + \nabla\Phi(x^k) - \nabla\Phi(x^{k+1}) \\
        & +  \theta_k(w^{k+1} - \E\left[w^{k+1}\middle|\setF_k\right])
    \end{aligned}
\end{align}
Hence we know 
\begin{align}
        &\E\left[\norm{h^{k+1} - \nabla\Phi(x^{k+1})}^2\middle|\setF_k\right] \notag\\
        = &\norm{(1 - \theta_k)(h^k - \nabla\Phi(x^k)) + \theta_k(\E\left[w^{k+1}\middle|\setF_k\right] - \nabla\Phi(x^k)) + \nabla\Phi(x^k) - \nabla\Phi(x^{k+1})}^2 \notag\\
         &+ \theta_k^2\E\left[\norm{w^{k+1} - \E\left[w^{k+1}\middle|\setF_k\right]}^2\middle|\setF_k\right] \notag\\ 
        \leq &(1 - \theta_k)\norm{h^k - \nabla\Phi(x^k)}^2 \notag\\
        & + \theta_k\norm{\E\left[w^{k+1}\middle|\setF_k\right] - \nabla\Phi(x^k) + \frac{1}{\theta_k}(\nabla\Phi(x^k) - \nabla\Phi(x^{k+1}))}^2 +\theta_k^2\sigma_{w,k+1}^2 \notag\\
        \leq &(1 - \theta_k)\norm{h^k - \nabla\Phi(x^k)}^2 \notag\\
        & + 2\theta_k\norm{\E\left[w^{k+1}\middle|\setF_k\right] - \nabla\Phi(x^k)}^2 + \frac{2}{\theta_k}\norm{\nabla\Phi(x^k) - \nabla\Phi(x^{k+1})}^2 + \theta_k^2\sigma_{w,k+1}^2 \notag\\
        \leq &(1 - \theta_k)\norm{h^k - \nabla\Phi(x^k)}^2 \notag\\
        & + 2\theta_k\norm{\E\left[w^{k+1}\middle|\setF_k\right] - \nabla\Phi(x^k)}^2 + \frac{2\alpha_k^2L_{\nabla \Phi}^2}{\theta_k}\norm{x_+^k - x^k}^2 + \theta_k^2\sigma_{w,k+1}^2,\label{ineq: h_decrease_simple}
\end{align}
where the first equality uses the fact that $x^k, h^k, x^{k+1}, $ are all $\setF_k$-measurable and are independent of $w^{k+1}$ given $\setF_k$, the first inequality uses the convexity of $\norm{\cdot}^2$ and \eqref{ineq: w_var_simple}, the second inequality uses Cauchy-Schwarz inequality, the third inequality uses the Lipschitz continuity of $\nabla\Phi$ in Lemma \ref{lem: smoothness}, and the update rules of $x^{k+1}$. Taking summation, expectation on both sides of \eqref{ineq: h_decrease_simple}, dividing $c_3$ and using \eqref{ineq: w_var_simple}, we know \eqref{ineq: dual_decrease_simple} holds.
\end{proof}

\subsubsection{Proof of Theorem \ref{thm: soba_convergence}}
Now we are ready to prove Theorem \ref{thm: soba_convergence}. From Lemma \ref{lem: x_convergence} we know it suffices to bound $V_k$. By definition of $V_k$ in \eqref{eq: V_k_masoba}, \eqref{ineq: primal_decrease_simple} and \eqref{ineq: dual_decrease_simple} we have
\begin{align*}
    &\sum_{k=0}^{K}\alpha_k\E\left[V_k\right] = \sum_{k=0}^{K}\left(\frac{\alpha_k}{\tau^2}\E\left[\norm{x_+^k-x^k}^2\right] + \alpha_k\E\left[\norm{h^k - \nabla\Phi(x^k)}^2\right]\right) \\
        \leq &\frac{2L_{\nabla \Phi}^2}{c_3^2}\sum_{k=0}^{K}\alpha_k\E\left[\norm{x_+^k - x^k}^2\right]  +\frac{1}{2}\sum_{k=0}^{K}\alpha_k\E\left[\norm{h^k - \nabla\Phi(x^k)}^2\right] \\
        & + 5\sum_{k=0}^{K}\alpha_k\E\left[\norm{\nabla\Phi(x^k) - \E\left[w^{k+1}\middle|\setF_k\right]}^2\right] +  (1 + 2c_3)\sigma_{g,2}^2\sum_{k=0}^{K}\alpha_k^2\E\left[\norm{z^k - z_*^k}^2\right]\\
        &+ \frac{2}{\tau}\E\left[W_{0,1}\right] + \frac{1}{c_3}\E\left[\norm{h^0 - \nabla\Phi(x^0)}^2\right] + (1+c_3)\sigma_w^2\left(\sum_{k=0}^{K}\alpha_k^2\right), \\
        \leq &\frac{2L_{\nabla \Phi}^2}{c_3^2}\sum_{k=0}^{K}\alpha_k\E\left[\norm{x_+^k - x^k}^2\right]  +\frac{1}{2}\sum_{k=0}^{K}\alpha_k\E\left[\norm{h^k - \nabla\Phi(x^k)}^2\right]\\
         & + 15\sum_{k=0}^{K}\alpha_k\E\left[\left(L_{\nabla f}^2 + L_{\nabla^2 g}^2\right)\norm{y^k - y_*^k}^2 + L_{\nabla g}^2\norm{z^k - z_*^k}^2\right]  + L_{\nabla g}^2\sum_{k=0}^{K}\alpha_k\E\left[\norm{z^k - z_*^k}^2\right] \\
         &+ \frac{2}{\tau}\E\left[W_{0,1}\right] + \frac{1}{c_3}\E\left[\norm{h^0 - \nabla\Phi(x^0)}^2\right] + (1+c_3)\sigma_w^2\left(\sum_{k=0}^{K}\alpha_k^2\right) \\
            \leq & C_{vx}\sum_{k=0}^{K}\alpha_k \E\left[\norm{x_+^k - x^k}^2\right] + C_{vh}\sum_{k=0}^{K}\alpha_k\E\left[\norm{h^k - \nabla\Phi(x^k)}^2\right] + C_{v,0} + C_{v,1}\left(\sum_{k=0}^{K}\alpha_k^2\right),
\end{align*}
where we assume
\begin{equation}\label{ineq: alpha_condition_simple}
    (1 + 2c_3)\sigma_{g,2}^2\alpha_k\leq L_{\nabla g}^2,\ 
\end{equation}
in the second inequality. The constants are defined as
\begin{align}
    \begin{aligned}
        &C_{vx} = 15\left(L_{\nabla f}^2 + L_{\nabla^2 g}^2\right)C_{yx} + 16L_{\nabla g}^2C_{zx} + \frac{2L_{\nabla \Phi}^2}{c_3^2} ,\ C_{vh} = \frac{1}{2}, \\
        &C_{v,0} = 15\left(L_{\nabla f}^2 + L_{\nabla^2 g}^2\right)C_{y,0} + 16L_{\nabla g}^2C_{z,0} + \frac{2}{\tau}\E\left[W_{0,1}\right] + \frac{1}{c_3}\E\left[\norm{h^0 - \nabla\Phi(x^0)}^2\right],\\
        &C_{v,1} = 15\left(L_{\nabla f}^2 + L_{\nabla^2 g}^2\right)C_{y,1} + 16L_{\nabla g}^2C_{z,1} + (1+c_3)\sigma_w^2.
    \end{aligned}
\end{align}
Using constants defined in Lemma \ref{lem: yz_error}, we know
\begin{align*}
    &C_{vx}= \cO\left(\frac{\kappa^8}{c_1^2} + \frac{\kappa^4}{c_2^2} + \frac{\kappa^6}{c_3^2}\right),\ C_{vh} = \cO(1),\\
    &C_{v,0} = \cO\left(\frac{\kappa^5}{c_1} + \frac{\kappa^2}{c_2} + \frac{1}{\tau}\right),\ C_{v,1} = \cO\left(c_1\kappa^5 + c_2\kappa^2\right).
\end{align*}
Hence we can pick 
\[
    \alpha_k \equiv \Theta\left(\frac{1}{\sqrt{K}}\right),\ \tau = \Theta\left(\kappa^{-4}\right),\ c_1 = \cO(1),\ c_2 = \cO(1),\ c_3 = \cO(1)
\]
so that the conditions (\eqref{ineq: beta_gamma_condition_simple}, \eqref{ineq: conditions_primal_simple} and \eqref{ineq: alpha_condition_simple}) in previous lemmas hold. Then we have
\[
    \frac{1}{K}\sum_{k=0}^{K}\E\left[V_k\right] = \cO\left(\frac{\kappa^5}{\sqrt{K}}\right).
\]
which, together with Lemma \ref{lem: x_convergence}, proves Theorem \ref{thm: soba_convergence}.

\subsection{Proof of Theorem \ref{thm: momasoba_convergence}}\label{sec: thm2proof}
In this section we present our proof of Theorem \ref{thm: momasoba_convergence}. For simplicity, we summarize the notations that will be used in our proof as follows.
\begin{align}\label{append: notations}
    \begin{aligned}
    &L_{\nabla f} = \max_{1\leq i\leq n}L_{\nabla f_i},\ L_{\nabla g} = \max_{1\leq i\leq n}L_{\nabla g_i},\ L_{\nabla^2 g_i} = \max_{1\leq i\leq n}L_{\nabla^2 g_i},\ \mu_g = \max_{1\leq i\leq n}\mu_{g_i}, \\
        &\kappa = \max\left(\frac{L_{\nabla g}}{\mu_g},\ \frac{L_{\nabla f}}{\mu_g}\right),\ u_x^{k+1} = \sum_{i=1}^{n}u_{x,i}^{k+1},\ w^{k+1} = \sum_{i=1}^{n}\lambda_i^k\left(u_{x,i}^{k+1} - J_i^{k+1}z_i^k\right), \\
        &\lambda_*(x) = \argmax_{\lambda \in \Delta_n}\Phi_{\mu_{\lambda}}(x,\lambda),\ \lambda_*^k = \lambda_*(x^k), \\
        &y_{*,i}^k = y_i^*(x^k) = \argmin_{y\in \realset^{d_y}} g_i(x^k, y),\ z_{*,i}^k =  \left(\nabla_{22}^2g_i(x^k, y_{*,i}^k)\right)^{-1}\nabla_2f_i(x^k, y_{*,i}^k), \\
        &\Phi_i(x) = f_i(x, y_i^*(x)),\ \Phi^k = \left(\Phi_1(x^k), ..., \Phi_n(x^k)\right)\T, \\
        &\Psi(x) = \max_{\lambda\in \Delta_n} \Phi_{\mu_{\lambda}}(x,\lambda) = \max_{\lambda\in\Delta_n}\left(\sum_{i=1}^{n}\lambda_i\Phi_i(x) -\frac{\mu_{\lambda}}{2}\norm{\lambda - \frac{\bfone_n}{n}}^2\right), \\
        &\eta_{X}(x, h, \tau) = \min_{d\in X}\left\{\<h, d - x> + \frac{1}{2\tau}\norm{d-x}^2\right\}, \text{ where } X = \cX \text{ or } \Delta_n.
    \end{aligned}
\end{align}
In this subsection we suppose Assumptions \ref{aspt: smoothness}, \ref{aspt: stochastic_grad} hold for all $f_i, g_i$ and Assumption \ref{aspt: f_bound} holds. We suppose stepsizes in Algorithm \ref{algo: momasoba} satisfy
\begin{equation}\label{eq: stepsize_ratio}
    \beta_k = c_1\alpha_k,\ \gamma_k = c_2\alpha_k, \theta_k = c_3\alpha_k,
\end{equation}
where $c_1, c_2, c_3 > 0$ are constants to be determined. We will utilize the following merit function in our analysis:
\begin{align}\label{eq: merit_tilde_W}
    \begin{aligned}
        &\tilde W_k = \tilde W_{k,1} + \tilde W_{k,2},\ \tilde W_{k,1} = \tilde W_{k,1}^{(1)} + \tilde W_{k,1}^{(2)}\\
        & \tilde W_{k,1}^{(1)} = \Psi(x^k) - \Phi_{\mu_{\lambda}}(x^k,\lambda^k) - \frac{1}{c_3}\eta_{\Delta_n}(\lambda^k,-h_{\lambda}^k,\tau_{\lambda})\\
        & \tilde W_{k,1}^{(2)} = \Psi(x^k) -\inf_{x\in\cX} \Psi(x) - \frac{1}{c_3}\eta_{\cX}(x^k,h_x^k,\tau_x) \\
        &\tilde W_{k,2} = \sum_{i=1}^{n}\left(\frac{1}{c_1}\norm{y_i^k - y_{*,i}^k}^2 + \frac{1}{c_2}\norm{z_i^k - z_{*,i}^k}^2\right).
    \end{aligned}
\end{align}
By definition of $\Psi, \eta_{\cX}, \eta_{\Delta_n}$, we can verify that $\tilde W_{k,1}^{(1)}\geq 0, \tilde W_{k,1}^{(2)}\geq 0$. Moreover, as discussed in Section \ref{sec: convergence_momasoba}, we consider the following optimality measure:
\begin{align}\label{eq: V_k_momasoba}
    \begin{aligned}
    &\tilde V_{k,1} = \frac{1}{\tau_x^2}\norm{x_+^k - x^k}^2 + \norm{h_x^k - \nabla_1\Phi_{\mu_{\lambda}}(x^k,\lambda^k)}^2,\\
        &\tilde V_{k,2} = \frac{1}{\tau_{\lambda}^2}\norm{\lambda_+^k  - \lambda^k}^2 + \norm{h_{\lambda}^k - \nabla_2\Phi_{\mu_{\lambda}}(x^k, \lambda^k)}^2, \\
        &\tilde V_k = \underbrace{\frac{1}{\tau_x^2}\norm{x_+^k - x^k}^2 + \norm{h_x^k - \nabla_1\Phi_{\mu_{\lambda}}(x^k,\lambda^k)}^2}_{\text{Optimality of min problem}} + \underbrace{\frac{1}{\tau_{\lambda}^2}\norm{\lambda_+^k  - \lambda^k}^2 + \norm{h_{\lambda}^k - \nabla_2\Phi_{\mu_{\lambda}}(x^k, \lambda^k)}^2}_{\text{Optimality of max problem}}.
    \end{aligned}
\end{align}
The following lemma provides some smoothness of functions that we will use in our proof.
\begin{lemma}\label{lem: smoothness}
    Functions $\nabla\Psi(\cdot), \nabla_1\Phi_{\mu_{\lambda}}(\cdot, \lambda), \nabla_1\Phi(\cdot, \lambda), \nabla_1\Phi_{\mu_{\lambda}}(x,\cdot), \nabla_1\Phi(x,\cdot), \nabla_2\Phi_{\mu_{\lambda}}(\cdot, \lambda), \\ \nabla_2\Phi_{\mu_{\lambda}}(x, \cdot)$ are $L_{\nabla \Psi}, L_{\nabla\Phi}, L_{\nabla\Phi}, L_{\nabla_1\Phi_{\mu_{\lambda}}}, L_{\nabla_1\Phi_{\mu_{\lambda}}}, L_{\nabla_2\Phi_{\mu_{\lambda}}}, \mu_{\lambda}$-Lipschitz continuous respectively, with the constants given by
    \[
        L_{\nabla \Psi} = \frac{n}{\mu_{\lambda}}\left(L_{\Phi}^2 + b_{\Phi}L_{\nabla\Phi}\right) + L_{\nabla\Phi},\ L_{\nabla_1\Phi_{\mu_{\lambda}}} = L_{\nabla_2\Phi_{\mu_{\lambda}}} = \sqrt{n}L_{\Phi}
    \]
\end{lemma}
\begin{proof}
    For $\nabla\Psi$ we first notice that the nonconvex-strongly-concave problem in \eqref{eq: minmaxbo_reform_reg} can be reformulated as a bilevel problem:
    \[
        \min_{x\in \cX}\Psi(x) = \Phi_{\mu_{\lambda}}(x, \lambda^*(x))\ \ \text{s.t. }\ \lambda^*(x) = \argmin_{\lambda\in \Delta_n}\left(-\Phi_{\mu_{\lambda}}(x,\lambda)\right) = \frac{\mu_{\lambda}}{2}\norm{\lambda - \frac{\bfone_n}{n}}^2 - \sum_{i=1}^{n}\lambda_i\Phi_i(x).
    \]
    By Lemma \ref{lem: hypergrad} we know
    \begin{align*}
        \nabla \Psi(x) = &\nabla_1\Phi_{\mu_{\lambda}}(x,\lambda^*(x)) - \nabla_{12}^2\Phi_{\mu_{\lambda}}(x,\lambda^*(x))\left(\nabla_{22}^2\Phi_{\mu_{\lambda}}(x,\lambda^*(x))\right)^{-1}\nabla_2\Phi_{\mu_{\lambda}}(x,\lambda^*(x)) \\
        =&\sum_{i=1}^{n}\lambda_i^*(x)\nabla\Phi_i(x) + \frac{1}{\mu_{\lambda}} \left(\nabla\Phi_1(x), ..., \nabla\Phi_n(x)\right)\left[ \begin{pmatrix}
            \Phi_1(x) \\
            \vdots \\
            \Phi_n(x)
        \end{pmatrix} - \mu_{\lambda}\left(\lambda^*(x) - \frac{\bfone_n}{n}\right)\right] \\
        = & \frac{1}{\mu_{\lambda}}\sum_{i=1}^{n}\Phi_i(x)\nabla\Phi_i(x) + \frac{1}{n}\sum_{i=1}^{n}\nabla\Phi_i(x),
    \end{align*}
    from which we know $\nabla\Psi(\cdot)$ is $L_{\nabla\Psi}$-Lipschitz continuous since 
    \begin{align*}
        &\norm{\Phi_i(x)\nabla\Phi_i(x) - \Phi_i(\tilde x)\nabla\Phi_i(\tilde x)} \\
        \leq &\norm{\Phi_i(x)\nabla\Phi_i(x) - \Phi_i(x)\nabla\Phi_i(\tilde x)} + \norm{\Phi_i(x)\nabla\Phi_i(\tilde x) - \Phi_i(\tilde x)\nabla\Phi_i(\tilde x)} \\
        \leq &\left(L_{\Phi}^2 + b_{\Phi}L_{\nabla\Phi}\right)\norm{x-y}.
    \end{align*}
    Note that for any fixed $\lambda \in \Delta_n$ and $x, \tilde x\in\cX$, we have
    \begin{equation}\label{eq: phi_mu_and_phi}
        \nabla_1\Phi_{\mu_{\lambda}}(x,\lambda) = \nabla_1\Phi(x,\lambda) = \sum_{i=1}^{n}\lambda_i\nabla\Phi_i(x),
    \end{equation}
    and 
    \begin{equation}\label{ineq: nabla_1_phi_mu_lip}
        \norm{\nabla_1\Phi_{\mu_{\lambda}}(x,\lambda) - \nabla_1\Phi_{\mu_{\lambda}}(\tilde x, \lambda)} = \norm{\sum_{i=1}^{n}\lambda_i\left(\nabla\Phi_i(x) - \nabla\Phi_i(\tilde x)\right)}\leq L_{\nabla \Phi}\norm{x - \tilde x}.
    \end{equation}
    Similarly, for any fixed $x\in \cX$ and $\lambda, \tilde \lambda \in \Delta_n$ we know
    \begin{equation}\label{ineq: nabla_phi_mu_lip}
        \norm{\nabla_1\Phi_{\mu_{\lambda}}(x,\lambda) - \nabla_1\Phi_{\mu_{\lambda}}(x,\tilde \lambda)} = \norm{\sum_{i=1}^{n}\left(\lambda_i-\tilde \lambda_i\right)\nabla\Phi_i(x)}\leq \sqrt{n}L_{\Phi}\norm{\lambda - \tilde \lambda}.
    \end{equation}
    \eqref{eq: phi_mu_and_phi}, \eqref{ineq: nabla_1_phi_mu_lip} and \eqref{ineq: nabla_phi_mu_lip} imply $\nabla_1\Phi_{\mu_{\lambda}}(\cdot, \lambda), \nabla_1\Phi(\cdot, \lambda)$ are $L_{\nabla \Phi}$-Lipschitz continuous and $\nabla_1\Phi_{\mu_{\lambda}}(x,\cdot), \nabla_1\Phi(x,\cdot)$ are $L_{\nabla_1\Phi_{\mu_{\lambda}}}$-Lipschitz continuous. Finally, for $\nabla_2\Phi_{\mu_{\lambda}}(x,\lambda)$ we have
    \[
        \nabla_2\Phi_{\mu_{\lambda}}(x,\lambda) = \left(\Phi_1(x), ..., \Phi_n(x)\right)\T - \mu_{\lambda}\left(\lambda - \frac{\bfone_n}{n}\right),
    \]
    and thus $\nabla_2\Phi_{\mu_{\lambda}}(\cdot,\lambda), \nabla_2\Phi_{\mu_{\lambda}}(x,\cdot)$ are $\sqrt{n}L_{\Phi}, \mu_{\lambda}$-Lipschitz continuous respectively.
\end{proof}

Next we present a technical lemma that will be used in analyzing the strongly convex function over a convex compact set.
\begin{lemma}\label{lem: coercivity}
    Suppose $f(x)$ is $\mu$-strongly convex and $L$-smooth over a convex compact set $\cX$. For any $\tau \leq \frac{1}{L}$ define $x_+ = \Pi_{\cX}(x - \tau\nabla f(x))$ and $x_* = \argmin_{x\in\cX} f(x)$, we have
    \[
        \left(1 - \sqrt{1 - \tau\mu}\right)\norm{x - x_*}\leq \norm{x-x_+}.
    \]
\end{lemma}

\begin{proof}
    By Corollary 2.2.4 in \cite{nesterov2018lectures} we know
    \begin{align*}
        \frac{1}{\tau}\<x-x_+, x - x_*> &\geq \frac{1}{2\tau}\norm{x - x_+}^2 + \frac{\mu}{2}\norm{x - x_*}^2 + \frac{\mu}{2}\norm{x_+ - x_*}^2\\
        &= \left(\frac{1}{2\tau} + \frac{\mu}{2}\right)\norm{x - x_+}^2 + \mu\norm{x - x_*}^2 - \mu\<x - x_+, x - x_*>
    \end{align*}
    which implies
    \[
        \norm{x - x_+}\norm{x - x_*}\geq \<x-x_+, x - x_*> \geq \frac{1}{2}\norm{x - x_+}^2 + r\norm{x - x_*}^2
    \]
    where $r = \frac{\mu}{\frac{1}{\tau} + \mu}\leq \frac{1}{2}$. Applying Young's inequality to the left hand side of the above inequality, we know
    \[
        \frac{1 + \sqrt{1-2r}}{4r}\norm{x - x_+}^2 + \frac{r}{1 + \sqrt{1-2r}}\norm{x - x_*}^2\geq \frac{1}{2}\norm{x - x_+}^2 + r\norm{x - x_*}^2
    \]
    which gives
    \[
        \norm{x - x_+}\geq \left(1 - \sqrt{1-2r}\right)\norm{x - x_*}\geq \left(1 - \sqrt{1 - \tau\mu}\right)\norm{x - x_*}.
    \]
    This completes the proof.
\end{proof}
The next lemma shows the relation between the stationarity used in Theorem \ref{thm: momasoba_convergence} and our measure of optimality $\tilde V_k$ in \eqref{eq: V_k_momasoba}.
\begin{lemma}\label{lem: lambda_convergence}
    Suppose Assumptions \ref{aspt: smoothness}, \ref{aspt: stochastic_grad} hold for all $f_i, g_i$ and Assumption \ref{aspt: f_bound} holds. If $\tau_{\lambda}\mu_{\lambda} = 1$, then in Algorithm \ref{algo: momasoba} we have
    \begin{align*}
        &\frac{1}{\tau_x^2}\norm{x^k - \Pi_{\cX}\left(x^k - \tau_x\nabla_1\Phi_{\mu_{\lambda}}(x^k, \lambda^k)\right) }^2\leq 2\left(\frac{1}{\tau_x^2}\norm{x_+^k - x^k}^2 + \norm{h_x^k - \nabla_1\Phi_{\mu_{\lambda}}(x^k, \lambda^k) }^2\right),\\
            &\norm{\lambda^k - \lambda_*^k}^2\leq \frac{2}{\mu_{\lambda}^2}\left(\frac{1}{\tau_{\lambda}^2}\norm{\lambda_+^k - \lambda^k}^2 + \norm{h_{\lambda}^k - \nabla_2\Phi_{\mu_{\lambda}}(x^k, \lambda^k)}^2\right),
    \end{align*}
    which together imply
    \begin{align*}
        \norm{\frac{1}{\tau_x}\left(x^k - \Pi_{\cX}\left(x^k - \tau_x\nabla_1\Phi_{\mu_{\lambda}}(x^k, \lambda^k)\right) \right)}^2 + \norm{\lambda^k - \lambda_*^k}^2\leq\max\left(2, \frac{2}{\mu_{\lambda}^2}\right)\tilde V_k.
    \end{align*}
\end{lemma}

\begin{proof}
    The first inequality follows \eqref{lem: x_convergence}:
    \begin{align*}
        &\norm{x^k - \Pi_{\cX}\left(x^k - \tau_x\nabla_1\Phi_{\mu_{\lambda}}(x^k, \lambda^k)\right)}^2\\
        \leq &2\left(\norm{x_+^k - x^k}^2 + \norm{\Pi_{\cX}\left(x^k - \tau_x h_x^k\right) - \Pi_{\cX}\left(x^k - \tau_x\nabla_1\Phi_{\mu_{\lambda}}(x^k, \lambda^k)\right)}^2\right) \\
        \leq &2\left(\norm{x_+^k - x^k}^2 + \tau_x^2\norm{h_x^k - \nabla_1\Phi_{\mu_{\lambda}}(x^k, \lambda^k) }^2\right),
    \end{align*}
    where the first inequality uses Cauchy-Schwarz inequality and the second inequality uses the non-expansiveness of projection onto a convex compact set. Recall that 
    \[
        \lambda_*^k = \argmin_{\lambda\in\Delta_n} \Phi_{\mu_{\lambda}}(x^k, \lambda)
    \]
    which is a minimizer (over the probability simplex) of a $\mu_{\lambda}$-smooth and $\mu_{\lambda}$-strongly convex function $\Phi_{\mu_{\lambda}}(x^k, \cdot)$. Hence we know from Lemma \ref{lem: coercivity} that
    \begin{align*}
        &\mu_{\lambda}^2\norm{\lambda_*^k - \lambda^k}^2\\
        \leq &\frac{\left(1 + \sqrt{1-\tau_{\lambda}\mu_{\lambda}}\right)^2}{\tau_{\lambda}^2} \norm{\lambda^k - \Pi_{\Delta_n}\left(\lambda^k + \tau_{\lambda}\nabla_2\Phi_{\mu_{\lambda}}(x^k, \lambda^k)\right)}^2\\
        \leq &\frac{2\left(1 + \sqrt{1-\tau_{\lambda}\mu_{\lambda}}\right)^2}{\tau_{\lambda}^2}\left(\norm{\lambda_+^k - \lambda^k}^2 + \norm{\Pi_{\Delta_n}(\lambda^k + \tau_{\lambda}h_{\lambda}^k) - \Pi_{\Delta_n}\left(\lambda^k + \tau_{\lambda}\nabla_2\Phi_{\mu_{\lambda}}(x^k, \lambda^k)\right)}^2\right) \\
        \leq &\frac{2\left(1 + \sqrt{1-\tau_{\lambda}\mu_{\lambda}}\right)^2}{\tau_{\lambda}^2}\left(\norm{\lambda_+^k - \lambda^k}^2 + \tau_{\lambda}^2\norm{h_{\lambda}^k - \nabla_2\Phi_{\mu_{\lambda}}(x^k, \lambda^k)}^2\right),
    \end{align*}
    where the second inequality uses Cauchy-Schwarz inequality and the third inequality uses non-expansiveness of the projection onto a convex compact set. Setting $\tau_{\lambda}\mu_{\lambda}= 1$ completes the proof.
\end{proof}
\begin{lemma}\label{lem: variance}
    Suppose Assumptions \ref{aspt: smoothness}, \ref{aspt: stochastic_grad} hold for all $f_i, g_i$ and Assumption \ref{aspt: f_bound} holds. In Algorithm \ref{algo: momasoba} we have
    \begin{align}
        &\E\left[\norm{w^{k+1} - \E\left[w^{k+1}\middle|\setF_k\right]}^2\right] \leq \sigma_{w,k+1}^2 \notag\\
        &\sigma_{w,k+1}^2 := \sigma_w^2 + 2\sigma_{g,2}^2\E\left[\sum_{i=1}^{n}\lambda_i^k\norm{z_i^k - z_{*,i}^k}^2\right], \sigma_w^2 = \sigma_{f,1}^2 + \frac{2\sigma_{g,2}^2L_f^2}{\mu_g^2}\label{ineq: w_var}\\
        &\E\left[\norm{h_x^{k+1} - h_x^k}^2\right]\leq \sigma_{h_x, k}^2,\quad \E\left[\norm{h_{\lambda}^{k+1} - h_{\lambda}^k}^2\right]\leq \sigma_{h_{\lambda}, k}^2, \label{ineq: h_var} \\
        &\sigma_{h_x, k}^2 := 2\theta_k^2\E\left[\norm{h_x^k - \nabla_1\Phi_{\mu_{\lambda}}(x^k,\lambda^k)}^2 + \norm{\E\left[w^{k+1}\middle|\setF_k\right] -\nabla_1\Phi_{\mu_{\lambda}}(x^k, \lambda^k) }^2\right] +\theta_k^2\sigma_{w,k+1}^2 \notag\\
        &\sigma_{h_{\lambda}, k}^2 :=\theta_k^2\E\left[\norm{h_{\lambda}^k - \nabla_2\Phi_{\mu_{\lambda}}(x^k,\lambda^k)}^2\right] + n\theta_k^2\sigma_{f,0}^2.
    \end{align}
\end{lemma}
\begin{proof}
We first consider $w^k$. Note that
\[
    w^{k+1} - \E\left[w^{k+1}\middle|\setF_k\right] = \sum_{i=1}^{n}\lambda_i^k\left(u_{x,i}^{k+1} - \E\left[u_{x,i}^{k+1}\middle|\setF_k\right] - \left(J_i^{k+1} - \E\left[J_i^{k+1}\middle|\setF_k\right]\right)z_i^k\right).
\]
Hence we know 
\begin{align*}
    &\E\left[\norm{w^{k+1} - \E\left[w^{k+1}\middle|\setF_k\right]}^2\middle|\setF_k\right] \\
        = &\sum_{i=1}^{n}\left(\lambda_i^k\right)^2\left(\E\left[\norm{u_{x,i}^k - \E\left[u_{x,i}^k\middle|\setF_k\right]}^2\middle|\setF_k\right] + 
        \E\left[\norm{J_i^{k+1} - \E\left[J_i^{k+1}\middle|\setF_k\right]}^2\middle|\setF_k\right]\norm{z_i^k}^2  \right) \\
        \leq &\sum_{i=1}^{n}\lambda_i^k\left(\sigma_{f,1}^2 + 2\sigma_{g,2}^2\norm{z_{*,i}^k}^2 + 2\sigma_{g,2}^2\norm{z_i^k - z_{*,i}^k}^2\right) \\
        \leq & \sigma_{f,1}^2 + \frac{2\sigma_{g,2}^2L_f^2}{\mu_g^2} + 2\sigma_{g,2}^2\sum_{i=1}^{n}\lambda_i^k\norm{z_i^k - z_{*,i}^k}^2,
\end{align*}
which proves \eqref{ineq: w_var}. Next for $\norm{h_x^{k+1} - h_x^k}$ we have
\begin{align*}
    &\E\left[\norm{h_x^{k+1} - h_x^k}^2\middle|\setF_k\right] \\
    = &\theta_k^2\E\left[\norm{h_x^k - \E\left[w^{k+1}\middle|\setF_k\right]}^2\middle|\setF_k\right] + \theta_k^2\E\left[\norm{w^{k+1} - \E\left[w^{k+1}\middle|\setF_k\right] }^2\middle|\setF_k\right] \\
    \leq &2\theta_k^2\E\left[\norm{h_x^k - \nabla_1\Phi(x^k,\lambda^k)}^2\middle|\setF_k\right] + 2\theta_k^2\E\left[\norm{\E\left[w^{k+1}\middle|\setF_k\right] -\nabla_1\Phi(x^k, \lambda^k) }^2 \middle|\setF_k\right] + \theta_k^2\sigma_{w,k+1}^2,
\end{align*}
which proves the first inequality of \eqref{ineq: h_var}. Similarly we have
\begin{align*}
    &\E\left[\norm{h_{\lambda}^{k+1} - h_{\lambda}^k}^2\middle|\setF_k\right] \\
    = &\theta_k^2\E\left[\norm{h_{\lambda}^k - \E\left[s^{k+1}\middle|\setF_k\right] + \mu_{\lambda}\left(\lambda^k - \frac{\bfone_n}{n}\right)}^2\middle|\setF_k\right] + \theta_k^2\E\left[\norm{s^{k+1} - \E\left[s^{k+1}\middle|\setF_k\right]}^2\middle|\setF_k\right] \\
    \leq &\theta_k^2\E\left[\norm{h_{\lambda}^k - \nabla_2\Phi_{\mu_{\lambda}}(x^k,\lambda^k)}^2\middle|\setF_k\right] + n\theta_k^2\sigma_{f,0}^2,
\end{align*}
which proves the second inequality of \eqref{ineq: h_var}.
\end{proof}

\subsubsection{Hypergradient Estimation Error}
\begin{lemma}\label{lem: yz_error_minmax}
Suppose Assumptions \ref{aspt: smoothness}, \ref{aspt: stochastic_grad} hold for all $f_i, g_i$ and Assumption \ref{aspt: f_bound} holds. In Algorithm \ref{algo: momasoba} if the stepsizes satisfy
    \begin{equation}\label{ineq: beta_gamma_condition}
        \beta_k < \frac{2}{\mu_g + L_{\nabla g}},\ \gamma_k \leq \min\left(\frac{1}{4\mu_g},\ \frac{0.06\mu_g}{\sigma_{g,2}^2}\right)
    \end{equation}
    then we have
    \begin{align}
        \begin{aligned}
            &\sum_{k=0}^{K}\alpha_k\E\left[\sum_{i=1}^{n}\norm{y_i^k - y_{*,i}^k}^2\right] \leq nC_{yx}\sum_{k=0}^{K}\alpha_k\E\left[\norm{x_+^k - x^k}^2\right] + \sum_{i=1}^{n}C_{y_i, 0} +  nC_{y,1}\left(\sum_{k=0}^{K}\alpha_k^2\right) \\
        &\sum_{k=0}^{K}\alpha_k\E\left[\sum_{i=1}^{n}\norm{z_i^k-z_{*,i}^k}^2\right] \leq nC_{zx}\sum_{k=0}^{K}\alpha_k\E\left[\norm{x_+^k - x^k}^2\right] + \sum_{i=1}^{n}C_{z_i, 0} +  nC_{z,1}\left(\sum_{k=0}^{K}\alpha_k^2\right)
        \end{aligned}
    \end{align}
    where constants $C_{yx}, C_{y,1}, C_{zx}, C_{z,1}$ are defined the same as those in Lemma \ref{lem: yz_error}. $C_{y_i, 0}, C_{z_i, 0}$ are defined as
    \begin{align*}
        &C_{y_i, 0} = \frac{1}{c_1\mu_g}\E\left[\norm{y_i^0 - y_{*,i}^0}^2\right],\\
        &C_{z_i,0} = \frac{5L_f^2}{\mu_g^2}\left(\frac{L_{\nabla_{22}^2g}^2}{\mu_g^2} + 1\right)\cdot\frac{1}{c_1\mu_g}\E\left[\norm{y_i^0 - y_{*,i}^0}^2\right] + \frac{1}{c_2\mu_g}\E\left[\norm{z_i^0 - z_{*,i}^0}^2\right].
    \end{align*}
\end{lemma}
\begin{proof}
    Note that the proof follows almost the same reasoning in Lemma \ref{lem: yz_error}. Since Assumptions \ref{aspt: smoothness} and \ref{aspt: stochastic_grad} hold for all $f_i, g_i$, by replacing $y^k, y_*^k, z^k, z_*^k$ with $y_i^k, y_{*,i}^k, z_i^k, z_{*,i}^k$ respectively, we have similar results hold for each $1\leq i\leq n$
   \begin{align}\label{ineq: yz_decrease}
        \begin{aligned}
            &\sum_{k=0}^{K}\alpha_k\E\left[\norm{y_i^k - y_{*,i}^k}^2\right]\leq C_{yx}\sum_{k=0}^{K}\alpha_k \E\left[\norm{x_+^k - x^k}^2\right] + C_{y_i,0} + C_{y,1}\left(\sum_{k=0}^{K}\alpha_k^2\right) \\
        &\sum_{k=0}^{K}\alpha_k\E\left[\norm{z_i^k-z_{*,i}^k}^2\right] \leq C_{zx}\sum_{k=0}^{K}\alpha_k \E\left[\norm{x_+^k - x^k}^2\right] + C_{z_i,0} + C_{z,1}\left(\sum_{k=0}^{K}\alpha_k^2\right).
        \end{aligned}
    \end{align}
    Taking summation on both sides of \eqref{ineq: yz_decrease}, we complete the proof.
\end{proof}
The next lemma shows that the inequalities above will be used in the error analysis of \\
$\norm{\E\left[w^{k+1}\middle|\setF_k\right] - \nabla_1\Phi(x^k, \lambda^k)}$. 
\begin{lemma}\label{lem: hypergrad_error}
    Suppose Assumptions \ref{aspt: smoothness}, \ref{aspt: stochastic_grad} hold for all $f_i, g_i$ and Assumption \ref{aspt: f_bound} holds. We have
    \begin{align*}
        \norm{\E\left[w^{k+1}\middle|\setF_k\right] - \nabla_1\Phi_{\mu_{\lambda}}(x^k, \lambda^k)}^2\leq &\sum_{i=1}^{n}3\lambda_i^k \left\{\left(L_{\nabla f}^2 + L_{\nabla^2 g}^2\right)\norm{y_i^k - y_{*,i}^k}^2 + L_{\nabla g}^2\norm{z_i^k - z_{*,i}^k}^2\right\}, \\
        \norm{\E\left[w^{k+1}\middle|\setF_k\right] - \nabla\Psi(x^k)}^2 \leq &\sum_{i=1}^{n}4\lambda_i^k \left\{\left(L_{\nabla f}^2 + L_{\nabla^2 g}^2\right)\norm{y_i^k - y_{*,i}^k}^2 + L_{\nabla g}^2\norm{z_i^k - z_{*,i}^k}^2\right\} \\
        & + 8nL_{\Phi}^2\left\{\norm{\lambda_+^k - \lambda^k}^2 + \frac{1}{\mu_{\lambda}^2}\norm{h_{\lambda}^k - \nabla_2\Phi_{\mu_{\lambda}}(x^k, \lambda^k)}^2\right\}.
    \end{align*}
\end{lemma}
\begin{proof}
Note that we have the following decomposition:
\begin{align}\label{eq: w_exp_decompose}
    \begin{aligned}
        &\E\left[w^{k+1}\middle|\setF_k\right] - \nabla_1\Phi_{\mu_{\lambda}}(x^k, \lambda^k) \\
        = &\E\left[u_x^{k+1}\middle|\setF_k\right] - \sum_{i=1}^{n}\lambda_i^k\nabla_1 f_i(x^k, y_{*,i}^k) - \sum_{i=1}^{n}\lambda_i^k \left(\E\left[J_i^{k+1}\middle|\setF_k\right]z_i^k - \nabla_{12}^2g_i(x^k,y_{*,i}^k)z_{*,i}^k\right) \\
        = &\sum_{i=1}^{n}\lambda_i^k\bigg\{\nabla_1 f_i(x^k, y_i^k) - \nabla_1 f_i(x^k, y_{*,i}^k) -  \nabla_{12}^2g_i(x^k, y_i^k)\left( z_i^k - z_{*,i}^k\right)\\
        & \hspace{12em} - \left[\nabla_{12}^2g_i(x^k, y_i^k) - \nabla_{12}^2g_i(x^k,y_{*,i}^k)\right]z_{*,i}^k\bigg\}.
    \end{aligned}
\end{align}
which, together with Cauchy-Schwarz inequality, implies
\begin{align*}
    &\norm{\E\left[w^{k+1}\middle|\setF_k\right] - \nabla_1\Phi_{\mu_{\lambda}}(x^k, \lambda^k)}^2 \\
    \leq &3\norm{\sum_{i=1}^{n}\lambda_i^k \left(\nabla_1 f_i(x^k, y_i^k) - \nabla_1 f_i(x^k, y_{*,i}^k)\right)}^2 + 3\norm{\sum_{i=1}^{n}\lambda_i^k\nabla_{12}^2g_i(x^k, y_i^k)\left( z_i^k - z_{*,i}^k\right)}^2\\
    & + 3\norm{\sum_{i=1}^{n}\left(\nabla_{12}^2g_i(x^k, y_i^k) - \nabla_{12}^2g_i(x^k,y_{*,i}^k)\right)z_{*,i}^k}^2\\
    \leq &\sum_{i=1}^{n}3\lambda_i^k \left(\left(L_{\nabla f}^2 + L_{\nabla^2 g}^2\right)\norm{y_i^k - y_{*,i}^k}^2 + L_{\nabla g}^2\norm{z_i^k - z_{*,i}^k}^2\right).
\end{align*}
Similarly we have
\[
    \E\left[w^{k+1}\middle|\setF_k\right] - \nabla\Psi(x^k) = \E\left[w^{k+1}\middle|\setF_k\right] - \nabla_1\Phi_{\mu_{\lambda}}(x^k, \lambda^k) + \nabla_1\Phi_{\mu_{\lambda}}(x^k, \lambda^k) - \nabla_1\Phi_{\mu_{\lambda}}(x^k, \lambda_*^k).
\]
Applying Cauchy-Schwarz inequality, Assumption \ref{aspt: smoothness} and Lemma \ref{lem: smoothness} to the above equation and \eqref{eq: w_exp_decompose}, we know

\begin{align*}
    &\norm{\E\left[w^{k+1}\middle|\setF_k\right] - \nabla\Psi(x^k)}^2 \\
    \leq &4\norm{\sum_{i=1}^{n}\lambda_i^k \left(\nabla_1 f_i(x^k, y_i^k) - \nabla_1 f_i(x^k, y_{*,i}^k)\right)}^2 + 4\norm{\sum_{i=1}^{n}\lambda_i^k\nabla_{12}^2g_i(x^k, y_i^k)\left( z_i^k - z_{*,i}^k\right)}^2\\
    & + 4\norm{\sum_{i=1}^{n}\left(\nabla_{12}^2g_i(x^k, y_i^k) - \nabla_{12}^2g_i(x^k,y_{*,i}^k)\right)z_{*,i}^k}^2  + 4\norm{\nabla_1\Phi(x^k,\lambda^k) - \nabla_1\Phi(x^k,\lambda_*^k)}^2\\
    \leq &\sum_{i=1}^{n}4\lambda_i^k \left\{\left(L_{\nabla f}^2 + L_{\nabla^2 g}^2\right)\norm{y_i^k - y_{*,i}^k}^2 + L_{\nabla g}^2\norm{z_i^k - z_{*,i}^k}^2\right\} + 4nL_{\Phi}^2\norm{\lambda^k-\lambda_*^k}^2,
\end{align*}
which together with Lemma \ref{lem: lambda_convergence} completes the proof.
\end{proof}
\subsubsection{Primal Convergence}
\begin{lemma}\label{lem: primal}
    Suppose Assumptions \ref{aspt: smoothness}, \ref{aspt: stochastic_grad} hold for all $f_i, g_i$ and Assumption \ref{aspt: f_bound} holds. If
    \begin{align}
         &\alpha_k\leq \min\left(\frac{\tau_x^2}{20c_3},\ \frac{c_3}{2\tau_x\left(c_3L_{\nabla\Phi} + L_{\nabla\eta_{\cX}}\right)},\ \frac{c_3}{4\tau_{\lambda}(L_{\nabla\eta_{\Delta_n}} +c_3\mu_{\lambda})},\ \frac{n\tau_{\lambda}L_{\Phi}^2}{L_{\Psi} + L_{\nabla\Phi}},\ 1\right),\notag \\
        &\tau_x < 1,\ \tau_{\lambda} = \frac{1}{\mu_{\lambda}},\ c_3 \leq \min\left(\frac{1}{10},\ \frac{1}{8(\mu_{\lambda} + 1)^2}\right),\label{ineq: conditions_primal}
    \end{align}
    then in Algorithm \ref{algo: momasoba} we have
    \begin{align}
        &\sum_{k=0}^{K}\frac{\alpha_k}{\tau_x^2}\E\left[\norm{x_+^k-x^k}^2\right] \notag\\
            \leq &\frac{2}{\tau_x}\E\left[\tilde W_{0,1}^{(1)}\right] + 2\sum_{k=0}^{K}\alpha_k\E\left[\norm{\E\left[w^{k+1}\middle|\setF_k\right] - \nabla\Psi(x^k)}^2\right]  \notag\\
            & + \sum_{k=0}^{K}\alpha_k\E\left[\norm{\E\left[w^{k+1}\middle|\setF_k\right] -\nabla_1\Phi_{\mu_{\lambda}}(x^k,\lambda^k)}^2\right] +\frac{1}{2}\sum_{k=0}^{K}\alpha_k\E\left[\norm{h_x^k - \nabla_1\Phi_{\mu_{\lambda}}(x^k,\lambda^k)}^2\right] \notag\\
            & + \sigma_{g,2}^2\sum_{k=0}^{K}\alpha_k^2\E\left[\sum_{i=1}^{n}\lambda_i^k\norm{z_i^k - z_{*,i}^k}^2\right] + \sigma_w^2\sum_{k=0}^{K}\alpha_k^2, \notag\\
            &\sum_{k=0}^{K}\frac{\alpha_k}{\tau_{\lambda}^2}\E\left[\norm{\lambda_+^k - \lambda^k}^2\right] \notag\\
            \leq &\frac{2}{\tau_{\lambda}}\E\left[\tilde W_{0,1}^{(2)}\right] + \frac{1}{2}\sum_{k=0}^{K}\alpha_k\E\left[\norm{h_{\lambda}^k - \nabla_2\Phi_{\mu_{\lambda}}(x^k, \lambda^k)}^2\right] + 4L_f^2\sum_{k=0}^{K}\alpha_k\E\left[\sum_{i=1}^{n}\norm{y_i^k-y_{*,i}^k}^2\right] \notag \\
         & + 13nL_{\Phi}^2\sum_{k=0}^{K}\alpha_k\E\left[\norm{x_+^k - x^k}^2\right] + n\sigma_{f,0}^2\sum_{k=0}^{K}\alpha_k^2. \label{ineq: primal_decrease}
    \end{align}
\end{lemma}

\begin{proof}
The proof of the first inequality in \eqref{ineq: primal_decrease} is almost the same as that in \eqref{lem: primal_simple}. Note that by replacing $\Phi, h^k, W_{k,1}$ with $\Psi, h_x^k, \tilde W_{k,1}$, we know
\begin{align}\label{ineq: primal_x}
    \begin{aligned}
           &\frac{\alpha_k}{\tau_x}\norm{x_+^k-x^k}^2  \\
        \leq & \tilde W_{k,1}^{(1)} - \E\left[\tilde W_{k+1,1}^{(1)}\middle|\setF_k\right] + \alpha_k\left(\tau_x\norm{\nabla\Psi(x^k) - \E\left[w^{k+1}\middle|\setF_k\right]}^2 + \frac{1}{4\tau_x}\norm{ x_+^k - x^k}^2\right) \\
        & + \frac{\alpha_k}{4\tau_x}\norm{x_+^k-x^k}^2 + \frac{5}{2c_3\tau_x}\E\left[\norm{h_x^{k+1} - h_x^k}^2\middle|\setF_k\right],
    \end{aligned}
\end{align}
Similar to \eqref{ineq: h_term_simple}, from \eqref{ineq: h_var} we have that
\begin{align}\label{ineq: h_term}
    \begin{aligned}
        &\frac{5}{c_3\tau_x^2}\E\left[\norm{h_x^{k+1} - h_x^k}^2\right]  \\
        \leq &\frac{10c_3\alpha_k^2}{\tau_x^2}\E\left[\norm{h_x^k - \nabla_1\Phi_{\mu_{\lambda}}(x^k,\lambda^k)}^2 + \norm{\E\left[w^{k+1}\middle|\setF_k\right] -\nabla_1\Phi_{\mu_{\lambda}}(x^k,\lambda^k)}^2\right] + \frac{5c_3\alpha_k^2}{\tau_x^2}\sigma_w^2 \\
        & + \frac{10c_3\alpha_k^2\sigma_{g,2}^2}{\tau_x^2}\E\left[\sum_{i=1}^{n}\lambda_i^k\norm{z_i^k - z_{*,i}^k}^2\right]. \\
    \leq &\frac{\alpha_k}{2}\E\left[\norm{h_x^k - \nabla_1\Phi_{\mu_{\lambda}}(x^k,\lambda^k)}^2\right] + \alpha_k\E\left[\norm{\E\left[w^{k+1}\middle|\setF_k\right] -\nabla_1\Phi_{\mu_{\lambda}}(x^k,\lambda^k)}^2\right] + \alpha_k^2\sigma_w^2 \\
    & + \alpha_k^2\sigma_{g,2}^2\E\left[\sum_{i=1}^{n}\lambda_i^k\norm{z_i^k - z_{*,i}^k}^2\right], 
    \end{aligned}
\end{align}
where the second inequality uses \eqref{ineq: conditions_primal}. Taking summation and expectation on both sides of \eqref{ineq: primal_x} and using \eqref{ineq: h_term}, we obtain the first inequality in \eqref{ineq: primal_decrease}.
For the second inequality in \eqref{ineq: primal_decrease}, the $L_{\nabla \Psi}$-smoothness of $\Psi(x)$ and $L_{\nabla\eta_{\cX}}$-smoothness of $\eta_{\cX}$ in Lemma \ref{lem: smoothness} imply
\begin{align}\label{ineq: phi_decrease}
    \begin{aligned}
        \Psi(x^{k+1}) - \Psi(x^k)\leq\alpha_k\<\nabla \Psi(x^k), x_+^k - x^k> + \frac{L_{\nabla \Psi}}{2}\norm{x^{k+1}-x^k}^2,
    \end{aligned}
\end{align}
and
\begin{align}
        &\eta_{\Delta_n}(\lambda^k,-h_{\lambda}^k,\tau_{\lambda}) - \eta_{\Delta_n}(\lambda^{k+1},-h_{\lambda}^{k+1},\tau_{\lambda})\notag\\
        \leq & \<h_{\lambda}^k + \frac{1}{\tau_{\lambda}}(\lambda^k - \lambda_+^k), \lambda^k - \lambda^{k+1}> + \<\lambda_+^k - \lambda^k, -h_{\lambda}^k + h_{\lambda}^{k+1}>  + \frac{L_{\nabla\eta_{\Delta_n}}}{2}\left(\norm{\lambda^{k+1}-\lambda^k}^2 + \norm{-h_{\lambda}^{k+1} + h_{\lambda}^k}^2\right) \notag\\
        = &\alpha_k\<-h_{\lambda}^k, \lambda_+^k - \lambda^k> + \frac{\alpha_k}{\tau_{\lambda}}\norm{\lambda_+^k - \lambda^k}^2 + \theta_k\<\lambda_+^k - \lambda^k, s^{k+1} -h_{\lambda}^k - \mu_{\lambda}\left(\lambda^k - \frac{\bfone_n}{n}\right)> \notag \\
        & +  \frac{L_{\nabla\eta_{\Delta_n}}}{2}\left(\norm{\lambda^{k+1}-\lambda^k}^2 + \norm{h_{\lambda}^{k+1} - h_{\lambda}^k}^2\right) \notag \\
    \leq &- \frac{\theta_k}{\tau_{\lambda}}\norm{\lambda_+^k - \lambda^k}^2 + \theta_k\<s^{k+1} - \mu_{\lambda}\left(\lambda^k - \frac{\bfone_n}{n}\right), \lambda_+^k-\lambda^k> + \frac{L_{\nabla\eta_{\Delta_n}}}{2}\left(\norm{\lambda^{k+1}-\lambda^k}^2 + \norm{h_{\lambda}^{k+1} - h_{\lambda}^k}^2\right).\label{ineq: eta_lam_decrease_exp}
\end{align}
We also have
\begin{align}
        &\Phi_{\mu_{\lambda}}(x^k,\lambda^k) - \Phi_{\mu_{\lambda}}(x^{k+1}, \lambda^{k+1})\\
        =&\sum_{i=1}^{n}\left(\lambda_i^k\Phi_i(x^k) - \lambda_i^{k+1}\Phi_i(x^{k+1})\right) + \frac{\mu_{\lambda}}{2}\norm{\lambda^{k+1} - \frac{\bfone_n}{n}}^2 - \frac{\mu_{\lambda}}{2}\norm{\lambda^k - \frac{\bfone_n}{n}}^2\notag\\
        =&\<\lambda^k, \Phi^k> - \<\lambda^{k+1}, \Phi^{k+1}> + \frac{\mu_{\lambda}}{2}\left(\norm{\lambda^{k+1} - \lambda^k + \lambda^k - \frac{\bfone_n}{n}}^2 - \norm{\lambda^k - \frac{\bfone_n}{n}}^2\right)\notag\\
        =&\<\lambda^k - \lambda^{k+1}, \Phi^k> + \<\lambda^{k+1}, \Phi^k - \Phi^{k+1}>  + \mu_{\lambda}\alpha_k\<\lambda^k - \frac{\bfone_n}{n}, \lambda_+^k-\lambda^k> + \frac{\mu_{\lambda}}{2}\norm{\lambda^{k+1}-\lambda^k}^2\notag\\
        =&\alpha_k\<\lambda^k - \lambda_+^k, \E\left[s^{k+1}\middle|\setF_k\right] - \mu_{\lambda}\left(\lambda^k - \frac{\bfone_n}{n}\right)> + \alpha_k\<\lambda^k - \lambda_+^k, \Phi^{k} - \E\left[s^{k+1}\middle|\setF_k\right]> \notag\\
        & + \frac{\mu_{\lambda}}{2}\norm{\lambda^{k+1}-\lambda^k}^2 + \<\lambda^{k+1}, \Phi^k - \Phi^{k+1}> \notag\\
        \leq &\alpha_k\<\lambda^k - \lambda_+^k, \E\left[s^{k+1}\middle|\setF_k\right] - \mu_{\lambda}\left(\lambda^k - \frac{\bfone_n}{n}\right)> + \alpha_k\<\lambda^k - \lambda_+^k, \Phi^{k} - \E\left[s^{k+1}\middle|\setF_k\right]> \notag\\
        & + \frac{\mu_{\lambda}}{2}\norm{\lambda^{k+1}-\lambda^k}^2-\alpha_k\<\nabla_1\Phi(x^k, \lambda^k), x_+^k - x^k> + \sqrt{n}L_{\Phi}\norm{\lambda^{k+1} - \lambda^k}\norm{x_+^k - x^k} \notag\\
        & + \frac{L_{\nabla \Phi}}{2}\norm{x^{k+1} - x^k}^2.\label{eq: F_decrease}
\end{align}
where the inequality uses Lemma \ref{lem: smoothness} and $(c)$ in Assumption \ref{aspt: smoothness} to obtain
\begin{align*}
    &\<\lambda^{k+1}, \Phi^k - \Phi^{k+1}> \\
    =&\sum_{i=1}^{n}\lambda_i^{k+1}(\Phi_i(x^k) - \Phi_i(x^{k+1})) \\
        \leq & \sum_{i=1}^{n}\lambda_i^{k+1}\left(\<\nabla \Phi_i(x^k), x^k - x^{k+1}> + \frac{L_{\nabla \Phi}}{2}\norm{x^k - x^{k+1}}^2\right) \\
        = &-\alpha_k\<\nabla_1\Phi(x^k, \lambda^{k+1}), x_+^k - x^k> + \frac{L_{\nabla \Phi}}{2}\norm{x^{k+1} - x^k}^2 \\
        \leq & -\alpha_k\<\nabla_1\Phi(x^k, \lambda^k), x_+^k - x^k> + \sqrt{n}L_{\Phi}\norm{\lambda^{k+1} - \lambda^k}\norm{x_+^k - x^k} + \frac{L_{\nabla \Phi}}{2}\norm{x^{k+1} - x^k}^2.
\end{align*}
Taking conditional expectation with respect to $\setF_k$ on $\eqref{ineq: phi_decrease} + \eqref{ineq: eta_lam_decrease_exp} / c_3 + \eqref{eq: F_decrease}$,  we know
\begin{align}
        &\frac{\alpha_k}{\tau_{\lambda}}\norm{\lambda_+^k - \lambda^k}^2\notag\\
         \leq &\tilde W_{k,1}^{(2)} - \E\left[\tilde W_{k+1,1}^{(2)}\middle|\setF_k\right] + \alpha_k\<\nabla \Psi(x^k) - \nabla_1\Phi(x^k, \lambda^k), x_+^k - x^k>\notag \\
         & + \alpha_k\<\lambda^k - \lambda_+^k, \Phi^k - \E\left[s^{k+1}\middle|\setF_k\right]> + \frac{(L_{\nabla\Psi} + L_{\nabla \Phi})}{2}\norm{x^{k+1} - x^k}^2\notag \\
         & + \frac{(L_{\nabla\eta_{\Delta_n}} +c_3\mu_{\lambda})}{2c_3}\norm{\lambda^{k+1}-\lambda^k}^2 +  \sqrt{n}L_{\Phi}\norm{\lambda^{k+1} - \lambda^k}\norm{x_+^k - x^k} + \frac{L_{\nabla\eta_{\Delta_n}}}{2c_3}\E\left[\norm{h_{\lambda}^{k+1} - h_{\lambda}^k}^2\middle|\setF_k\right]\notag \\
         \leq & \tilde W_{k,1}^{(2)} - \E\left[\tilde W_{k+1,1}^{(2)}\middle|\setF_k\right] + \alpha_k\sqrt{n}L_{\Phi}\norm{\lambda^k-\lambda_*^k}\norm{x_+^k - x^k} \notag\\
         & + \alpha_kL_f\norm{\lambda_+^k - \lambda^k}\left(\sum_{i=1}^{n}\norm{y_i^k-y_{*,i}^k}^2\right)^{\frac{1}{2}} + \alpha_k\sqrt{n}L_{\Phi}\norm{\lambda_+^k - \lambda^k}\norm{x_+^k - x^k} \notag\\
         & + \frac{\alpha_k^2(L_{\nabla\Psi} + L_{\nabla \Phi})}{2}\norm{x_+^k - x^k}^2 + \frac{\alpha_k^2(L_{\nabla\eta_{\Delta_n}} +c_3\mu_{\lambda})}{2c_3}\norm{\lambda_+^k-\lambda^k}^2 + \frac{L_{\nabla\eta_{\Delta_n}}}{2c_3}\E\left[\norm{h_{\lambda}^{k+1} - h_{\lambda}^k}^2\middle|\setF_k\right] \notag\\
         \leq &\tilde W_{k,1}^{(2)} - \E\left[\tilde W_{k+1,1}^{(2)}\middle|\setF_k\right] + \alpha_k\left(\frac{1}{16\tau_{\lambda}}\norm{\lambda^k-\lambda_*^k}^2 + 4n\tau_{\lambda}L_{\Phi}^2\norm{x_+^k - x^k}^2 \right) \notag\\
         & + \alpha_k\left(\frac{1}{8\tau_{\lambda}}\norm{\lambda_+^k - \lambda^k}^2 + 2\tau_{\lambda}L_f^2\sum_{i=1}^{n}\norm{y_i^k-y_{*,i}^k}^2\right) + \alpha_k\left(\frac{1}{8\tau_{\lambda}}\norm{\lambda_+^k - \lambda^k}^2 + 2n\tau_{\lambda}L_{\Phi}^2\norm{x_+^k - x^k}^2\right) \notag\\
         & + \frac{\alpha_kn\tau_{\lambda}L_{\Phi}^2}{2}\norm{x_+^k - x^k}^2 + \frac{\alpha_k}{8\tau_{\lambda}}\norm{\lambda_+^k-\lambda^k}^2 
         + \frac{L_{\nabla\eta_{\Delta_n}}}{2c_3}\E\left[\norm{h_{\lambda}^{k+1} - h_{\lambda}^k}^2\middle|\setF_k\right],\label{ineq: lam_primal_long}
\end{align}
where the second inequality uses Lemma \ref{lem: smoothness}, and the third inequality uses Young's inequality and the conditions on $\alpha_k$ (see \eqref{ineq: conditions_primal}):
\[
    \frac{\alpha_k}{8\tau_{\lambda}} - \frac{\alpha_k^2(L_{\nabla\eta_{\Delta_n}} +c_3\mu_{\lambda})}{2c_3}\geq 0,\ \alpha_k^2(L_{\nabla\Psi} + L_{\nabla\Phi})\leq \alpha_kn\tau_{\lambda}L_{\Phi}^2.
\]
Recall that in Lemma \ref{lem: lambda_convergence} we have
\begin{equation}\label{ineq: lam_opt}
    \norm{\lambda^k-\lambda_*^k}^2\leq 2\norm{\lambda_+^k - \lambda^k}^2 + \frac{2}{\mu_{\lambda}^2}\norm{h_{\lambda}^k - \nabla_2\Phi_{\mu_{\lambda}}(x^k, \lambda^k)}^2,
\end{equation}
and by \eqref{ineq: h_var} we know
\begin{align}\label{ineq: hlam_simplify}
    \begin{aligned}
        \frac{L_{\nabla\eta_{\Delta_n}}}{c_3\tau_{\lambda}}\E\left[\norm{h_{\lambda}^{k+1} - h_{\lambda}^k}^2\right]&\leq 2c_3\alpha_k^2(\mu_{\lambda} + 1)^2\left(\E\left[\norm{h_{\lambda}^k - \nabla_2\Phi_{\mu_{\lambda}}(x^k,\lambda^k)}^2\right] + n\sigma_{f,0}^2\right) \\
        &\leq \frac{\alpha_k}{4}\E\left[\norm{h_{\lambda}^k - \nabla_2\Phi_{\mu_{\lambda}}(x^k,\lambda^k)}^2\right] + n\alpha_k^2\sigma_{f,0}^2.
    \end{aligned}
\end{align}
where the second inequality uses
\[
    2c_3(\mu_{\lambda} + 1)^2\leq \frac{1}{4},\ \alpha_k\leq 1
\]
in \eqref{ineq: conditions_primal}. Combining \eqref{ineq: lam_primal_long}, \eqref{ineq: lam_opt}, and \eqref{ineq: hlam_simplify}, we have
\begin{align}
    \begin{aligned}
        &\frac{\alpha_k}{\tau_{\lambda}^2}\E\left[\norm{\lambda_+^k - \lambda^k}^2\right] \\
        \leq &\frac{2}{\tau_{\lambda}}\E\left[\tilde W_{k,1}^{(2)} - \tilde W_{k+1,1}^{(2)}\right] + \frac{\alpha_k}{2}\E\left[\norm{h_{\lambda}^k - \nabla_2\Phi_{\mu_{\lambda}}(x^k,\lambda^k)}^2\right] \\
        & + 4\alpha_k L_f^2\E\left[\sum_{i=1}^{n}\norm{y_i^k-y_{*,i}^k}^2\right] + 13\alpha_kn L_{\Phi}^2\E\left[\norm{x_+^k - x^k}^2\right] + n\alpha_k^2\sigma_{f,0}^2,
    \end{aligned}
\end{align}
which implies the second inequality in \eqref{ineq: primal_decrease} by taking summation.
\end{proof}

\subsubsection{Dual Convergence}
\begin{lemma}\label{lem: x_dual}
    Suppose Assumptions \ref{aspt: smoothness}, \ref{aspt: stochastic_grad} hold for all $f_i, g_i$ and Assumption \ref{aspt: f_bound} holds. In Algorithm \ref{algo: momasoba} we have
    \begin{align}
        &\sum_{k=0}^{K}\alpha_k\E\left[\norm{h_x^k - \nabla_1\Phi_{\mu_{\lambda}}(x^k, \lambda^k)}^2\right] \notag\\
        \leq &\frac{1}{c_3}\E\left[\norm{h_x^0 - \nabla_1\Phi_{\mu_{\lambda}}(x^0, \lambda^0)}^2\right] + 3\sum_{k=0}^{K}\alpha_k\E\left[\norm{\E\left[w^{k+1}\middle|\setF_k\right] - \nabla_1\Phi_{\mu_{\lambda}}(x^k, \lambda^k)}^2\right] \notag\\
        & + \frac{3L_{\nabla \Phi}^2}{c_3^2}\sum_{k=0}^{K}\alpha_k\E\left[\norm{x_+^k - x^k}^2\right] +  \frac{3nL_{\Phi}^2}{c_3^2}\sum_{k=0}^{K}\alpha_k\E\left[\norm{\lambda_+^k - \lambda^k}^2\right] \notag\\
        & + 2c_3\sigma_{g,2}^2\sum_{k=0}^{K}\alpha_k^2\E\left[\sum_{i=1}^{n}\lambda_i^k\norm{z_i^k - z_{*,i}^k}^2\right] + c_3\sigma_w^2\sum_{k=0}^{K}\alpha_k^2, \notag\\
        &\sum_{k=0}^{K}\alpha_k\E\left[\norm{h_{\lambda}^k - \nabla_2\Phi_{\mu_{\lambda}}(x^k, \lambda^k)}^2\right] \notag\\
        \leq &\frac{1}{c_3}\E\left[\norm{h_{\lambda}^0 - \nabla_2\Phi_{\mu_{\lambda}}(x^0, \lambda^0)}^2\right] + 3\alpha_kL_f^2\sum_{i=1}^{n}\E\left[\norm{y_i^k - y_{*,i}^k}^2\right] \notag\\
        & + \frac{3nL_{\Phi}^2}{c_3^2}\sum_{k=0}^{K}\alpha_k\E\left[\norm{x_+^k - x^k}^2\right]  + \frac{3\mu_{\lambda}^2}{c_3^2}\sum_{k=0}^{K}\alpha_k\E\left[\norm{\lambda_+^k - \lambda^k}^2\right] + nc_3\sigma_{f,0}^2\sum_{k=0}^{K}\alpha_k^2.\label{ineq: dual_decrease}
    \end{align}
\end{lemma}
\begin{proof}
    The proof is similar to that of Lemma \ref{lem: x_dual_simple}, except that we now have another $\lambda^k$ to handle. Since $\nabla_1\Phi(x,\lambda) = \nabla_1\Phi_{\mu_{\lambda}}(x,\lambda)$ for all $(x,\lambda)$ (see \eqref{eq: minmaxbo_reform_reg}), for simplicity we omit the subscript $\mu_{\lambda}$ in $\nabla_1\Phi_{\mu_{\lambda}}(x,\lambda)$ in this proof. Note that by moving average update of $h_x^k$, we have
\begin{align}
    \begin{aligned}
        &h_x^{k+1} - \nabla_1\Phi(x^{k+1}, \lambda^{k+1}) \\
        = &(1 - \theta_k)h_x^k + \theta_k(w^{k+1} - \E\left[w^{k+1}\middle|\setF_k\right]) + \theta_k\E\left[w^{k+1}\middle|\setF_k\right] - \nabla_1\Phi(x^{k+1}, \lambda^{k+1})\\
        = &(1 - \theta_k)(h_x^k - \nabla_1\Phi(x^k, \lambda^k)) + \theta_k(\E\left[w^{k+1}\middle|\setF_k\right] - \nabla_1\Phi(x^k, \lambda^k)) \\
        & + \nabla_1\Phi(x^k, \lambda^k) - \nabla_1\Phi(x^{k+1}, \lambda^{k+1})  + \theta_k(w^{k+1} - \E\left[w^{k+1}\middle|\setF_k\right])
    \end{aligned}
\end{align}
Hence we know 
\begin{align}\label{ineq: h_x_decrease}
    \begin{aligned}
        &\E\left[\norm{h_x^{k+1} - \nabla_1\Phi(x^{k+1}, \lambda^{k+1})}^2\middle|\setF_k\right] \\
        = &\left\| (1 - \theta_k)(h_x^k - \nabla_1\Phi(x^k, \lambda^k)) + \theta_k(\E\left[w^{k+1}\middle|\setF_k\right] - \nabla_1\Phi(x^k, \lambda^k)) \right. \\
        & + \left. \nabla_1\Phi(x^k, \lambda^k) - \nabla_1\Phi(x^{k+1}, \lambda^{k+1})\right\|^2+ \theta_k^2\E\left[\norm{w^{k+1} - \E\left[w^{k+1}\middle|\setF_k\right]}^2\middle|\setF_k\right]\\ 
        \leq &(1 - \theta_k)\norm{h_x^k - \nabla_1\Phi(x^k, \lambda^k)}^2 +\theta_k^2\sigma_{w,k+1}^2 \\
        & + \theta_k\norm{(\E\left[w^{k+1}\middle|\setF_k\right] - \nabla_1\Phi(x^k, \lambda^k)) + \frac{1}{\theta_k}(\nabla_1\Phi(x^k, \lambda^k) - \nabla_1\Phi(x^{k+1}, \lambda^{k+1}))}^2 \\
        \leq &(1 - \theta_k)\norm{h_x^k - \nabla_1\Phi(x^k, \lambda^k)}^2 + 3\theta_k\norm{\E\left[w^{k+1}\middle|\setF_k\right] - \nabla_1\Phi(x^k, \lambda^k)}^2 +\theta_k^2\sigma_{w,k+1}^2 \\
        & + \frac{3}{\theta_k}\norm{\nabla_1\Phi(x^k, \lambda^k) - \nabla_1\Phi(x^{k+1}, \lambda^{k})}^2 + \frac{3}{\theta_k}\norm{\nabla_1\Phi(x^{k+1}, \lambda^k) - \nabla_1\Phi(x^{k+1}, \lambda^{k+1})}^2 \\
        \leq &(1 - \theta_k)\norm{h_x^k - \nabla_1\Phi(x^k, \lambda^k)}^2 + 3\theta_k\norm{\E\left[w^{k+1}\middle|\setF_k\right] - \nabla_1\Phi(x^k, \lambda^k)}^2 \\
        & + \frac{3\alpha_k^2}{\theta_k}\left(L_{\nabla \Phi}^2\norm{x_+^k - x^k}^2 + nL_{\Phi}^2\norm{\lambda_+^k - \lambda^k}^2\right) + \theta_k^2\sigma_{w,k+1}^2,
    \end{aligned}
\end{align}
where the first equality uses the fact that $x^k, \lambda^k, h_x^k, x^{k+1}, \lambda^{k+1}$ are all $\setF_k$-measurable and are independent of $w^{k+1}$ given $\setF_k$, the first inequality uses the convexity of $\norm{\cdot}^2$ and \eqref{ineq: w_var}, the second inequality uses Cauchy-Schwarz inequality, the third inequality uses the Lipschitz continuity of $\nabla_1\Phi$ in Lemma \ref{lem: smoothness}, and the update rules of $x^{k+1}$ and $\lambda^{k+1}$. Taking summation, expectation on both sides of \eqref{ineq: h_x_decrease}, dividing $c_3$, and applying \eqref{ineq: w_var}, we know the first inequality in \eqref{ineq: dual_decrease} holds.

Similarly we have
\begin{align}
    \begin{aligned}
        &h_{\lambda}^{k+1} - \nabla_2\Phi_{\mu_{\lambda}}(x^{k+1}, \lambda^{k+1}) \\
        = &(1 - \theta_k)h_{\lambda}^k + \theta_k\left(s^{k+1} - \mu_{\lambda}\lambda^k + \mu_{\lambda}\frac{\bfone_n}{n}\right) - \nabla_2\Phi_{\mu_{\lambda}}(x^{k+1}, \lambda^{k+1}) \\
        = &(1 - \theta_k)(h_{\lambda}^k - \nabla_2\Phi_{\mu_{\lambda}}(x^k, \lambda^k)) +  \theta\left(\E\left[s^{k+1}\middle|\setF_k\right] - \nabla_2\Phi(x^k,\lambda^k)\right) \\
        & + \nabla_2\Phi_{\mu_{\lambda}}(x^k, \lambda^k) - \nabla_2\Phi_{\mu_{\lambda}}(x^{k+1}, \lambda^{k+1}) +  \theta_k(s^{k+1} - \E\left[s^{k+1}\middle|\setF_k\right]).
    \end{aligned}
\end{align}
where the second equality uses $\nabla_2\Phi_{\mu_{\lambda}}(x^k, \lambda^k) = \nabla_2\Phi(x^k, \lambda^k) - \mu_{\lambda}\left(\lambda^k - \frac{\bfone_n}{n}\right)$. Hence we know
\begin{align}
        &\E\left[\norm{h_{\lambda}^{k+1} - \nabla_2\Phi_{\mu_{\lambda}}(x^{k+1}, \lambda^{k+1})}^2\middle|\setF_k\right] \notag\\
        = &\left\|(1 - \theta_k)(h_{\lambda}^k - \nabla_2\Phi_{\mu_{\lambda}}(x^k, \lambda^k)) + \theta\left(\E\left[s^{k+1}\middle|\setF_k\right] - \nabla_2\Phi(x^k,\lambda^k)\right) \right. \notag\\
        & + \left. \nabla_2\Phi_{\mu_{\lambda}}(x^k, \lambda^k) - \nabla_2\Phi_{\mu_{\lambda}}(x^{k+1}, \lambda^{k+1})\right\|^2 + \theta_k^2\E\left[\norm{s^{k+1} - \E\left[s^{k+1}\middle|\setF_k\right]}^2\middle|\setF_k\right]\notag\\ 
        \leq &(1 - \theta_k)\norm{h_{\lambda}^k - \nabla_2\Phi_{\mu_{\lambda}}(x^k, \lambda^k)}^2 + n\theta_k^2\sigma_{f,0}^2\notag\\
        & + \theta_k\norm{\E\left[s^{k+1}\middle|\setF_k\right] - \nabla_2\Phi(x^k,\lambda^k) + \frac{1}{\theta_k}(\nabla_2\Phi_{\mu_{\lambda}}(x^k, \lambda^k) - \nabla_2\Phi_{\mu_{\lambda}}(x^{k+1}, \lambda^{k+1}))}^2 \notag\\
        \leq &(1 - \theta_k)\norm{h_{\lambda}^k - \nabla_2\Phi_{\mu_{\lambda}}(x^k, \lambda^k)}^2 + 3\theta_k\norm{\E\left[s^{k+1}\middle|\setF_k\right] - \nabla_2\Phi(x^k,\lambda^k)}^2 + n\theta_k^2\sigma_{f,0}^2 \notag\\
        & + \frac{3}{\theta_k}\norm{\nabla_2\Phi_{\mu_{\lambda}}(x^k, \lambda^k) - \nabla_2\Phi_{\mu_{\lambda}}(x^{k+1}, \lambda^{k})}^2 + \frac{3}{\theta_k}\norm{\nabla_2\Phi_{\mu_{\lambda}}(x^{k+1}, \lambda^k) - \nabla_2\Phi_{\mu_{\lambda}}(x^{k+1}, \lambda^{k+1})}^2 \notag\\
        \leq &(1 - \theta_k)\norm{h_{\lambda}^k - \nabla_2\Phi_{\mu_{\lambda}}(x^k, \lambda^k)}^2 + 3\theta_kL_f^2\sum_{i=1}^{n}\norm{y_i^k - y_{*,i}^k}^2\notag \\
        & + \frac{3\alpha_k^2}{\theta_k}\left(nL_{\Phi}^2\norm{x_+^k - x^k}^2 + \mu_{\lambda}^2\norm{\lambda_+^k - \lambda^k}^2\right) + n\theta_k^2\sigma_{f,0}^2,\label{ineq: h_lambda_decrease}
\end{align}
where the third inequality uses Lemma \ref{lem: smoothness} and the fact that
\[
    \E\left[s^{k+1}\middle|\setF_k\right] = \left(f_1(x^k, y_1^k), ..., f_n(x^k, y_n^k)\right)\T,\ \nabla_2\Phi(x^k,\lambda^k) = \left(f_1(x^k, y_{*,1}^k), ..., f_n(x^k, y_{*,n}^k)\right)\T
\]
Taking summation, expectation on both sides of \eqref{ineq: h_lambda_decrease}, and dividing $c_3$, we know the second inequality in \eqref{ineq: dual_decrease} holds.
\end{proof}

\subsubsection{Proof of Theorem \ref{thm: momasoba_convergence}}
Now we are ready to present our main convergence results. Note that by Lemmas \eqref{lem: primal} and \eqref{lem: x_dual}, for $\tilde V_{k,1}$ we have
\begin{align}
        &\sum_{k=0}^{K}\alpha_k\E\left[\tilde V_{k,1}\right] = \sum_{k=0}^{K}\frac{\alpha_k}{\tau_x^2}\E\left[\norm{x_+^k-x^k}^2\right] + \sum_{k=0}^{K}\alpha_k\E\left[\norm{h_x^k - \nabla_1\Phi_{\mu_{\lambda}}(x^k, \lambda^k)}^2\right]\notag\\
        \leq & \frac{3L_{\nabla \Phi}^2}{c_3^2}\sum_{k=0}^{K}\alpha_k\E\left[\norm{x_+^k - x^k}^2\right] + \frac{1}{2}\sum_{k=0}^{K}\alpha_k\E\left[\norm{h_x^k - \nabla_1\Phi_{\mu_{\lambda}}(x^k,\lambda^k)}^2\right]\notag\\
        & + \frac{2}{\tau_x}\E\left[\tilde W_{0,1}^{(1)}\right] + \frac{1}{c_3}\E\left[\norm{h_x^0 - \nabla_1\Phi_{\mu_{\lambda}}(x^0, \lambda^0)}^2\right] + 2\sum_{k=0}^{K}\alpha_k\E\left[\norm{\E\left[w^{k+1}\middle|\setF_k\right] - \nabla\Psi(x^k)}^2\right]\notag \\
        & + 4\sum_{k=0}^{K}\alpha_k\E\left[\norm{\E\left[w^{k+1}\middle|\setF_k\right] -\nabla_1\Phi_{\mu_{\lambda}}(x^k,\lambda^k)}^2\right]  + \frac{3nL_{\Phi}^2}{c_3^2}\sum_{k=0}^{K}\alpha_k\E\left[\norm{\lambda_+^k - \lambda^k}^2\right]\notag \\
        & + \left(1+2c_3\right)\sigma_{g,2}^2\sum_{k=0}^{K}\alpha_k^2\E\left[\sum_{i=1}^{n}\lambda_i^k\norm{z_i^k - z_{*,i}^k}^2\right] + \left(1+c_3\right)\sigma_w^2\left(\sum_{k=0}^{K}\alpha_k^2\right).\label{ineq: tildeV_x}
\end{align}
By Lemma \ref{lem: hypergrad_error} we know
\begin{align}
    &4\sum_{k=0}^{K}\alpha_k\E\left[\norm{\E\left[w^{k+1}\middle|\setF_k\right] -\nabla_1\Phi_{\mu_{\lambda}}(x^k,\lambda^k)}^2\right] + 2\sum_{k=0}^{K}\alpha_k\E\left[\norm{\E\left[w^{k+1}\middle|\setF_k\right] - \nabla\Psi(x^k)}^2\right] \notag\\
    \leq & \sum_{k=0}^{K}\alpha_k\E\left[\sum_{i=1}^{n}20\left(\left(L_{\nabla f}^2 + L_{\nabla^2 g}^2\right)\norm{y_i^k - y_{*,i}^k}^2 + L_{\nabla g}^2\norm{z_i^k - z_{*,i}^k}^2\right)\right] \notag\\
    & + \sum_{k=0}^{K}16nL_{\Phi}^2\alpha_k\E\left[\norm{\lambda_+^k - \lambda^k}^2 + \frac{1}{\mu_{\lambda}^2}\norm{h_{\lambda}^k - \nabla_2\Phi_{\mu_{\lambda}}(x^k, \lambda^k)}^2\right]. \label{ineq: minmax_finalyz}
\end{align}
Choosing 
\begin{equation}\label{ineq: minmax_alpha_condition}
    \left(1+2c_3\right)\sigma_{g,2}^2\alpha_k\leq L_{\nabla g}^2
\end{equation}
in \eqref{ineq: tildeV_x}, and using \eqref{ineq: minmax_finalyz}, we know 
\begin{align}\label{ineq: tildeV_1_final}
    \begin{aligned}
        \sum_{k=0}^{K}\alpha_k\E\left[\tilde V_{k,1}\right]\leq &C_{v_1, x}\tau_x^2\sum_{k=0}^{K}\frac{\alpha_k}{\tau_x^2}\E\left[\norm{x_+^k - x^k}^2\right] + C_{v_1, h_x}\sum_{k=0}^{K}\alpha_k\E\left[\norm{h_x^k - \nabla_1\Phi_{\mu_{\lambda}}(x^k,\lambda^k)}^2\right] \\
        &+ C_{v_1, \lambda}\tau_{\lambda}^2\sum_{k=0}^{K}\frac{\alpha_k}{\tau_{\lambda}^2}\E\left[\norm{\lambda_+^k - \lambda^k}^2\right] +   C_{v_1, h_{\lambda}}\sum_{k=0}^{K}\alpha_k\E\left[\norm{h_{\lambda}^k - \nabla_2\Phi_{\mu_{\lambda}}(x^k, \lambda^k)}^2\right]  \\
        &+ C_{v_1, 0} + C_{v_1, 1}\left(\sum_{k=0}^{K}\alpha_k^2\right),
    \end{aligned}
\end{align}
where the constants are defined as
\begin{align}
    C_{v_1, x} = &20n\left(L_{\nabla f}^2 + L_{\nabla^2 g}^2\right)C_{yx} + 21nL_{\nabla g}^2C_{zx}  + \frac{3L_{\nabla \Phi}^2}{c_3^2},\ C_{v_1, h_x} = \frac{1}{2}, \notag \\
    C_{v_1, \lambda} = &\left(16+\frac{3}{c_3^2}\right)nL_{\Phi}^2,\ C_{v_1, h_{\lambda}} = \frac{16nL_{\Phi}^2}{\mu_{\lambda}^2}, \notag \\
    C_{v_1, 0} = &20\left(L_{\nabla f}^2 + L_{\nabla^2 g}^2\right)\left(\sum_{i=1}^{n}C_{y_i, 0}\right) + 21L_{\nabla g}^2\left(\sum_{i=1}^{n}C_{z_i, 0}\right)  + \frac{2}{\tau_x}\E\left[\tilde W_{0,1}^{(1)}\right] \notag \\
    &+ \frac{1}{c_3}\E\left[\norm{h_x^0 - \nabla_1\Phi_{\mu_{\lambda}}(x^0, \lambda^0)}^2\right], \notag \\
    C_{v_1, 1} = &20n\left(L_{\nabla f}^2 + L_{\nabla^2 g}^2\right)C_{y,1} + 21nL_{\nabla g}^2C_{z,1}. \notag 
\end{align}
For $\tilde V_{k,2}$ we have
\begin{align}\label{ineq: tildeV_lam}
    \begin{aligned}
        &\sum_{k=0}^{K}\alpha_k\E\left[\tilde V_{k,2}\right] = \sum_{k=0}^{K}\frac{\alpha_k}{\tau_{\lambda}^2}\E\left[\norm{\lambda_+^k - \lambda^k}^2\right] + \sum_{k=0}^{K}\alpha_k\E\left[\norm{h_{\lambda}^k - \nabla_2\Phi_{\mu_{\lambda}}(x^k, \lambda^k)}^2\right]\\
        \leq & \frac{3\mu_{\lambda}^2}{c_3^2}\sum_{k=0}^{K}\alpha_k\E\left[\norm{\lambda_+^k - \lambda^k}^2\right] + \frac{1}{2}\sum_{k=0}^{K}\alpha_k\E\left[\norm{h_{\lambda}^k - \nabla_2\Phi_{\mu_{\lambda}}(x^k, \lambda^k)}^2\right] \\
        &+ \frac{2}{\tau_{\lambda}}\E\left[\tilde W_{0,1}^{(2)}\right] + \frac{1}{c_3}\E\left[\norm{h_{\lambda}^0 - \nabla_2\Phi_{\mu_{\lambda}}(x^0, \lambda^0)}^2\right] + 7L_f^2\sum_{k=0}^{K}\alpha_k\E\left[\sum_{i=1}^{n}\norm{y_i^k-y_{*,i}^k}^2\right] \\
        &+ \left(13 + \frac{3}{c_3^2}\right)nL_{\Phi}^2\sum_{k=0}^{K}\alpha_k\E\left[\norm{x_+^k - x^k}^2\right] + n(1 + c_3)\sigma_{f,0}^2\left(\sum_{k=0}^{K}\alpha_k^2\right),
    \end{aligned}
\end{align}
which implies
\begin{align}\label{ineq: tildeV_2_final}
    \begin{aligned}
        \sum_{k=0}^{K}\alpha_k\E\left[\tilde V_{k,2}\right] \leq &C_{v_2, x}\tau_x^2\sum_{k=0}^{K}\frac{\alpha_k}{\tau_x^2}\E\left[\norm{x_+^k - x^k}^2\right] + C_{v_2, h_x}\sum_{k=0}^{K}\alpha_k\E\left[\norm{h_x^k - \nabla_1\Phi_{\mu_{\lambda}}(x^k,\lambda^k)}^2\right] \\
        &+ C_{v_2, \lambda}\tau_{\lambda}^2\sum_{k=0}^{K}\frac{\alpha_k}{\tau_{\lambda}^2}\E\left[\norm{\lambda_+^k - \lambda^k}^2\right] +   C_{v_2, h_{\lambda}}\sum_{k=0}^{K}\alpha_k\E\left[\norm{h_{\lambda}^k - \nabla_2\Phi_{\mu_{\lambda}}(x^k, \lambda^k)}^2\right]  \\
        &+ C_{v_2, 0} + C_{v_2, 1}\left(\sum_{k=0}^{K}\alpha_k^2\right)
    \end{aligned}
\end{align}
where the constants are defined as
\begin{align}
    \begin{aligned}
        C_{v_2, x} = &7nL_f^2C_{yx} + \left(13 + \frac{3}{c_3^2}\right)nL_{\Phi}^2,\ C_{v_2, h_x} = 0,\\
        C_{v_2, \lambda} = &\frac{3\mu_{\lambda}^2}{c_3^2},\ C_{v_2, h_{\lambda}} = \frac{1}{2},\\
        C_{v_2, 0} = &7L_f^2\left(\sum_{i=1}^{n}C_{y_i, 0}\right) + \frac{2}{\tau_{\lambda}}\E\left[\tilde W_{0,1}^{(2)}\right] + \frac{1}{c_3}\E\left[\norm{h_{\lambda}^0 - \nabla_2\Phi_{\mu_{\lambda}}(x^0, \lambda^0)}^2\right] \\
        C_{v_2, 1} = &7nL_f^2C_{y,1} + n(1+c_3)\sigma_{f,0}^2.
    \end{aligned}
\end{align}
According to the definition of the constants in Lemmas \ref{lem: yz_error} and \ref{lem: yz_error_minmax}, we could obtain (for simplicity we omit the dependency on $\kappa$ here)
\begin{align*}
    &C_{v_1, x} = \cO\left(\frac{n}{c_1^2} + \frac{n}{c_2^2} + \frac{1}{c_3^2}\right),\ C_{v_1, h_x} = \frac{1}{2} = \cO\left(1\right),\ C_{v_1,\lambda} = \cO\left(n + \frac{n}{c_3^2}\right),\ C_{v_1, h_{\lambda}} = \cO\left(\frac{n}{\mu_{\lambda}^2}\right),\\
    &C_{v_1, 0} = \cO\left(\frac{n}{c_1} + \frac{n}{c_2} + \frac{1}{c_3} + \frac{1}{\tau_x} + \frac{1}{c_3\tau_x}\right),\ C_{v_1, 1} = \cO\left(nc_1 + nc_2\right), \\
    &C_{v_2, x} = \cO\left(\frac{n}{c_1^2} + n + \frac{n}{c_3^2}\right),\ C_{v_2, h_x} = 0,\ C_{v_2,\lambda} = \cO\left(\frac{1}{c_3^2}\right),\ C_{v_2, h_{\lambda}} = \frac{1}{2} = \cO\left(1\right),\\
    &C_{v_2, 0} = \cO\left(\frac{n}{c_1} + \frac{1}{c_3}\right),\ C_{v_2, 1} = \cO\left(nc_1 + n + nc_3\right).
\end{align*}
Hence we can pick $\alpha_k, c_1, c_2, c_3, \tau_x, \tau_{\lambda}$ such that
\begin{align*}
    \alpha_k \equiv \Theta\left(\frac{1}{\sqrt{nK}}\right),\ c_1 = c_2 = \sqrt{n},\ c_3 = \Theta(1),\ \tau_x = \cO\left(\frac{\mu_{\lambda}}{n}\right),\ \tau_{\lambda} = \frac{1}{\mu_{\lambda}}
\end{align*}
which leads to
\begin{align*}
    C_{v_1, x}\tau_x^2 \leq \frac{1}{2},\ C_{v_2, x}C_{v_1,\lambda}\tau_x^2\tau_{\lambda}^2\leq \frac{1}{8},\ C_{v_2, \lambda}\tau_{\lambda}^2\leq \frac{1}{2},\ 
\end{align*}
and the conditions (\eqref{ineq: beta_gamma_condition}, \eqref{ineq: conditions_primal}, and \eqref{ineq: minmax_alpha_condition}) in previous lemmas hold. Moreover, using the above conditions in \eqref{ineq: tildeV_1_final} and \eqref{ineq: tildeV_2_final}, we can get
\begin{align*}
    &\sum_{k=0}^{K}\alpha_k\E\left[\tilde V_{k,1}\right]\leq \frac{1}{2}\sum_{k=0}^{K}\alpha_k\E\left[\tilde V_{k,1}\right] + C_{v_1,\lambda}\tau_{\lambda}^2\sum_{k=0}^{K}\alpha_k\E\left[\tilde V_{k,2}\right] + \cO\left(n\right)\\
    &\sum_{k=0}^{K}\alpha_k\E\left[\tilde V_{k,2}\right]\leq \frac{1}{2}\sum_{k=0}^{K}\alpha_k\E\left[\tilde V_{k,2}\right] + C_{v_2,x}\tau_x^2\sum_{k=0}^{K}\alpha_k\E\left[\tilde V_{k,1}\right] + \cO\left(\sqrt{n}\right).
\end{align*}
Combining the above two inequalities, we have
\[
    \frac{1}{K}\sum_{k=0}^{K}\E\left[\tilde V_{k,1}\right] = \cO\left(\frac{n^2}{\mu_{\lambda}^2\sqrt{K}}\right),\ \frac{1}{K}\sum_{k=0}^{K}\E\left[\tilde V_{k,2}\right] = \cO\left(\frac{n}{\sqrt{K}}\right),
\]
which completes the proof of Theorem \ref{thm: momasoba_convergence} since we have
\begin{align*}
    &\norm{\frac{1}{\tau_x}\left(x^k - \Pi_{\cX}\left(x^k - \tau_x\nabla\Psi_{\mu_{\lambda}}(x^k)\right) \right)}^2\\
    \leq &\frac{2}{\tau_x^2}\norm{x^k - \Pi_{\cX}\left(x^k - \tau_x\nabla_1\Phi_{\mu_{\lambda}}(x^k, \lambda^k)\right) }^2 + \frac{2}{\tau_x^2}\norm{\Pi_{\cX}\left(x^k - \tau_x\nabla\Psi_{\mu_{\lambda}}(x^k)\right) - \Pi_{\cX}\left(x^k - \tau_x\nabla_1\Phi_{\mu_{\lambda}}(x^k, \lambda^k)\right) }^2 \\
    \leq &\frac{2}{\tau_x^2}\norm{x^k - \Pi_{\cX}\left(x^k - \tau_x\nabla_1\Phi_{\mu_{\lambda}}(x^k, \lambda^k)\right) }^2 + 2nL_{\Phi}^2\norm{\lambda^k - \lambda_*^k}^2 \\
    \leq & 4\tilde V_{k,1} + \frac{4nL_{\Phi}^2}{\mu_{\lambda}^2}\tilde V_{k,2} = \cO\left(\frac{n^2}{\mu_{\lambda}^2\sqrt{K}}\right)
\end{align*}
where the second inequality uses non-expansiveness of projection operator and $\sqrt{n}L_{\Phi}$-Lipschitz continuity of $\nabla_1\Phi_{\mu_{\lambda}}(x,\cdot)$ in Lemma \ref{lem: smoothness}. Note that we have $n^2$ in the numerator since we explicitly write out the Lipschitz constant $L_{\nabla_1\Phi_{\mu_{\lambda}}}$.

\section{Discussions on the Prior Work \cite{gu2022min}}\label{sec: app_discussion}
In this section, we discuss several issues in the current form of \cite{gu2022min}, which introduces a \textbf{M}ulti-\textbf{O}bjective \textbf{R}obust \textbf{Bi}level \textbf{T}wo-timescale optimization algorithm (\texttt{MORBiT}). 



The primary issue in the current analysis of \texttt{MORBiT} arises from the \textit{ambiguity} and \textit{inconsistency} regarding the \textit{expectation and filtration}. As a consequence, the current form of the paper was unable to demonstrate $\E[\max_{i\in [n]}\|y_i^{k} - y_i^*(x^{(k-1)})\|^2 ]\leq \tilde{\cO}(\sqrt{n}K^{-2/5})$ claimed in Theorem 1 (10b) of \cite{gu2022min}. The subsequent arguments are incorrect. We discuss some mistakes made in \cite{gu2022min} as follows.

We start by looking at Lemma 8 (informal) and Lemma 14 (formal) in \cite{gu2022min} that characterize the upper bound of the $\cL^{(k+1)} - \cL^{(k)}$ where $\cL^{(k)} = \E[\sum_{i=1}^{n}\lambda_i^{(k)}\ell_i(x^{(k)})]$. Here, the function $\ell_i$ is the function $\Phi_i(x)$ in our notation. The paper incorrectly asserted that $$\cL^{k} = \sum_{i=1}^{n}\lambda_i^{(k)}\E[\ell_i(x^{(k)})].$$ 
To see why, let $\setF_k$ denote the sigma algebra generated by all iterates ($x, y, \lambda$) with superscripts not greater than $k$. It is important to note that both $\{\lambda_i^{(k)}\}$ and $x^{(k)}$ are random objectives given the filtration 
$\setF_{k}$. The ambiguity lies in the lack of clarity regarding the randomness over which the expectation operation is performed. In fact, we can rewrite the claim of Lemma 14 in \cite{gu2022min} without hiding the randomness. Let $\cL^{(k)} = \sum_{i=1}^{n} \lambda_i^{(k)} \ell_i(x^{(k)})$. Then, we have
\begin{align}
    \cL^{(k+1)} - \cL^{(k)} \label{eq: appendix-c problem 1} \leq & \cO(\alpha) \underbrace{\left(\sum_{i=1}^{n} \lambda_i^k \norm{y_i^{k+1} - y_i^*(x^{(k)})} \right)^2}_{\leq \max_{i\in[n]}\norm{y_i^{k+1} - y_i^*(x^{(k)})}^2} \\
    &- \frac{1}{\alpha} \norm{x^{k+1} - x^{k}}^2 + \cO(\gamma n) + \cO(\alpha) \norm{h_{x}^{(k)} - \E[h_{x}^{(k)}\mid \setF_k]}^2,\notag
\end{align}
where $\alpha,\beta, \gamma$ are step sizes for $x$, $y$, and $\lambda$ respectively. We hide the dependency for constants in their assumptions for simplicity. In addition, we want to emphasize that, unlike our notation, $h_x^{(k)}$ and $h_\lambda^{(k)}$ are stochastic gradients at step $k$. Therefore, $h_x^{(k)}$ and $h_\lambda^{(k)}$ are random objects given $\setF_k$. By taking expectations over all the randomness above, we can see that Lemma 14 in \cite{gu2022min} is incorrect because it writes in the form of $\max\E[\cdot]$ instead of $\E[\max(\cdot)]$. Therefore, the subsequent arguments regarding the convergence of $x,y,\lambda$ are incorrect, at least in the current form. 

Regardless of the error, one may be able to proceed with the proof by utilizing Eq.\eqref{eq: appendix-c problem 1} since our ultimate goal is to demonstrate the convergence of $\E[\max_{i\in [n]}\|y_i^{k} - y_i^*(x^{(k-1)})\|^2]$. One possible direction is to utilize the basic recursive inequality of $\max_{i\in [n]}\|y_i^{k+1} - y_i^*(x^{(k)}) \|^2$. Observe that for each $i\in[n]$, we can establish the following inequality similar to Lemma 13 in \cite{gu2022min} without hiding the randomness:
\begin{align}
    &\norm{y_i^{(k+1)} - y_i^*(x^{(k)})}^2 \leq \left(1-\cO(\mu_g\beta)\right) \norm{y_i^{(k)} - y_i^*(x^{(k-1)})}^2 + \cO\left(\frac{1}{\mu_g\beta}\right)\norm{x^{k} - x^{k-1}}^2 \\
    &+ \cO(\beta^2) \norm{h_{y,i}^{(k)} - \E[h_{y,i}^{(k)}|\setF_k]}^2 + \cO(\beta)\< y_i^{(k)} - y_i^*(x^{(k-1)}),h_{y,i}^{(k)} - \E[h_{y,i}^{(k)}|\setF_k] >\notag
\end{align}
However, the order of taking the expectation over the randomness and the maximum over $i\in[n]$ adds complexity to the problem. The last inner-product term can only be zero when first taking the conditional expectation with respect to $\setF_k$. When applying Young's inequality to bound this term, it inevitably introduces terms such as $\cO(\beta) \|h_{y,i}^{(k)} - \E[h_{y,i}^{(k)}|\setF_k]\|^2$ or $\cO(1) \|y_i^{(k)} - y_i^*(x^{(k-1)})\|^2$, which make it challenging to proceed further with the convergence analysis.

Finally, we remark about the choice of the stationarity condition used in~\cite{gu2022min}. Although the algorithmic aspect in~\cite{gu2022min} is motivated by~\cite{lin2020gradient}, the notion of stationarity for $\lambda$ in \cite{gu2022min} is different from \cite{lin2020gradient}. Under the notion of stationarity in \cite{lin2020gradient} (Definition 3.7) $\Phi_{1/2\ell}(\cdot)$ is the Moreau envelope of $\Phi(\cdot)$, which is defined after taking the max over $y$ (i.e., $\lambda$ in our notation) in Definition 3.5 in \cite{lin2020gradient}, and a point $x$ is $\epsilon$-stationarity when $\norm{\nabla \Phi_{1/2\ell}(x)}\leq \epsilon$. It is unclear if (10a) and (10c) in \cite{gu2022min} will imply similar convergence results under the notion of stationarity in Definition 3.7 in \cite{lin2020gradient}. 

\end{document}